\documentclass{lms}

\usepackage{enumerate}
\usepackage{calrsfs}
\usepackage[pdftex]{graphicx}
\usepackage{amsmath}
\usepackage{amssymb}

\newtheorem{thm}{Theorem}[section]
\newtheorem{prop}[thm]{Proposition}
\newtheorem{lem}[thm]{Lemma}
\newtheorem{cor}[thm]{Corollary}
\newtheorem{definition}[thm]{Definition}
\newtheorem{notation}[thm]{Notation}
\newnumbered{remark}[thm]{Remark}

\newnumbered{assertion}{Assertion}   
\newnumbered{conjecture}{Conjecture} 
\newnumbered{hypothesis}{Hypothesis}
\newnumbered{note}{Note}
\newnumbered{observation}{Observation}
\newnumbered{problem}{Problem}
\newnumbered{question}{Question}
\newnumbered{algorithm}{Algorithm}
\newnumbered{example}{Example}

\numberwithin{equation}{section}

\newcommand{\Apc}{\pazocal{A}} 
 
\newcommand{\Cpc}{\pazocal{C}} 
 
\newcommand{\Epc}{\pazocal{E}} 
\newcommand{\Fpc}{\pazocal{F}} 
\newcommand{\Gpc}{\pazocal{G}} 
 
\newcommand{\Ipc}{\pazocal{I}}

\newcommand{\Lpc}{\pazocal{L}} 
\newcommand{\Mpc}{\pazocal{M}}

\newcommand{\Ppc}{\pazocal{P}} 
\newcommand{\Qpc}{\pazocal{Q}} 
 
\newcommand{\Spc}{\pazocal{S}} 
\newcommand{\Tpc}{\pazocal{T}} 
\newcommand{\Upc}{\pazocal{U}}

\newcommand{\Pmc}{\mathcal{P}}

\newcommand{\Umc}{\mathcal{U}} 
\newcommand{\Vmc}{\mathcal{V}}

\newcommand{\Hbf}{\mathbf{H}}

\newcommand{\pbf}{\mathbf{p}}

\newcommand{\ybf}{\mathbf{y}}

\newcommand{\Hbbb}{\mathbb{H}}

\newcommand{\Pbbb}{\mathbb{P}} 
 
\newcommand{\Rbbb}{\mathbb{R}}

\newcommand{\Zbbb}{\mathbb{Z}}

\newcommand{\Amf}{\mathfrak{A}} 
\newcommand{\Bmf}{\mathfrak{B}} 
\newcommand{\Cmf}{\mathfrak{C}}

\newcommand{\Mmf}{\mathfrak{M}}

\newcommand{\Rmf}{\mathfrak{R}} 
 
\newcommand{\Tmf}{\mathfrak{T}}

\newcommand{\Xmf}{\mathfrak{X}} 
\newcommand{\Ymf}{\mathfrak{Y}} 
\newcommand{\Zmf}{\mathfrak{Z}}

\newcommand{\Mtd}{\widetilde{M}}
\newcommand{\Ntd}{\widetilde{N}}

\newcommand{\dtd}{\widetilde{d}}

\newcommand{\ltd}{\widetilde{l}}

\newcommand{\ptd}{\widetilde{p}}
\newcommand{\qtd}{\widetilde{q}}

\newcommand{\xtd}{\widetilde{x}}

\newcommand{\uora}{\overrightarrow{u}}
\newcommand{\vora}{\overrightarrow{v}}

\DeclareMathAlphabet{\pazocal}{OMS}{zplm}{m}{n}

\DeclareMathOperator{\suc}{suc}

\DeclareMathOperator{\Span}{Span}

\DeclareMathOperator{\Mod}{Mod}

\DeclareMathOperator{\id}{id}

\title[The degeneration of convex $\Rbbb\Pbbb^2$ structures.]%
{The degeneration of convex $\Rbbb\Pbbb^2$ structures on surfaces.}

\author{Tengren Zhang}

\classno{53A20}

\extraline{The author was partially supported by NSF grant DMS - 1006298.}

\begin{document}
\maketitle

\begin{abstract}
Let $M$ be a compact surface of negative Euler characteristic and let $\Cmf(M)$ be the deformation space of convex real projective structures on $M$. For every choice of pants decomposition for $M$, there is a well known parameterization of $\Cmf(M)$ known as the Goldman parameterization. In this paper, we study how some geometric properties of the real projective structure on $M$ degenerate as we deform it so that the internal parameters of the Goldman parameterization leave every compact set while the boundary invariants remain bounded away from zero and infinity.
\end{abstract}

\section{Introduction} 
Let $M$ be a closed orientable surface of genus $g\geq 2$, and consider the deformation space of convex $\Rbbb\Pbbb^2$ structures $\Cmf(M)$ on $M$. Topologically, we know that $\Cmf(M)$ is a cell of dimension $16g-16$ by Goldman \cite{Go1}. Moreover, in \cite{Go1}, Goldman proved that, just as in the case of hyperbolic structures on $M$, convex $\Rbbb\Pbbb^2$ structures on $M$ are given by specifying convex $\Rbbb\Pbbb^2$ structures on the pairs of pants in a pants decomposition of $M$ and then assembling them together. This allowed him to parameterize $\Cmf(M)$ in a way similar to the Fenchel-Nielsen parameterization of Teichm\"uller space $\Tmf(M)$. 

To define the Goldman parameterization of $\Cmf(M)$, we first need to choose a pants decomposition $\Pmc$ for $M$. There are then three kinds of parameters in the Goldman parameterization; the twist-bulge parameters, the boundary invariants and the internal parameters. The twist-bulge parameters describe how to glue pairs of pants together, and are the analog of the Fenchel-Nielsen twist coordinates. The boundary invariants contain the eigenvalue information of the holonomy about each simple closed curve in $\Pmc$, and correspond to the Fenchel-Nielsen length coordinates. However, unlike the hyperbolic case, specifying only the boundary invariants and the twist-bulge parameters is insufficient to nail down a point in $\Cmf(M)$. The internal parameters are thus designed to describe this ``residual deformation". For every pair of pants in the pants decomposition of $M$, we have two internal parameters, which take values in a $2$-cell $\Rmf$ in $\Rbbb^2$.

Choose a Goldman parameterization of $\Cmf(M)$, and consider a sequence in $\Cmf(M)$ so that all the boundary invariants vary within a compact set, but every pair of internal parameters eventually leave every compact set in $\Rmf$. We call such a sequence a \emph{Goldman sequence}. One should think of a Goldman sequence as a deformation of the convex $\Rbbb\Pbbb^2$ structure on $M$ ``away" from the image of the natural embedding $\Tmf(M)\subset\Cmf(M)$, also known as the \emph{Fuchsian locus}. The following is the main theorem of this paper.

\begin{thm*} 
Let $\Pmc$ be a pants decomposition on $M$ and choose a Goldman parameterization of $\Cmf(M)$ that is compatible with $\Pmc$. Let $\{\Mpc_j\}_{j=1}^\infty$ be a Goldman sequence in $\Cmf(M)$. Then the following hold:
\begin{enumerate}[(1)]
\item the length of the shortest (in the Hilbert metric) homotopically non-trivial closed curve in $\Mpc_j$ that is not homotopic to a multiple of a curve in $\Pmc$ grows arbitrarily large as $j$ approaches $\infty$.
\item the topological entropy of $\Mpc_j$ converges to $0$ as $j$ approaches $\infty$.
\end{enumerate}
\end{thm*}

Here, the \emph{Hilbert metric} is a canonical Finsler metric that one can define on any convex $\Rbbb\Pbbb^2$ surface, and the \emph{topological entropy} is a quantification of the ``complexity" of the geodesic flow on the unit tangent bundle of the convex $\Rbbb\Pbbb^2$ surface. These are discussed more carefully in Section \ref{thehilbertmetric} and Section \ref{entropygeodesicflow} respectively.  A slightly more general version of the above theorem (allowing for surfaces with boundary) is stated as Theorem \ref{mainthm1}.

This theorem has several interesting implications, which we list here. The general theme of these corollaries is to highlight the differences between $\Tmf(M)$ and $\Cmf(M)$.

There is a natural action of the mapping class group of $M$, denoted $\Mod(M)$, on any deformation space of $(X,G)$-structures on $M$ by pre-composition with the marking. In the case when this deformation space is $\Tmf(M)$, there are two well-known results. The first is known as \emph{Mumford compactness} \cite{Mu1}, which states that for any $\epsilon>0$, the quotient of
\[\Tmf(M)_\epsilon:=\{\Mpc\in\Tmf(M):\text{any closed curve in } \Mpc\text{ has length at least }\epsilon\}\]
by $\Mod(M)$ is compact. A corollary of (1) of the main theorem (via an easy diagonalization argument) is that there exists a sequence $\{\Mpc_j\}_{j=1}^\infty$ in $\Cmf(M)$ such that the length of the shortest closed curve in $\Mpc_j$ grows arbitrarily large. Since the length of the shortest closed curve is a continuous function on $\Cmf(M)$ that is invariant under the action of $\Mod(M)$, this implies that in contrast to the case of $\Tmf(M)$, Mumford compactness fails when the deformation space is $\Cmf(M)$. 

\begin{cor*}
For any $\epsilon>0$, the quotient of
\[\Cmf(M)_\epsilon:=\{\Mpc\in\Cmf(M):\text{any closed curve in }\Mpc\text{ has length at least }\epsilon\}\]
by $\Mod(M)$ is not compact. 
\end{cor*}
This corollary can also be obtained by applying Proposition 3.4 of Benoist-Hulin \cite{BeHu1} to Loftin's work in \cite{Lo2}.

The second is a result by Abikoff \cite{Ab1}, who proved that the action of $\Mod(M)$ on $\Tmf(M)$ extends continuously to the augmented Teichm\"uller space $\widehat{\Tmf(M)}$, and the quotient $\widehat{\Tmf(M)}/\Mod(M)$ is a compactification of $\Tmf(M)/\Mod(M)$. By considering an algebraic construction of $\widehat{\Tmf(M)}$, one can define a natural analog of this augmentation for $\Cmf(M)$, on which the natural action of $\Mod(M)$ on $\Cmf(M)$ extends continuously. (Canary-Storm \cite{CaSt1} constructed such an analog for the deformation space of Kleinian surface groups, which can also be done in the convex $\Rbbb\Pbbb^2$ structures setting.) A consequence of (1) of our theorem is that the quotient of this analogous augmentation by $\Mod(M)$ is not compact.

Also, Crampon \cite{Cr1} proved that the topological entropy of the convex $\Rbbb\Pbbb^2$ structures in $\Cmf(M)$ are bounded above by $1$, and that the value $1$ is achieved if and only if the structure lies in the Fuchsian locus. Later, Nie \cite{Ni1} showed that one can find a sequence in $\Cmf(M)$ so that the topological entropy converges to $0$. (In fact, their results hold for convex real projective structures on some manifolds of higher dimensions as well.) The next corollary of our theorem is then the natural next step in this progression of questions, and gives a positive answer to a question asked by Crampon and Marquis. (Questions 13 of \cite{Ques}.)

\begin{cor*}
For any $\alpha\in[0,1]$, there is a diverging sequence of convex projective structures so that the topological entropy of the structures along this sequence converges to $\alpha$.
\end{cor*}

Finally, our theorem also implies how several other geometric properties, such as the area of the convex $\Rbbb\Pbbb^2$ surface, degenerate along Goldman sequences.

\begin{cor*}
Let $\{\Mpc_j\}_{j=1}^\infty$ be a Goldman sequence in $\Cmf(M)$, then the area of $\Mpc_j$ grows to $\infty$ as $j$ approaches $\infty$.
\end{cor*}

Now, we will give a brief description of the strategy to prove the main theorem. Let $\Mpc$ be any convex $\Rbbb\Pbbb^2$ structure in $\Cmf(M)$, i.e. a marked convex $\Rbbb\Pbbb^2$ surface. The key behind the proof is the observation that we can write the Hilbert distance between certain pairs of points in $\Mpc$ in terms of the Goldman parameters for $\Mpc$. Moreover, the formulas for these distances are relatively simple, so we can understand how they change as we vary the Goldman parameters. With this in mind, we find a way to decompose each closed curve in $\Mpc$ into geodesic segments whose lengths we can bound from below by the distances between these special pairs of points. Fortuitously, the lower bound that we obtain this way goes to infinity as we deform the convex projective structure along Goldman sequences. This gives us (1). 

To obtain (2), we use the fact that the geodesic flow corresponding to the convex $\Rbbb\Pbbb^2$ structures on $M$ is Anosov, and that the non-wandering set for this flow is all of $T^1M$. The work of Bowen \cite{Bo1} then allows us to compute the topological entropy of the geodesic flow by the asymptotic exponential growth rate of orbits. It turns out that the lower bound in (1) is strong enough to yield an upper bound on the topological entropy, which vanishes as we deform along Goldman sequences.

The rest of this paper is divided into two sections. In Section \ref{preliminaries}, we give a brief introduction to the main objects of study in this paper and explain the tools used to prove the main theorem. In Section \ref{maincontent}, we will give the proof of the main theorem, as well as state and prove a couple of corollaries.

\begin{acknowledgements}
This work has benefitted from conversations with Gye-Seon Lee while he and the author were at the Center for Quantum Geometry of Moduli Spaces during August of 2013. The author also especially wishes to thank Richard Canary for introducing him to this subject, and the many fruitful discussions they have had in the course of writing this paper. Finally, the author is very grateful for the referee's careful reading of this paper and his/her many helpful comments. 
\end{acknowledgements}

\section{Preliminaries}\label{preliminaries}

\subsection{Convex $\Rbbb\Pbbb^2$ structures}\label{convexprojectivestructuresonsurfaces}

In this subsection, we will define the main objects of study in this paper, and introduce some of their well-known properties. Before we do so, we will make some comments on terminology and notation. For the rest of this paper, we will only consider orientable surfaces that admit a pants decomposition. Hence, we will use the word ``surface" to mean ``compact smooth orientable surface with negative Euler characteristic". Also, we will always use $M$ to denote a surface (without any projective structure) and $P$ to denote a smooth pair of pants. By an oriented curve in $M$, we will mean an equivalence class of continuous injective maps $\eta:[0,1]\to M$ (or $\eta:S^1\to M$ if $\eta$ is closed), where $\eta$ is equivalent to $\eta'$ if they have the same image and $\eta$ is homotopic to $\eta'$. A (unoriented) curve in $M$ is then an equivalence class of continuous injective maps $\eta:[0,1]\to M$ (or $\eta:S^1\to M$ if $\eta$ is closed), where $\eta$ is equivalent to $\eta'$ if they have the same image and $\eta$ is homotopic to either $\eta'$ or $\eta'$ with its parameterization reversed. We will denote both oriented and unoriented curves by a choice of a representative. However, we will also abuse notation by denoting the image of $\eta$ by $\eta$. This ambiguity is introduced so as to simplify notation; moreover, it should be clear from context which $\eta$ we are referring to. 

The deformation spaces that we will be considering in this paper are of surfaces modeled locally on $\Rbbb\Pbbb^2$, with an additional convexity condition. First, recall that the group of projective transformations on $\Rbbb\Pbbb^2$ is the group $PGL(3,\Rbbb):=GL(3,\Rbbb)/\Rbbb^*$, which can be identified with $SL(3,\Rbbb)$ by choosing the unique representative in each equivalence class in $PGL(3,\Rbbb)$ which has determinant $1$. 

\begin{definition}\label{projectivestructure}
Let $M$ be a surface (possible with boundary). An \emph{$\Rbbb\Pbbb^2$ structure} $\Mpc$ on $M$ is a maximal atlas of smooth charts $\{\psi_\alpha:U_\alpha\to\Rbbb\Pbbb^2\}$ such that 
\begin{enumerate}[(1)]
\item If $U_\alpha\cap U_\beta$ is nonempty, then the transition maps $\psi_\alpha\circ\psi_\beta^{-1}:\psi_\beta(U_\alpha\cap U_\beta)\to\psi_\alpha(U_\alpha\cap U_\beta)$ are restrictions of projective transformations on $\Rbbb\Pbbb^2$ to each component of $U_\alpha\cap U_\beta$. 
\item $\psi_\alpha$ maps each connected component of $\partial M\cap U_\alpha$ to a line segment in $\Rbbb\Pbbb^2$.
\end{enumerate}
The smooth surface $M$ equipped with an $\Rbbb\Pbbb^2$ structure is called an \emph{$\Rbbb\Pbbb^2$ surface}. Two $\Rbbb\Pbbb^2$ structures $\Mpc$ and $\Mpc'$ are \emph{isotopic} if there is a diffeomorphism $f:M\to M$ that is isotopic to the identity, so that $\Mpc'=\{\psi_\alpha\circ f:\psi_\alpha\in\Mpc\}$. \end{definition} 

Choose an $\Rbbb\Pbbb^2$ structure $\Mpc$ on $M$, fix any point $x\in M$ and choose $\xtd\in\Mtd$ so that $\Pi(\xtd)=x$, where $\Pi:\Mtd\to M$ is a universal covering. Given any initial germ $\psi$ of a chart in $\Mpc$ at the point $x$, we can construct a unique local diffeomorphism $d_\Mpc:\Mtd\to\Rbbb\Pbbb^2$ and a unique group homomorphism $h_\Mpc:\pi_1(M)\to SL(3,\Rbbb)$ such that the germ of $d_\Mpc$ at $\xtd$ is $\psi\circ\Pi$ and $d_\Mpc$ is $h_\Mpc$-equivariant. Here, $d_\Mpc$ is called the \emph{developing map}, $h_\Mpc$ is called the \emph{holonomy representation}, and the pair $(d_\Mpc,h_\Mpc)$ is known as the \emph{developing pair}. 

If we do not make a choice of the initial germ $\psi$, then the developing pair $(d_\Mpc,h_\Mpc)$ is only well-defined up to the action of $SL(3,\Rbbb)$, where $SL(3,\Rbbb)$ acts on the developing map by post composition and on the holonomy representation by conjugation. Furthermore, if $\Mpc$ and $\Mpc'$ are isotopic, then there is some $X\in SL(3,\Rbbb)$ so that $h_\Mpc=c_X\circ h_{\Mpc'}$, where $c_X:SL(3,\Rbbb)\to SL(3,\Rbbb)$ is conjugation by $X$, and $d_\Mpc$ is isotopic to $X\circ d_{\Mpc'}$ via $h_\Mpc$-equivariant diffeomorphisms. For more details, refer to Sections 2 and 3 of Goldman \cite{Go2} or Section 4 of Goldman \cite{Go3}. Since we will only be considering $\Rbbb\Pbbb^2$ structures up to isotopy, our developing pairs will only be well-defined up to $SL(3,\Rbbb)$ action and isotopy of the developing map. We will henceforth simplify terminology by referring to an isotopy class of $\Rbbb\Pbbb^2$ structures as an $\Rbbb\Pbbb^2$ structure.

Next, we want to define what it means for an $\Rbbb\Pbbb^2$ structure on a surface to be convex. Before we do so, we need to introduce a couple more definitions.

\begin{definition}
A subset $\Omega$ in $\Rbbb\Pbbb^2$ is \emph{properly convex} if 
\begin{enumerate}[(1)]
\item For any pair of distinct points $x,y$ in $\Omega$, there is a line segment between $x$ and $y$ that is entirely contained in $\Omega$.
\item There is a line $L$ in $\Rbbb\Pbbb^2$ that does not intersect the closure $\overline{\Omega}$ of $\Omega$.
\end{enumerate}
A properly convex subset $\Omega$ is \emph{strictly convex} if $\partial\Omega$ does not contain any line segments.
\end{definition}

\begin{definition}
We say that $X\in SL(3,\Rbbb)$ is \emph{positive hyperbolic} if $X$ is diagonalizable with positive, pairwise distinct eigenvalues. We will denote by $\Hbf\ybf\pbf^+$ the set of positive hyperbolic elements in $SL(3,\Rbbb)$. For any $X\in\Hbf\ybf\pbf^+$, we define its \emph{attracting fixed point} and \emph{repelling fixed point} to be the points in $\Rbbb\Pbbb^2$ corresponding to the eigenvectors in $\Rbbb^3$ of $X$ with the largest and smallest eigenvalues respectively. Any line segment in $\Rbbb\Pbbb^2$ connecting the attracting and repelling fixed points of $X$ is called an \emph{axis} of $X$.
\end{definition}

For any positive hyperbolic $X\in SL(3,\Rbbb)$, denote the smallest eigenvalue of $X$ by $\lambda_X$ and the sum of the other two eigenvalues by $\tau_X$. Then define
\[\Rmf:=\{(\lambda_X,\tau_X):X\in SL(3,\Rbbb) \text{ is positive hyperbolic }\}\]
and observe that this set can also be written as the $2$-cell 
\[\bigg\{(\lambda,\tau)\in\Rbbb^2:0<\lambda<1,\, \frac{2}{\sqrt{\lambda}}<\tau<\frac{1}{\lambda^2}+\lambda\bigg\}.\]
With these definitions, we can state the convexity condition mentioned above.

\begin{definition}
\begin{enumerate}[1]
\item An $\Rbbb\Pbbb^2$ structure $\Mpc$ is called \emph{convex} if $d_\Mpc$ is a diffeomorphism onto a properly convex subset of $\Rbbb\Pbbb^2$, and every non-identity element in $\pi_1(M)$ is mapped by $h_\Mpc$ into $\Hbf\ybf\pbf^+$. 
\item The \emph{deformation space of convex $\Rbbb\Pbbb^2$ structures} on $M$, denoted $\Cmf(M)$, is the set of (isotopy classes) of convex $\Rbbb\Pbbb^2$ structures on $M$. 
\item A \emph{convex $\Rbbb\Pbbb^2$ surface} is the quotient of a properly convex subset $\Omega$ by a group of projective transformations that acts freely, properly discontinuously and cocompactly on $\Omega$.
\end{enumerate}
\end{definition} 

Intuitively, one can think of $\Cmf(M)$ as the set of ``essentially different" ways to put a convex $\Rbbb\Pbbb^2$ structure on $M$, subject to some ``hyperbolicity" conditions on the holonomy about the boundary components of $M$. In the case when $M$ is a closed surface and $d_\Mpc$ is a diffeomorphism onto a properly convex subset of $\Rbbb\Pbbb^2$, then the image of $d_\Mpc$ is strictly convex with $C^1$ boundary (Kuiper \cite{Ku1} and Benzecri \cite{Bz1}; also see Theorem 1.1 of Benoist \cite{Be1} for a more general statement), and every non-identity element in $\pi_1(M)$ is automatically mapped by $h_\Mpc$ into $\Hbf\ybf\pbf^+$. (Theorem 3.2 of Goldman \cite{Go1}.) 

The Klein model of hyperbolic space allows one to realize every hyperbolic structure on a surface as a convex projective structure. This thus induces a natural embedding of Teichm\"uller space $\Tmf(M)$ into $\Cmf(M)$, whose image is called the \emph{Fuchsian locus}.

For any convex $\Rbbb\Pbbb^2$ structure $\Mpc$, we will denote the image of $d_\Mpc$ by $\Omega_\Mpc$. In this case, $h_\Mpc$ is also injective, and $d_\Mpc$ descends to a diffeomorphism $\widehat{d_\Mpc}:M\to\Omega_\Mpc/h_\Mpc(\pi_1(M))=:\Mmf$, which is also known as a \emph{marking}. Since $\Mmf$ is a convex $\Rbbb\Pbbb^2$ surface, we can think of a convex $\Rbbb\Pbbb^2$ structure $\Mpc$ on $M$ as a \emph{marked convex $\Rbbb\Pbbb^2$ surface}, i.e. $\Mpc=[f,\Mmf]$, where $\Mmf$ is a convex $\Rbbb\Pbbb^2$ surface and $f:M\to\Mmf$ is the marking (which is well-defined up to isotopy). We will switch freely between these two different ways of thinking about $\Mpc$.

Let $\Mmf$ be a convex $\Rbbb\Pbbb^2$ surface and $\Pi:\widetilde{\Mmf}\to \Mmf$ a covering map. A \emph{closed line} $\eta$ in $\Mmf$ is a closed curve that lifts to a straight line in $\widetilde{\Mmf}\subset\Rbbb\Pbbb^2$. It is easy to check that the following proposition holds, so we omit the proof.

\begin{prop}\label{properties}
Let $\Mpc=[f,\Mmf]$ be a convex $\Rbbb\Pbbb^2$ structure on $M$. Then:
\begin{enumerate}[(1)]
\item For any $X\in\pi_1(M)\setminus\{Id\}$, the attracting and repelling fixed points of $h_\Mpc(X)$ lie on $\partial\Omega_\Mpc$.
\item Every homotopically  non-trivial closed curve in $\Mmf$ is freely homotopic to a unique closed line in $\Mmf$.
\item Every path in $\Mmf$ is homotopic (relative to its end points) to a unique line segment in $\Mmf$.
\end{enumerate}
\end{prop}

\subsection{Cross ratio}\label{Cross ratio}

In this subsection, we will set up notation and establish some properties of a projective invariant which is commonly known as the cross ratio.

\begin{definition}\label{cross ratio def}
Let $L_1,L_2,L_3,L_4$ be four lines in $\Rbbb\Pbbb^2$ that intersect at a common point $o\in\Rbbb\Pbbb^2$, so that no three of the four $L_i$ agree. Choose vectors $v_o,v_1,\dots,v_4\in\Rbbb^3$ so that $o$ is the projectivization of $\Span_\Rbbb(v_o)$, and for all $i=1,\dots,4$, $L_i$ is the projectivization of $\Span_\Rbbb(v_o,v_i)$. Then define the \emph{cross ratio} 
\[(L_1,L_2,L_3,L_4):=\frac{v_o\wedge v_1\wedge v_3}{v_o\wedge v_1\wedge v_2}\cdot\frac{v_o\wedge v_4\wedge v_2}{v_o\wedge v_4\wedge v_3},\]
where the right hand side is evaluated as a number in the one point compactification $\Rbbb\cup\{\infty\}$ of $\Rbbb$ via a choice of linear identification between $\displaystyle \bigwedge^3_{i=1}\Rbbb^3$ and $\Rbbb$.
\end{definition}

One can easily verify that the cross ratio depends neither on the choice of $v_o, v_1,\dots,v_4$, nor on the linear identification between $\displaystyle \bigwedge^3_{i=1}\Rbbb^3$ and $\Rbbb$. Also, the condition that no three of the four $L_i$ agree ensures that if any of the two terms in the numerator on the right hand side evaluate to $0$, then both terms in the denominator do not evaluate to $0$, and vice versa. Hence, the cross ratio is well-defined. It is also clear from definition that 
\[(X\cdot L_1,X\cdot L_2,X\cdot L_3,X\cdot L_4)=(L_1,L_2,L_3,L_4)\] 
for all $X\in SL(3,\Rbbb)$, so the cross ratio is a projective invariant. If $a_i\neq o$ are points in $\Rbbb\Pbbb^2$ that lie in $L_i$, then we also use the notation
\[(a_1,a_2,a_3,a_4)_o:=(L_1,L_2,L_3,L_4)\]
when convenient. 

In the case where the four points $a_1,a_2,a_3,a_4$ lie on a single line, we have the following interpretation of the cross ratio.

\begin{prop}\label{crossratioline}
Let $a_1,a_2,a_3,a_4$ be collinear points in $\Rbbb\Pbbb^2$ so that no three of the four of them agree, and let $o\in\Rbbb\Pbbb^2$ be a point that is not collinear with $a_1,a_2,a_3,a_4$. Let $U$ be any affine chart of $\Rbbb\Pbbb^2$ containing $o$ and $a_i$ for $i=1,\dots,4$. Then 
\[(a_1,a_2,a_3,a_4)_o=\frac{|a_1-a_3||a_2-a_4|}{|a_1-a_2||a_3-a_4|},\] 
where $|p_1-p_2|$ is the Euclidean distance between $p_1,p_2\in U$ for any choice of Euclidean metric on $U$ that is compatible with the affine structure.
\end{prop}

The proof of the above proposition is another easy computation which we will omit. In particular, under the hypothesis of Proposition \ref{crossratioline}, the cross ratio $(a_1,a_2,a_3,a_4)_o$ is independent of $o$. In this case, we will use the notation $(a_1,a_2,a_3,a_4)$ in place of $(a_1,a_2,a_3,a_4)_o$. 
 
For our purposes, we will mainly be evaluating the cross ratio of points in the boundary of a properly convex domain in $\Rbbb\Pbbb^2$. Some properties of the cross ratio used in this setting are listed as the next proposition. These are simple consequences of Proposition \ref{crossratioline}, but will be very useful in the proof of our main theorem.

\begin{prop}\label{crossratioinequality}
Let $\Omega$ be a properly convex domain in $\Rbbb\Pbbb^2$, and let $o,a_1,a_2,a_3,a_4$ be pairwise distinct points that lie on $\partial \Omega$ in that order, so that the lines $L_i$ through $o$ and $a_i$ for $i=1,\dots,4$ are pairwise distinct. Also, let $I_{a_1,a_2}$ and $I_{a_3,a_4}$ be the closed subintervals of $\partial \Omega$ with endpoints $a_1$, $a_2$ and $a_3$, $a_4$ respectively, so that $o$ does not lie in $I_{a_1,a_2}$ and $I_{a_3,a_4}$. Let $x\in I_{a_1,a_2}$ and $y\in I_{a_3,a_4}$. (See Figure \ref{convexposition}.) Then the following hold:
\begin{enumerate}[(1)] 
\item $(a_1,a_2,a_3,a_4)_o\cdot (a_1,a_3,y,a_4)_o=(a_1,a_2,y,a_4)_o.$
\item $1<(a_1,a_2,a_3,a_4)_o<\infty$. 
\item $(a_1,a_2,a_3,a_4)_o\leq (a_1,x,a_3,a_4)_o$, and equality holds if and only if the line through $o$ and $x$ agrees with the line through $o$ and $a_2$.
\item $(a_1,a_2,a_3,a_4)_o\leq (a_1,a_2,y,a_4)_o$, and equality holds if and only if the line through $o$ and $y$ agrees with the line through $o$ and $a_3$.
\item $(a_1,a_2,a_3,a_4)_o\leq (x,a_2,a_3,a_4)_o$, and equality holds if and only if the line through $o$ and $x$ agrees with the line through $o$ and $a_1$.
\item $(a_1,a_2,a_3,a_4)_o\leq (a_1,a_2,a_3,y)_o$, and equality holds if and only if the line through $o$ and $y$ agrees with the line through $o$ and $a_4$.
\end{enumerate} 
\end{prop}
 
\begin{figure}
\includegraphics[scale=0.6]{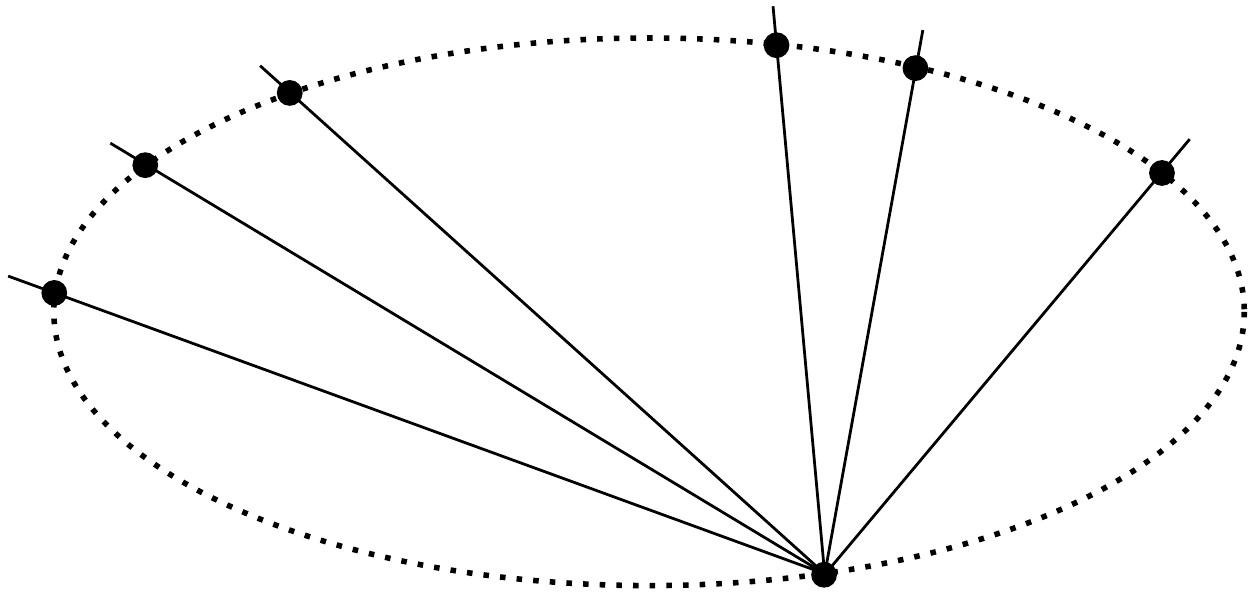}
\put (-352, 0){\makebox[0.7\textwidth][r]{$o$ }}
\put (-473, 68){\makebox[0.7\textwidth][r]{$x$ }}
\put (-336, 85){\makebox[0.7\textwidth][r]{$y$ }}
\put (-288, 72){\makebox[0.7\textwidth][r]{$a_4$ }}
\put (-369, 90){\makebox[0.7\textwidth][r]{$a_3$ }}
\put (-448, 80){\makebox[0.7\textwidth][r]{$a_2$ }}
\put (-492, 47){\makebox[0.7\textwidth][r]{$a_1$ }}
\caption{Cross ratio of points along the boundary of a properly convex domain in $\Rbbb\Pbbb^2$.}
\label{convexposition}
\end{figure}
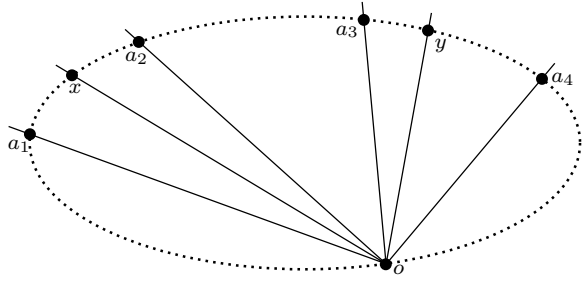

\subsection{The Hilbert metric}\label{thehilbertmetric} 

Using the cross ratio, we can define a metric, called the Hilbert metric, on the interior of any properly convex subset of $\Rbbb\Pbbb^2$. 

\begin{definition}
Let $\Omega\subset\Rbbb\Pbbb^2$ be a properly convex domain. For any two points $b$, $c$ in $\Omega$, let $L$ be the line in $\Rbbb\Pbbb^2$ through $b$ and $c$, and let $a$ and $d$ be the two points where $L$ intersects $\partial \Omega$, such that $a,b,c,d$ lie on $L$ in that order. The \emph{Hilbert distance} between $b$ and $c$ is 
\[Hd_\Omega(b,c):=\frac{1}{2}\log(a,b,c,d).\]
For any rectifiable path $\gamma$ in $\Omega$, we will denote the length of $\gamma$ by $l_\Omega(\gamma)$.
\end{definition}

Since the Hilbert metric is defined using the cross ratio, it is invariant under projective transformations that preserve $\Omega$. Moreover, the Hilbert metric is a Finsler metric, i.e. it is given by a norm $||\cdot||_x$ on the tangent space at every $x\in \Omega$, which varies smoothly with $x$. 

To obtain an explicit formula for this norm, choose an affine chart $U$ of $\Rbbb\Pbbb^2$ that contains $\overline \Omega$, and equip $U$ with an Euclidean metric. This induces a norm $|\cdot|_x$ on the tangent space of every $x\in \Omega$. For any tangent vector $v$ at a point $x\in \Omega$, let $\gamma$ be the oriented line through $x$ so that $v$ is tangential to $\gamma$ at $x$, and let $x^+$ and $x^-$ be the two points where $\gamma$ intersects $\partial \Omega$. Then define
\[||v||_x:=\frac{|v|_x}{2}(\frac{1}{|x-x^+|}+\frac{1}{|x-x^-|}).\]
where $|x-x^+|$, $|x-x^-|$ are the Euclidean distances between $x$ and $x^+$, $x$ and $x^-$ respectively. One can verify that $||v||_x$ depends neither on the choice of the affine chart $U$ nor on the choice of Euclidean metric on $U$, and that this norm gives rise to the Hilbert metric on $\Omega$.

The next proposition gives several properties of the Hilbert metric. We will omit the proof as they follow from the properties of the cross ratio discussed in Section \ref{Cross ratio}. 

\begin{prop}\label{hilbertmetricprop}
Let $\Omega',\Omega$ be open properly convex domains in $\Rbbb\Pbbb^2$ such that $\Omega'\subset \Omega$ and $\Omega$ is strictly convex. Let $b,c$ be a pair of distinct points in $\Omega'$, let $L$ be the line in $\Rbbb\Pbbb^2$ through $b$ and $c$, and let $\gamma$ be the line segment in $\Omega'$ between $b$ and $c$. The following hold:
\begin{enumerate}[(1)]
\item The line segment $\gamma$ is rectifiable, and $l_{\Omega'}(\gamma)=Hd_{\Omega'}(b,c)$.
\item $Hd_{\Omega'}(b,c)\geq Hd_\Omega(b,c)$, and equality holds if and only if $L\cap\partial \Omega'=L\cap\partial \Omega$.
\item If $\eta$ is a rectifiable path in $\Omega$ between $b$ and $c$, then $l_\Omega(\eta)\geq l_\Omega(\gamma)$ and equality holds if and only if $\eta=\gamma$.
\item Let $L_1$ and $L_2$ be line segments in $\Omega$ with endpoints in $\partial \Omega$. If $\partial \Omega$ has regularity $C^1$, then either $Hd_\Omega(L_1,L_2)=0$ or there is a unique pair of points $p_1\in L_1$ and $p_2\in L_2$ so that $Hd_\Omega(L_1,L_2)=Hd_\Omega(p_1,p_2)$.
\end{enumerate}
\end{prop}

For any convex $\Rbbb\Pbbb^2$ surface $\Mmf=\Omega/\Gamma$, we can define the Hilbert metric $Hd_{\Omega}$ on $\Omega$. The $SL(3,\Rbbb)$ invariance of the cross ratio implies that this descends to a metric on $\Mmf$, also called the Hilbert metric, which we denote by $Hd_\Mmf$. If $\gamma$ is a rectifiable path in $\Mmf$, we also denote the length of $\gamma$ by $l_\Mmf(\gamma)$. In the case when $\Mpc$ is a marked convex $\Rbbb\Pbbb^2$ surface, we also denote the Hilbert metric on $\Mpc$ by $Hd_\Mpc$ and the length of $\gamma$ by $l_\Mpc(\gamma)$. When $M$ is a closed surface, the Hilbert metrics for convex $\Rbbb\Pbbb^2$ structures in the Fuchsian locus $\Tmf(M)\subset\Cmf(M)$ agree with the corresponding hyperbolic metrics.


For the rest of this paper, let $M_{g,n}$ denote a closed genus $g$ surface with $n$ open discs removed such that $2g-2+n>0$. Let $i:M_{g,n}\to M_{g',n'}$ be a smooth embedding that is $\pi_1$-injective. For any $\Mpc_{g',n'}\in\Cmf(M_{g',n'})$, let $i':M_{g,n}\to M_{g',n'}$ be a diffeomorphism homotopic to $i$ so that it maps the boundary components of $M_{g,n}$ to closed lines in $M_{g',n'}$ equipped with the convex $\Rbbb\Pbbb^2$ structure $\Mpc_{g',n'}$. Then let $\Mpc_{g,n}$ be the convex $\Rbbb\Pbbb^2$ structure on $M_{g,n}$ obtained by precomposing the charts in $\Mpc_{g',n'}$ by $i'$. In fact, as a consequence of Section 5 of Goldman \cite{Go1}, we know that every convex $\Rbbb\Pbbb^2$ structure in $\Cmf(M_{g,n})$ can be obtained this way. Thus, $i$ induces a surjection $i^\#:\Cmf(M_{g',n'})\to\Cmf(M_{g,n})$. 


If $(g,n)\neq (g',n')$, then $\partial\Omega_{\Mpc_{g,n}}\neq\partial\Omega_{\Mpc_{g',n'}}$, so we know that $Hd_{\Mpc_{g,n}}\neq Hd_{\Mpc_{g',n'}}|_{M_{g,n}}$ by (2) of Proposition \ref{hilbertmetricprop}. On the other hand, we have the following proposition, which is an easy consequence of Proposition \ref{hilbertmetricprop}.

\begin{prop}\label{uniquegeodesic}
Let $i:M_{g,n}\to M_{g',0}$ be any $\pi_1$-injective embedding, let $\Mpc_{g',0}=[f',\Mmf_{g',0}]\in\Cmf(M_{g',0})$ and let $\Mpc_{g,n}=[f,\Mmf_{g,n}]:=i^\#(\Mpc_{g',0})\in\Cmf(M_{g,n})$. Also, let $\eta$ be a closed line in $\Mmf_{g,n}\subset \Mmf_{g',0}$.
\begin{enumerate}[(1)]
\item $l_{\Mmf_{g,n}}(\eta)=l_{\Mmf_{g',0}}(\eta)$.
\item Let $\gamma$ be any rectifiable closed curve homotopic to $\eta$. Then $l_{\Mmf_{g,n}}(\eta)\leq l_{\Mmf_{g,n}}(\gamma)$, with equality if and only if $\eta=\gamma$. 
\end{enumerate}
\end{prop}

Part (2) of Proposition \ref{uniquegeodesic} tells us that even though a compact convex $\Rbbb\Pbbb^2$ surface $\Mmf$ equipped with $Hd_\Mmf$ is not a unique geodesic space, the closed curves in $\Mmf$ have unique length minimizing representatives in their free homotopy classes, namely the closed lines. Thus, from now on, we will refer to the closed lines as closed geodesics.

Since we have a Hilbert metric on any convex $\Rbbb\Pbbb^2$ surface $\Mmf$, we can define a canonical (up to scaling) measure on $\Mmf$.

\begin{definition}
Let $\Mmf$ be a convex $\Rbbb\Pbbb^2$ surface. The \emph{Busemann area} $\nu_\Mmf$ is the 2-dimensional Hausdorff measure of the Hilbert metric $Hd_\Mmf$, rescaled so that in the case when $\Mmf$ is a hyperbolic surface, $\nu_\Mmf$ agrees with the hyperbolic area on $\Mmf$. 
\end{definition}

For background on the Busemann area, one can refer to Chapter 3, Part 1 of Bao \cite{Ba1}. The Busemann area is a Borel measure, and so gives us a notion of area for measurable subsets of $\Mmf$. It lifts to a measure $\nu_{\widetilde{\Mmf}}$ on $\widetilde{\Mmf}$ that is $\pi_1(\Mmf)$-invariant. If we choose an affine chart $U$ of $\Rbbb\Pbbb^2$ containing $\widetilde{\Mmf}$ and choose an Euclidean metric on $U$, then we have the usual Lebesgue measure $\mu$ on $U$ and hence on $\widetilde{\Mmf}$. Busemann showed that the Busemann area on $\widetilde{\Mmf}$ is absolutely continuous with respect to the Lebesgue measure. In fact, the Radon-Nikodym derivative of the Busemann area with respect to the Lebesgue measure at some $x\in\Mmf$ is $\displaystyle\frac{C}{\mu(B_1(\xtd))}$, where $B_1(\xtd)$ is the unit ball in the Hilbert metric centered at a lift $\xtd\in\widetilde{\Mmf}$ of $x$ and $C$ is some constant. For more details, see Busemann \cite{Bu1}.

\subsection{The Goldman parameters}\label{goldmanparameters}

In his paper \cite{Go1}, Goldman gave an explicit parameterization of $\Cmf(M_{g,n})$. Roughly, he did this by first parameterizing the deformation space of convex $\Rbbb\Pbbb^2$ structures on a pair of pants, and then extending this parameterization to all compact surfaces by specifying how to assemble the pairs of pants together. In this subsection, we will explain how to obtain this parameterization for a pair of pants, and briefly describe how to extend this parameterization to compact surfaces. 

On a smooth pair of pants $P$, choose $A_0,B_0,C_0$ in $\pi_1(P)$ corresponding to the three boundary components of $P$, such that $C_0B_0A_0=I$. This choice induces a lamination on every $\Ppc\in\Cmf(P)$ in the following way. Let $a_0$, $b_0$, $c_0$ be the repelling fixed points of $A_0$, $B_0$, $C_0$ in the Gromov boundary $\partial_\infty\pi_1(P)$ of $\pi_1(P)$ respectively. For any $\Ppc\in\Cmf(P)$, $d_\Ppc$ induces a $h_\Ppc$-equivariant injection $\xi:\partial_\infty\pi_1(P)\to\partial\Omega_\Ppc$. This then gives a lamination of $\Omega_\Ppc$ by the $\pi_1(P)$ orbits of the three lines in $\Omega_\Ppc$ between $\xi(a_0)$ and $\xi(b_0)$, $\xi(b_0)$ and $\xi(c_0)$, $\xi(c_0)$ and $\xi(a_0)$. Since the orbits of these three lines are disjoint, the lamination descends to a lamination on $\Ppc$ with three leaves. 

\begin{definition} \label{idealtriangulation}
The lamination on $\Ppc$ constructed above is called the \emph{ideal triangulation} of $\Ppc$ corresponding to $A_0$, $B_0$, $C_0$.
\end{definition}

Goldman's parameterization of $\Cmf(P)$ is given by the following theorem.

\begin{thm}\label{Goldman}
The deformation space $\Cmf(P)$ is an open $8$-dimensional cell. Furthermore, the map
\[\Theta:\Cmf(P)\to\Rmf^3\]
obtained by associating to a convex $\Rbbb\Pbbb^2$ structure the boundary invariants 
\[((\lambda_A,\tau_A),(\lambda_B,\tau_B),(\lambda_C,\tau_C))\]
is a fibration over an open $6$-cell with fiber a $2$-dimensional open cell.
\end{thm} 

We will give a summary of parts of Goldman's proof as some of these will be used later. There are two main steps in the proof. In the first step, one argues that specifying a marked convex $\Rbbb\Pbbb^2$ structure on a pair of pants is equivalent to specifying the following data (see Figure \ref{goldman}), up to equivalence under the action of $SL(3,\Rbbb)$:
\begin{enumerate}[(1)]
\item Four closed triangles $\Delta_0$, $\Delta_1$, $\Delta_2$, $\Delta_3$ in $\Rbbb\Pbbb^2$ such that 
\begin{enumerate}[(1)]
\item $\Delta_0$ and $\Delta_i$ intersect exactly along an edge for $i=1,2,3$,
\item For any $i,j\in\{1,2,3\}$ such that $i\neq j$, $\Delta_i$ and $\Delta_j$ intersect exactly at a point,
\item $\Delta_0\cup\Delta_1\cup\Delta_2\cup\Delta_3$ is a properly convex hexagon.
\end{enumerate}
\item Three elements $A,B,C\in \Hbf\ybf\pbf^+$ such that 
\begin{enumerate}[(1)]
\item $A$ has $\Delta_2\cap\Delta_3$ as its repelling fixed point and $A\cdot\Delta_2=\Delta_3$,  
\item $B$ has $\Delta_3\cap\Delta_1$ as its repelling fixed point and $B\cdot\Delta_3=\Delta_1$, 
\item $C$ has $\Delta_1\cap\Delta_2$ as its repelling fixed point and $C\cdot\Delta_1=\Delta_2$, 
\item CBA=I.
\end{enumerate}
\end{enumerate}
More specifically, we have a bijection 
\[\Xi:\Cmf(P)\to\Qpc:=\{(\Delta_0,\Delta_1,\Delta_2,\Delta_3, A, B, C): (1), (2)\text{ hold}\}/SL(3,\Rbbb)\] 
where $SL(3,\Rbbb)$ acts on $\{(\Delta_0,\Delta_1,\Delta_2,\Delta_3, A, B, C): (1), (2)\text{ hold}\}$ coordinate wise, by the usual left action on the $\Delta_i$'s and conjugation on $A,B,C$. In fact, we can explicitly describe the map $\Xi$; for any $\Ppc\in\Cmf(P)$, we have $\Xi(\Ppc)=[\Delta_0,\Delta_1,\Delta_2,\Delta_3, A, B, C]$, where 
\begin{itemize}
\item $A$, $B$, $C$, are the images of $A_0$, $B_0$, $C_0$ respectively under $h_\Ppc$,
\item $\Delta_0$ is the unique triangle whose vertices are the repelling fixed points of $A$, $B$, $C$, which we denote by $a, b, c$ respectively, and whose interior lies in $\Omega_\Ppc$,
\item $\Delta_1$ is the unique triangle whose vertices are $b$, $c$, $d:=B\cdot a$ and whose interior lies in $\Omega_\Ppc$, 
\item $\Delta_2$ is the unique triangle whose vertices are $c$, $a$, $e:=C\cdot b$ and whose interior lies in $\Omega_\Ppc$, 
\item $\Delta_3$ is the unique triangle whose vertices are $a$, $b$, $f:=A\cdot c$ and whose interior lies in $\Omega_\Ppc$.
\end{itemize}
It is easy to see that $\Delta_0\cup\Delta_1$ is in fact a fundamental domain of the action of $\pi_1(P)$ on $\Omega_\Ppc$. Moreover, the $\pi_1(P)$ orbits of the edges of $\Delta_0$ give a lamination on $\Omega_\Ppc$, which descends to the ideal triangulation of $\Ppc$ corresponding to $A_0$, $B_0$ and $C_0$.

The first step thus reduces the problem to parameterizing $\Qpc$, which is the second step of the proof. One can parameterize $\Qpc$ by $\Rmf^3\times(\Rbbb^+)^2$ ($\Rmf$ was defined in Section \ref{convexprojectivestructuresonsurfaces}) so that the map $\Theta$ in the statement of Theorem \ref{Goldman}, when described in this parameterization, is just projection to the first six parameters. The formal proof that $\Rmf^3\times(\Rbbb^+)^2$ actually parameterizes $\Qpc$ involves solving a system of equations that one obtains from the data of the configuration of the $\Delta_i$'s along with their interaction with $A$, $B$, $C$. Rather than do that, we will simply describe a geometric way to interpret the eight parameters, and refer the reader to Section 4 of Goldman \cite{Go1} for the proof.

\begin{figure}
\includegraphics[scale=0.4]{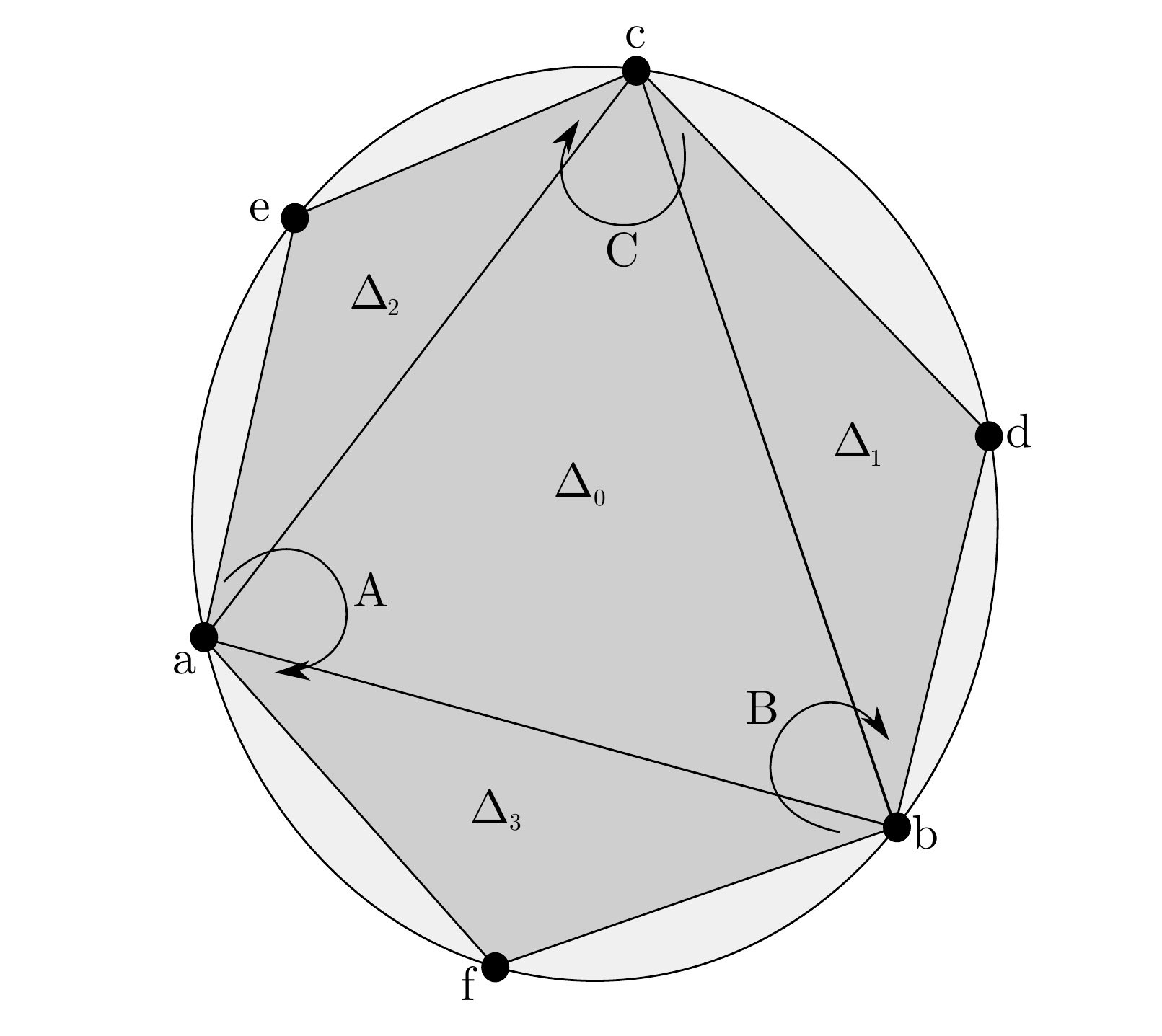}
\caption{$H_\Ppc$ (see Notation \ref{hexagonnotation}) contained in $\Omega_\Ppc$.}
\label{goldman}
\end{figure}

Any point $\Xi(\Ppc)=[\Delta_0,\Delta_1,\Delta_2,\Delta_3, A, B, C]$ in $\Qpc$ is parameterized by the parameters $((\lambda_A,\tau_A),(\lambda_B,\tau_B),(\lambda_C,\tau_C),s,t)\in\Rmf^3\times(\Rbbb^+)^2$. Here, for all $X=A,B,C$, $\lambda_X$ is the smallest eigenvalue of $X$ and $\tau_X$ is the sum of the other two eigenvalues of $X$. The first six parameters thus determine the eigenvalue data of the holonomy about each boundary component of the marked projective pair of pants $\Ppc$. In particular, the lengths (in the Hilbert metric $Hd_\Ppc$) of the boundary components of $\Ppc$ can be obtained explicitly from these six parameters. Indeed, if $\alpha$ is the boundary component of $\Ppc$ corresponding to $X=A,B,C$, one can easily compute that the three eigenvalues for $X$ are 
\[\lambda_X, \frac{\tau_X-\sqrt{\tau_X^2-\frac{4}{\lambda_X}}}{2}, \frac{\tau_X+\sqrt{\tau_X^2-\frac{4}{\lambda_X}}}{2}\]
listed in increasing order, so
\begin{equation}\label{lengthofc}
l_\Ppc(\alpha)=\log\Bigg(\frac{\tau_X+\sqrt{\tau_X^2-\frac{4}{\lambda_X}}}{2\lambda_X}\Bigg).
\end{equation}

Describing the geometrical significance of the last two parameters $s$ and $t$ is less straightforward. Before we do that, we shall introduce some notation that will be used in the rest of the paper.

\begin{notation} \label{hexagonnotation}
For any $\Xi(\Ppc)=[\Delta_0,\Delta_1,\Delta_2,\Delta_3, A, B, C]\in\Qpc$, let $H_\Ppc$ be the hexagon $\displaystyle\bigcup_{i=0}^3\Delta_i$ (this is well-defined up to translation by a projective transformation). Denote its vertices by $a,f,b,d,c,e$ in that order, where $a,b,c$ are the repelling fixed points of $A$, $B$, $C$ respectively, and $d:=B\cdot a$, $e:=C\cdot b$, $f:=A\cdot c$. 
\end{notation}

Goldman computed at the end of Section 4 of \cite{Go1} the following cross ratios in terms of the parameters of his parameterization:
\begin{equation}\label{importantcrossratio}
\begin{array}{r}
\displaystyle(e,c,b,f)_a=1+\sqrt{\frac{\lambda_C\lambda_A}{\lambda_B}}\tau_As+\frac{\lambda_C}{\lambda_B}s^2,\\
\displaystyle(f,a,c,d)_b=1+\sqrt{\frac{\lambda_A\lambda_B}{\lambda_C}}\tau_Bs+\frac{\lambda_A}{\lambda_C}s^2,\\
\displaystyle(d,b,a,e)_c=1+\sqrt{\frac{\lambda_B\lambda_C}{\lambda_A}}\tau_Cs+\frac{\lambda_B}{\lambda_A}s^2.
\end{array}\hspace{1cm}
\end{equation}
Notice that if we fix the first six parameters, then these cross ratios depend only on $s$ (and not $t$), and are strictly increasing with $s$. Moreover, all three of them converge to $1$ as $s$ converges to $0$ and grow arbitrarily large as $s$ converges to $\infty$. Thus, we can think of $s$ as the parameter ``controlling" these three cross ratios.

Observe that if we pick any two sets of four pairwise distinct points $\{x_1,x_2,x_3,x_4\}$ and $\{x_1',x_2',x_3',x_4'\}$ in $\Rbbb\Pbbb^2$ such that no three of $\{x_1,x_2,x_3,x_4\}$ lie on the same line and no three of $\{x_1',x_2',x_3',x_4'\}$ lie on the same line, then there exists a unique projective transformation $X\in SL(3,\Rbbb)$ such that $X\cdot x_i=x_i'$ for all $i$. This implies that any equivalence class $[\Delta_0,\Delta_1,\Delta_2,\Delta_3, A, B, C]\in\Qpc$ has a representative such that $a=[1:0:0]^T$, $b=[0:1:0]^T$, $c=[0:0:1]^T$ and $f=[2:2:-1]^T$. For this representative, we can then compute that 
\begin{eqnarray*}
d&=&[-1:\frac{(d,b,a,e)_c}{t}:\frac{(f,a,c,d)_b}{2}]^T,\\
e&=&[t:-1:\frac{(e,c,b,f)_a}{2}]^T,
\end{eqnarray*} 
In fact, this is how the $t$ parameter in the Goldman parameterization is defined. (See Section 4 of Goldman \cite{Go1} for the computation.) Moreover, the ray with source $a$ through $e$ and the ray with source $b$ through $d$ are determined entirely by $s$ because $(e,c,b,f)_a$ and $(f,a,c,d)_b$ depend only on $s$ and the points $a,b,c,f$ are fixed. We can then think of $t$ as the parameter that determines where $e$ lies along the ray with source $a$ through $e$, and this determines where $d$ is because we know the cross ratio $(d,b,a,e)_c$.

By Goldman's proof of Theorem \ref{Goldman} we now have an identification between the three spaces $\Cmf(P)$, $\Qpc$ and $\Rmf^3\times(\Rbbb^+)^2$, so we will blur the distinction between them in the rest of this paper. The next definition gives names to the parameters described above.

\begin{definition}
In the above coordinate system for $\Cmf(P)$, the first six parameters  $((\lambda_A,\tau_A), (\lambda_B,\tau_B), (\lambda_C,\tau_C))$ are called the \emph{boundary invariants} and the last two parameters $(s,t)$ are called the \emph{internal parameters}.
\end{definition}

Goldman showed that every convex projective structure on $M_{g,n}$ is obtained by gluing $2g-2+n$ pairs of convex projective pairs of pants along their boundaries, and that there are $2$ dimensions worth of ways to glue any two such boundaries together. In fact, he gives an explicit parameterization of the possible ways to do such a gluing by the parameters $(u,v)\in\Rbbb^2$. We will call these parameters the \emph{twist-bulge parameters}. Since these parameters do not feature much in our paper, we will not say more about them.

Choose a pants decomposition $\Pmc$ for $M_{g,n}$, i.e. a system of $3g-3+2n$ pairwise non-intersecting, homotopically non-trivial, simple closed curves in $M_{g,n}$. This system of curves decomposes $M_{g,n}$ into $2g-2+n$ pairs of pants. Hence, to parameterize $\Cmf(M_{g,n})$, we need $3g-3+2n$ pairs of boundary invariants $(\lambda_i,\tau_i)_{i=1}^{3g-3+2n}\in\Rmf^{3g-3+2n}$ (one pair for each simple closed curve in $\Pmc$), $3g-3+n$ pairs of twist-bulge parameters $(u_i,v_i)_{i=1}^{3g-3+n}\in\Rbbb^{6g-6+2n}$ (one pair for each simple closed curve in $\Pmc$ that is not a boundary component) and $2g-2+n$ pairs of internal parameters $(s_i,t_i)_{i=1}^{2g-2+n}\in(\Rbbb^+)^{4g-4+2n}$ (one pair for each pair of pants). This implies that $\Cmf(M_{g,n})$ is a $(16g-16+8n)$-dimensional cell.

\subsection{A reparameterization}\label{areparameterization}

Next, we will describe an order $3$ rotational symmetry of the hexagon $H_\Ppc$ for any $\Ppc$ in $\Cmf(P)$ (see Notation \ref{hexagonnotation}). Since the Goldman parameters do not behave very well under this rotational symmetry, we will also give a slightly different parameterization of the $2$-dimensional open cell fiber in Theorem \ref{Goldman} in order to exploit this symmetry to simplify the proof of our result.

Consider any properly convex hexagon $H$ in $\Rbbb\Pbbb^2$ with vertices $a,f,b,d,c,e$ in that order. (See Figure \ref{goldman}.) For any $x,y\in\{a,b,c,d,e,f\}$ such that $x\neq y$, let $Cr_{x,y}(H):=(p,q,r,s)_x$, where 
\begin{enumerate}[(1)]
\item $\{p,q,r,s\}=\{a,b,c,d,e,f\}\setminus\{x,y\}$,
\item $x,p,q,r,s$ lie along $\partial H$ in that order.
\end{enumerate}
There are thirty such cross ratios. Also, for any $\Ppc\in\Cmf(P)$, define $Cr_{x,y}(\Ppc):=Cr_{x,y}(H_\Ppc)$.  

We will be using some of these thirty cross ratios to give lower bounds for the lengths of closed curves. The main reason we chose these thirty cross ratios is that they have relatively simple closed form expressions in terms of the Goldman coordinates (and in the new coordinates as well, which we will see later). 

For any $R=((\lambda_A,\tau_A), (\lambda_B,\tau_B), (\lambda_C,\tau_C))\in\Rmf^3$ define the real valued functions 
\[\rho_1^R(x):=1+\sqrt{\frac{\lambda_C\lambda_A}{\lambda_B}}\tau_Ax+\frac{\lambda_C}{\lambda_B}x^2,\]
\[\rho_2^R(x):=1+\sqrt{\frac{\lambda_A\lambda_B}{\lambda_C}}\tau_Bx+\frac{\lambda_A}{\lambda_C}x^2,\]
\[\rho_3^R(x):=1+\sqrt{\frac{\lambda_B\lambda_C}{\lambda_A}}\tau_Cx+\frac{\lambda_B}{\lambda_A}x^2\]
with domain $\Rbbb^+$. Note that for any $R\in\Rmf^3$, each $\rho_i^R$ is strictly increasing with image $(1,\infty)$. Moreover, for any $\Ppc=(R,s,t)\in\Rmf^3\times(\Rbbb^+)^2$, we have that $\rho_1^R(s)=Cr_{a,d}(\Ppc)$, $\rho_2^R(s)=Cr_{b,e}(\Ppc)$, $\rho_3^R(s)=Cr_{c,f}(\Ppc)$ by (\ref{importantcrossratio}). This implies the following easy consequence, which we record as a lemma.

\begin{lem}\label{symmetrys}
For any $\Ppc=(R,s,t)\in\Rmf^3\times(\Rbbb^+)^2$, we have 
\[(\rho_1^R)^{-1}\big(Cr_{a,d}(\Ppc)\big)=(\rho_2^R)^{-1}\big(Cr_{b,e}(\Ppc)\big)=(\rho_3^R)^{-1}\big(Cr_{c,f}(\Ppc)\big)=s.\]
\end{lem}

This gives a coordinate free and symmetric description of the parameter $s$. (The sense in which this description is symmetric will be justified later.) Unfortunately, we do not have such a symmetric description for $t$, so we need to replace $t$ with a new parameter. This motivates the next lemma.

\begin{lem}\label{symmetryr}
For any $\Ppc\in\Cmf(P)$, we have 
\[(Cr_{a,e}(\Ppc)-1)Cr_{c,f}(\Ppc)=(Cr_{c,d}(\Ppc)-1)Cr_{b,e}(\Ppc)=(Cr_{b,f}(\Ppc)-1)Cr_{a,d}(\Ppc).\]
\end{lem}

\begin{proof}
Let $\Ppc=(R,s,t)\in\Rmf^3\times(\Rbbb^+)^2$. We can compute (see Section \ref{crossratiocomputationsandformulas}) that 
\begin{align*}
(Cr_{a,e}(\Ppc)-1)\rho_3^R(s)&=t\rho_2^R(s),\\
(Cr_{c,d}(\Ppc)-1)\rho_2^R(s)&=t\rho_2^R(s),\\
(Cr_{b,f}(\Ppc)-1)\rho_1^R(s)&=t\rho_2^R(s).\\
\end{align*}
This, together with Lemma \ref{symmetrys}, proves the lemma.
\end{proof}

Taking the three equal expressions in Lemma \ref{symmetryr} as our new eighth parameter, which we denote as $r$, we get our reparameterization of $\Cmf(P)$.

\begin{prop}\label{newparameterization}
The map $\Psi:\Rmf^3\times(\Rbbb^+)^2\to\Rmf^3\times(\Rbbb^+)^2$ given by $\Psi:(R,s,t)\mapsto(R,s,r)$, where 
\[r=t\rho_2^R(s)\]
is a diffeomorphism.
\end{prop}
\begin{proof}
The inverse map is given by 
\[\Psi^{-1}:(R,s,r)\mapsto(R,s,\frac{r}{\rho_2^R(s)}).\]
\end{proof}
For the rest of the paper, the parameterization we use for $\Cmf(P)=\Qpc=\Rmf^3\times(\Rbbb^+)^2$ will be the one where the last coordinate is $r$ given in Proposition \ref{newparameterization}.

Next, we will carefully describe the symmetric property of $s$ and $r$ that we mentioned above. There is a natural $\Zbbb_3$ action on $\Qpc$ which cyclically permutes $\Delta_1,\Delta_2,\Delta_3$ and $A,B,C$, i.e. if $\Zbbb_3=\{e,g,g^{-1}\}$, then 
\[g\cdot [\Delta_0,\Delta_1,\Delta_2,\Delta_3, A, B, C] = [\Delta_0,\Delta_2,\Delta_3,\Delta_1, B, C, A].\] 
We can interpret this action in the following way. Consider the marked convex $\Rbbb\Pbbb^2$ structure $\Ppc$ and the hexagon $H_\Ppc\subset\overline{\Omega_\Ppc}$. The marking endows the vertices of $H_\Ppc$ with a labeling as described in Notation \ref{hexagonnotation}. Then the action of $g$ is simply a cyclic relabeling of the vertices of $H_\Spc$. (See Figure \ref{relabel}.) The next proposition computes this action in terms of our parameterization.

\begin{figure}
\includegraphics[scale=0.4]{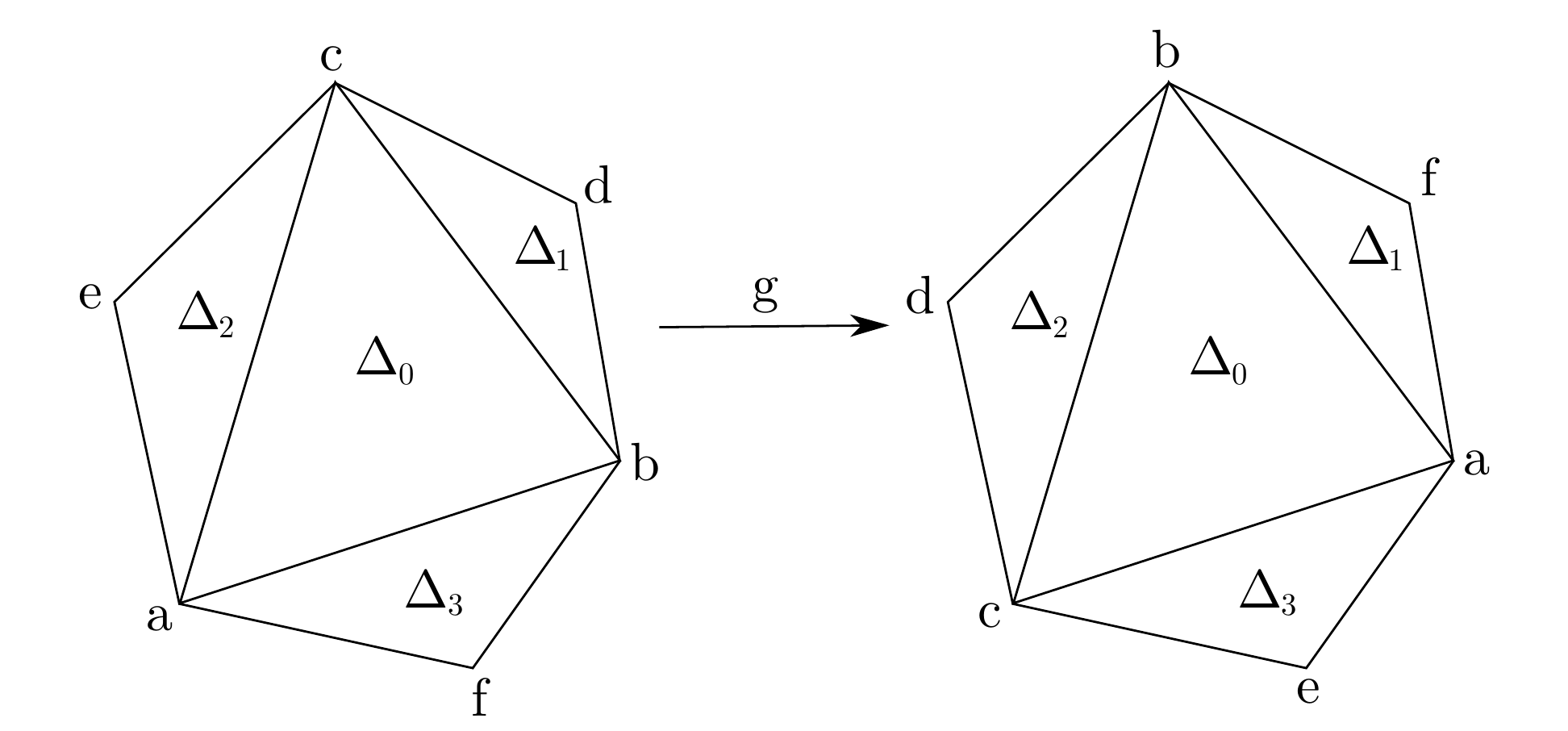}
\caption{$g: [\Delta_0,\Delta_1,\Delta_2,\Delta_3, A, B, C] \mapsto [\Delta_0,\Delta_2,\Delta_3,\Delta_1, B, C, A].$}
\label{relabel}
\end{figure}

\begin{prop}\label{symmetrysr}
Let $g\in\Zbbb_3$ be such that 
\[g\cdot [\Delta_0,\Delta_1,\Delta_2,\Delta_3, A, B, C] = [\Delta_0,\Delta_2,\Delta_3,\Delta_1, B, C, A].\] 
Then in the new parameterization of $\Qpc$, we have 
\[g\cdot \big((\lambda_A,\tau_A), (\lambda_B,\tau_B), (\lambda_C,\tau_C), s, r\big)=\big((\lambda_B,\tau_B), (\lambda_C,\tau_C), (\lambda_A,\tau_A), s, r\big).\]
\end{prop}

\begin{proof}
Since $g$ sends $A$ to $B$, $B$ to $C$ and $C$ to $A$, it is clear that 
\[g\cdot \big((\lambda_A,\tau_A), (\lambda_B,\tau_B), (\lambda_C,\tau_C), s, r\big)=\big((\lambda_B,\tau_B), (\lambda_C,\tau_C), (\lambda_A,\tau_A), s', r'\big)\] for some $s',r'\in\Rbbb^+$. First, we will show that $s'=s$. Let 
\[\Ppc=[\Delta_0,\Delta_1,\Delta_2,\Delta_3, A, B, C]=(R,s,r),\]
\[\Ppc'=[\Delta_0,\Delta_2,\Delta_3,\Delta_1, B, C, A]=(R',s',r'),\]
and note that $Cr_{a,d}(\Ppc)=Cr_{c,f}(\Ppc')$. This means that $\rho_1^R(s)=\rho_3^{R'}(s')$, so $s'=s$ because $\rho_1^R\equiv\rho_3^{R'}$ as functions, and they are both injective.

Next, we show that $r=r'$. Observe:
\[Cr_{a,e}(\Ppc)=Cr_{c,d}(\Ppc'),\]
\[Cr_{c,f}(\Ppc)=Cr_{b,e}(\Ppc').\] 
Thus, 
\[r=(Cr_{a,e}(\Ppc)-1)Cr_{c,f}(\Ppc)=(Cr_{c,d}(\Ppc')-1)Cr_{b,e}(\Ppc')=r'.\]
\end{proof}
This shows that with our choice of parameterization of the $2$-cell fibers of $\Theta$ in Theorem \ref{Goldman}, $\Zbbb_3$ acts as the identity on the the fibers of $\Theta$, and only permutes the boundary invariants of $\Cmf(P)$. It is in this sense that the parameters $s$ and $r$ are symmetric.

\subsection{Topological entropy of the geodesic flow}\label{entropygeodesicflow} 

In this subsection, we will give a brief description of some dynamics that naturally occurs in our set up. Suppose first that $\Mmf$ is a closed convex $\Rbbb\Pbbb^2$ surface. Since the Hilbert metric is a Finsler metric, $\Mmf$ induces a geodesic flow $\phi=\{\phi_t\}_{t\in\Rbbb}$ on $T^1\Mmf$, the unit tangent bundle of $\Mmf$. We will now define the topological entropy of this flow. 

\begin{definition}
Let $\phi$ be a flow on a compact manifold $X$. Choose a metric $d$ on $X$, and for each $t\in\Rbbb$, define $d_t:X\times X\to\Rbbb_{\geq 0}$ by 
\[d_t(a,b)=\sup\{d\big(\phi_s(a),\phi_s(b)\big):0\leq s\leq t\}.\] 
One can verify that $d_t$ is in fact a metric on $X$. For any $\epsilon>0$ and any $t>0$, consider the set of all open covers of $X$ satisfying the following property: every open set in the open cover has a $d_t$ diameter of at most $\epsilon$. Then define $D(\epsilon,t)$ to be the size of such an open cover with the fewest number of open sets. The \emph{topological entropy} of the flow $\phi$ is the quantity
\[h_{top}(\phi):=\limsup_{\epsilon\to 0}\lim_{t\to\infty}\frac{1}{t}\log\big(D(\epsilon,t)\big).\]
\end{definition}

It is known that the topological entropy is in fact independent of the choice of the metric $d$, and thus depends only on the flow and the topology of $X$. For more details, one may refer to Chapter 3.1 of Hasselblatt-Katok \cite{HaKa1}.

In the case when $X=T^1\Mmf$ and $\phi$ is the geodesic flow of the Hilbert metric on $\Mmf$, this quantity is interesting because we can think of it as a measure of how different $\Mmf$ is from a hyperbolic surface. By Crampon \cite{Cr1}, we know that in the case when $M$ is a closed surface, $h_{top}(\phi)\in(0,1]$ for any $\Mpc=[f,\Mmf]\in\Cmf(M)$, and $h_{top}(\phi)=1$ if and only if $\Mpc$ lies in the Fuchsian locus of $\Cmf(M)$. Thus, if we have a sequence in $\Cmf(M)$ on which the topological entropy converges to $0$, then the dynamics of the geodesic flow is becoming less and less like that of a hyperbolic surface as we move along this sequence.

By Theorem 1.1 of Benoist \cite{Be1}, we know that $\phi$ is Anosov. Moreover, the topological transitivity of the action of $\pi_1(M)$ on the set of pairs of distinct points on $\partial\Omega_\Mpc$ implies that the periodic points for $\phi$ are dense in $T^1\Mmf$. Theorem B of Bowen \cite{Bo1} then allows us to compute the topological entropy of $\phi$ by the formula 

\begin{equation}\label{Bowen}
h_{top}(\phi)=\lim_{t\to\infty}\frac{1}{t}\log\big(R(t)\big)
\end{equation}
where $R(t)$ is the number of closed orbits of $\phi$ with period at most $t$.

In the more general case when $\Mmf=\Omega/\Gamma$ possibly has boundary, we use the following generalization of the geodesic flow of the Hilbert metric. Define $U\Mmf$ to be the set of points $(p,v)$ in $T^1\Mmf$ such that the geodesic through $p$ tangential to $v$ has endpoints in the limit set of $\Gamma$ in $\partial\Omega$. This is also known as the non-wandering set of the geodesic flow. It is easy to see that $U\Mmf$ is compact.

\begin{definition}
For any compact projective surface $\Mmf$, define the \emph{geodesic flow} of $\Mmf$ to be the geodesic flow $\phi$ of the Hilbert metric on $\Mmf$ restricted to the subset $U\Mmf$ of $T^1\Mmf$. Denote the topological entropy of the geodesic flow of $\Mmf$ by $h_{top}(\Mmf)$. For any convex $\Rbbb\Pbbb^2$ structure $\Mpc=[f,\Mmf]$ on $M$, we define $h_{top}(\Mpc):=h_{top}(\Mmf)$.
\end{definition}

When $\Mmf$ is closed, $U\Mmf=T^1\Mmf$, so $h_{top}(\Mmf)=h_{top}(\phi)$. Moreover, even in the case when $\Mmf$ is not closed, $U\Mmf$ is a hyperbolic set for the geodesic flow $\phi$ of $\Mmf$. To see this, consider two copies of $\Mmf$ and the obvious pairing of the boundary components of these two copies. Choose twist-bulge parameters for each of these pair of boundary components to glue the two copies of $\Mmf$ together. Let $\Mmf'$ be the closed convex $\Rbbb\Pbbb^2$ surface obtained this way and let $\psi$ be the geodesic flow of $\Mmf'$ (acting on $T^1\Mmf'$). Then note that $U\Mmf$ is a $\psi$-invariant subset of $T^1\Mmf'$, and $\psi|_{U\Mmf}$ is exactly $\phi$. The hyperbolicity of $U\Mmf$ then follows immediately from the Anosovness of $\psi$. This allows us to use a result by Pollicott (see Theorem 8 of \cite{Po1}) to compute the geodesic flow of $\Mmf$ by the formula
\begin{equation}
h_{top}(\Mmf)=\lim_{t\to\infty}\frac{1}{t}\log(R_\Mmf(t))
\end{equation}
where $R_\Mmf(t)$ is the number of closed geodesics in $\Mmf$ with length at most $t$. 

\section{Main result and proofs}\label{maincontent}

\subsection{Statement of main theorem and its consequences}\label{statements}

For the rest of the paper, we will use the set up we now describe. Choose once and for all the following:
\begin{enumerate}[(1)]
\item A pants decomposition $\Pmc=\{\gamma_1,\dots,\gamma_{3g-3+2n}\}$ for $M=M_{g,n}$, 
\item A set of generators $A_0$, $B_0$ and $C_0$ for $\pi_1(P)$ such that $C_0B_0A_0=I$,
\item A diffeomorphism $f_i:P\to P_i$, where $\{P_1,\dots,P_{2g-2+n}\}$ are the closures of the connected components of 
\[M_{g,n}\setminus\Pmc.\] 
\end{enumerate}

The description of the Goldman parameterization for a pair of pants in Section \ref{goldmanparameters} tells us that this data gives us a parameterization of $\Cmf(M)$. We will now list a couple of definitions to simplify the statement of our theorem.

\begin{definition}
A closed curve $\eta$ in $M$ is \emph{typical} if $\eta$ is not homotopic to a multiple of any $\gamma\in\Pmc$. The set of homotopy classes of typical oriented closed curves in $M$ is denoted by $\Tpc_M$.
\end{definition}

\begin{definition}\label{goldmandef}
A sequence $\{\Ppc_j=(R_j,s_j,r_j)\}_{j=1}^\infty$ in $\Cmf(P)$ is a \emph{Goldman sequence} if 
\begin{enumerate}[(1)]
\item there are constants $C_1$ and $C_2$, $0<C_1<C_2<\infty$, such that for all $j$, the lengths of the boundary components of $\Ppc_j$ are bounded between $C_1$ and $C_2$.
\item for any compact set $K\subset(\Rbbb^+)^2$, there is a positive number $H$ such that if $j>H$, then $(s_j,r_j)\notin K$.
\end{enumerate}
A sequence $\{\Mpc_j\}_{j=1}^\infty$ in $\Cmf(M)$ is a \emph{Goldman sequence} if for all $i\in\{1,\dots,2g-2+n\}$, the sequence $\{\Ppc^{(i)}_j:=f_i^\#(\Mpc_j)\}_{j=1}^\infty$ (see Section \ref{thehilbertmetric}) in $\Cmf(P)$ is a Goldman sequence. 
\end{definition}

Let $\Lpc:\Cmf(M)\to\Rbbb^+$ be the function that sends every $\Mpc\in\Cmf(M)$ to the length of the shortest typical closed curve in $\Mpc$. The main theorem of this paper is the following.

\begin{thm}\label{mainthm1}
Let $\{\Mpc_j\}_{j=1}^{\infty}$ be a Goldman sequence in $\Cmf(M)$. Then 
\begin{enumerate}[(1)]
\item $\displaystyle\lim_{j\to\infty}\Lpc(\Mpc_j)=\infty$
\item $\displaystyle\lim_{j\to\infty}h_{top}(\Mpc_j)=0$
\end{enumerate}
\end{thm}

Before we begin the proof of this theorem, we will state and prove some of its corollaries. The first of these is a generalization of (1) of Theorem \ref{mainthm1}. 

\begin{cor}\label{firstcor}
Let $\{\Mpc_j\}_{j=1}^\infty$ be a sequence in $\Cmf(M)$ such that $\{f_1^\#(\Mpc_j)=:\Ppc_j\}_{j=1}^\infty$ in $\Cmf(P)$ is a Goldman sequence. Let $\eta$ be a closed curve in $M$ that cannot be homotoped to be disjoint from $P_1$ and let $\eta_j$ be the geodesic representative of $\eta$ in $\Mpc_j$. Then $\displaystyle\lim_{j\to\infty}l_{\Mpc_j}(\eta_j)=\infty$.
\end{cor}

\begin{proof}
Since $\eta$ cannot be homotoped to be disjoint from $P_1$, we know that $\eta$ is typical and one of the following must hold:
\begin{itemize}
\item for all $j$, $\eta_j$ is contained in $\Ppc_j\subset \Mpc_j$, 
\item for all $j$, there is a closed subsegment $\mu_j$ of $\eta_j$ such that $\mu_j\subset \Ppc_j$, the endpoints of $\mu_j$ lie in $\partial \Ppc_j$, and $\mu_j$ is not homotopic relative endpoints to a subset of $\partial \Ppc_j$.
\end{itemize}
If the former holds then (1) of Proposition \ref{uniquegeodesic} implies that $l_{\Mpc_j}(\eta_j)=l_{\Ppc_j}(\eta_j)$. By (1) of Theorem \ref{mainthm1}, $\displaystyle\lim_{j\to\infty}l_{\Ppc_j}(\eta_j)=\infty$.

Suppose instead that the latter holds. Let $\alpha_j$ and $\beta_j$ be the two boundary components of $\Ppc_j$ that contain the endpoints of $\mu_j$ (possibly $\alpha_j=\beta_j$) and let the endpoints of $\mu_j$ in $\alpha_j$ and $\beta_j$ be $p_j$ and $q_j$ respectively. Parameterize $\mu_j$, $\alpha_j$ and $\beta_j$ by the unit interval $[0,1]$ so that $\mu_j(0)=p_j$, $\mu_j(1)=q_j$, $\alpha_j(0)=\alpha_j(1)=p_j$, $\beta_j(0)=\beta_j(1)=q_j$, and observe that either $\alpha_j\cdot\mu_j\cdot\beta_j\cdot\mu_j^{-1}$ or $\alpha_j\cdot\mu_j\cdot\beta_j^{-1}\cdot\mu_j^{-1}$ is a typical closed curve in $\Ppc_j$. Here, $\cdot$ is concatenation and the inverse is reversing the parameterization. 

Assume without loss of generality that $\nu_j:=\alpha_j\cdot\mu_j\cdot\beta_j\cdot\mu_j^{-1}$ is typical.  Then
\begin{equation*}
l_{\Mpc_j}(\alpha_j)+l_{\Mpc_j}(\beta_j)+2l_{\Mpc_j}(\mu_j)=l_{\Mpc_j}(\nu_j)\geq\Lpc(\Ppc_j),
\end{equation*}
i.e.
\begin{equation*}
l_{\Mpc_j}(\mu_j)\geq\frac{1}{2}(\Lpc(\Ppc_j)-l_{\Mpc_j}(\alpha_j)-l_{\Mpc_j}(\beta_j)).
\end{equation*}
Thus (1) of Theorem \ref{mainthm1} implies that 
\begin{equation*}
\lim_{j\to\infty}l_{\Mpc_j}(\eta_j)\geq\lim_{j\to\infty}l_{\Mpc_j}(\mu_j)=\infty
\end{equation*}
\end{proof}

As another consequence of (1) of Theorem \ref{mainthm1}, we can also say how the maximum injectivity radius and the Busemann area of a convex $\Rbbb\Pbbb^2$ structure on a closed surface degenerate along a Goldman sequence. These results are listed as Corollary \ref{finalcor}. To prove these results, we need the following lemma.

\begin{lem}\label{growingball}
Let $\Mmf$ be a closed convex $\Rbbb\Pbbb^2$ surface and let $p$ be a point in $\Mmf$ where the injectivity radius is maximized. Let $r_p$ be the injectivity radius at $p$. Then there are pairwise distinct line segments $l_1$, $l_1'$, $l_2$, $l_2'$ in $\Mmf$ of length $r_p$ such that 
\begin{enumerate}[(1)]
\item $l_1$ and $l_1'$ have the same endpoints
\item $l_2$ and $l_2'$ have the same endpoints 
\item $p$ is a common endpoint for all four segments
\item the closed curves $l_1\cup l_1'$ and $l_2\cup l_2'$ are not homotopic relative to $p$, and neither of them are homotopically trivial.
\end{enumerate}
\end{lem}

\begin{proof}
For any point $q\in \Mmf=\Omega/\pi_1(\Mmf)$, let $r_q$ be the injectivity radius at $q$ and let $\qtd$ be any lift of $q$ in $\Omega$. Observe that there is some $g\in\pi_1(\Mmf)$ such that 
\[Hd_\Mmf(g\cdot\qtd,\qtd)=Hd_\Mmf(g^{-1}\cdot\qtd,\qtd)=2r_q.\] 
Also, for all $g\in\pi_1(\Mmf)\setminus\{\id\}$, we know that $Hd_\Mmf(g\cdot\qtd,\qtd)\geq 2r_q$. 

Note that $p$ exists because $\Mmf$ is closed. Choose a lift $\ptd$ of $p$ in $\Omega$, and suppose for contradiction that up to taking inverses, there is a unique $g_1\in\pi_1(\Mmf)$ such that $Hd_\Mmf(g_1\cdot \ptd,\ptd)=2r_p$. This is an open condition, so there is a neighborhood $U$ of $p$ with the following property: for all $p'\in U$ and up to taking inverses, $g_1\in\pi_1(\Mmf)$ is the unique element such that $Hd_\Mmf(g_1\cdot \ptd',\ptd')=2r_{p'}$. Let $g_1^+$, $g_1^-$ and $g_1^0$ be the three fixed points of $g_1$ in $\Rbbb\Pbbb^2$, where $g_1^+$ is the attracting fixed point and $g_1^-$ is the repelling fixed point. By renormalizing, we can assume that 
\[ g_1^+=\left[\begin{array}{c} 1\\ 0\\ 0 \end{array} \right], 
g_1^0=\left[ \begin{array}{c} 0\\ 1\\ 0 \end{array} \right], 
g_1^-=\left[ \begin{array}{c} 0\\ 0\\ 1 \end{array} \right],
\ptd=\left[ \begin{array}{c} 1\\ 1\\ 1 \end{array} \right].\]
This implies that
\[g_1=\left[\begin{array}{ccc} \lambda&0&0\\ 0&\mu&0\\ 0&0&\nu \end{array} \right]\]
for some $\lambda$, $\mu$, $\nu$ such that $\lambda>\mu>\nu>0$, $\lambda\mu\nu=1$. Define $\Delta$ to be the closed triangle
\[\Delta:=\Bigg\{\left[ \begin{array}{c} x\\ y\\ z \end{array} \right]:x,y,z\geq 0\Bigg\},\]
choose $\beta>1$ such that 
\[\ptd':=\left[ \begin{array}{c} 1\\ \beta\\ 1 \end{array} \right]\]
is a lift of some $p'\in U$, and note that $\ptd$ and $\ptd'$ both lie in $\Delta$. Let $L$ be the line in $\Rbbb\Pbbb^2$ that contains $\ptd$ and $g_1\cdot\ptd$, and let $L'$ to be the line that contains $\ptd'$ and $g_1\cdot\ptd'$. Then define $l$ and $l'$ to be the line segments of $L$ and $L'$ that lie in $\Omega$ and have endpoints in $\partial\Omega$. One can easily compute that 
\[L=\Bigg\{\left[\begin{array}{c} a+b\lambda\\ a+b\mu\\ a+b\nu \end{array} \right]:a,b\in\Rbbb\Bigg\}, L'=\Bigg\{\left[\begin{array}{c} a+b\lambda\\ (a+b\mu)\beta\\ a+b\nu \end{array} \right]:a,b\in\Rbbb\Bigg\}.\]
It is then easy to see that $L$ and $L'$ intersect at the point 
\[\left[\begin{array}{c} \lambda-\mu\\ 0 \\\nu-\mu \end{array} \right].\]
Since $\lambda-\mu>0$ and $\nu-\mu<0$, this point does not lie in $\Delta$ (and hence not in $\Omega$), so $l$ and $l'$ do not intersect. By applying Proposition \ref{crossratioline} and Proposition \ref{crossratioinequality}, it is clear that $2r_p=Hd_\Mmf(g_1\cdot \ptd,\ptd)<Hd_\Mmf(g_1\cdot \ptd',\ptd')=2r_{p'}$ (see Figure \ref{enlargingrp}). However, this contradicts the maximality of the injectivity radius at $p$.

\begin{figure}
\includegraphics[scale=0.7]{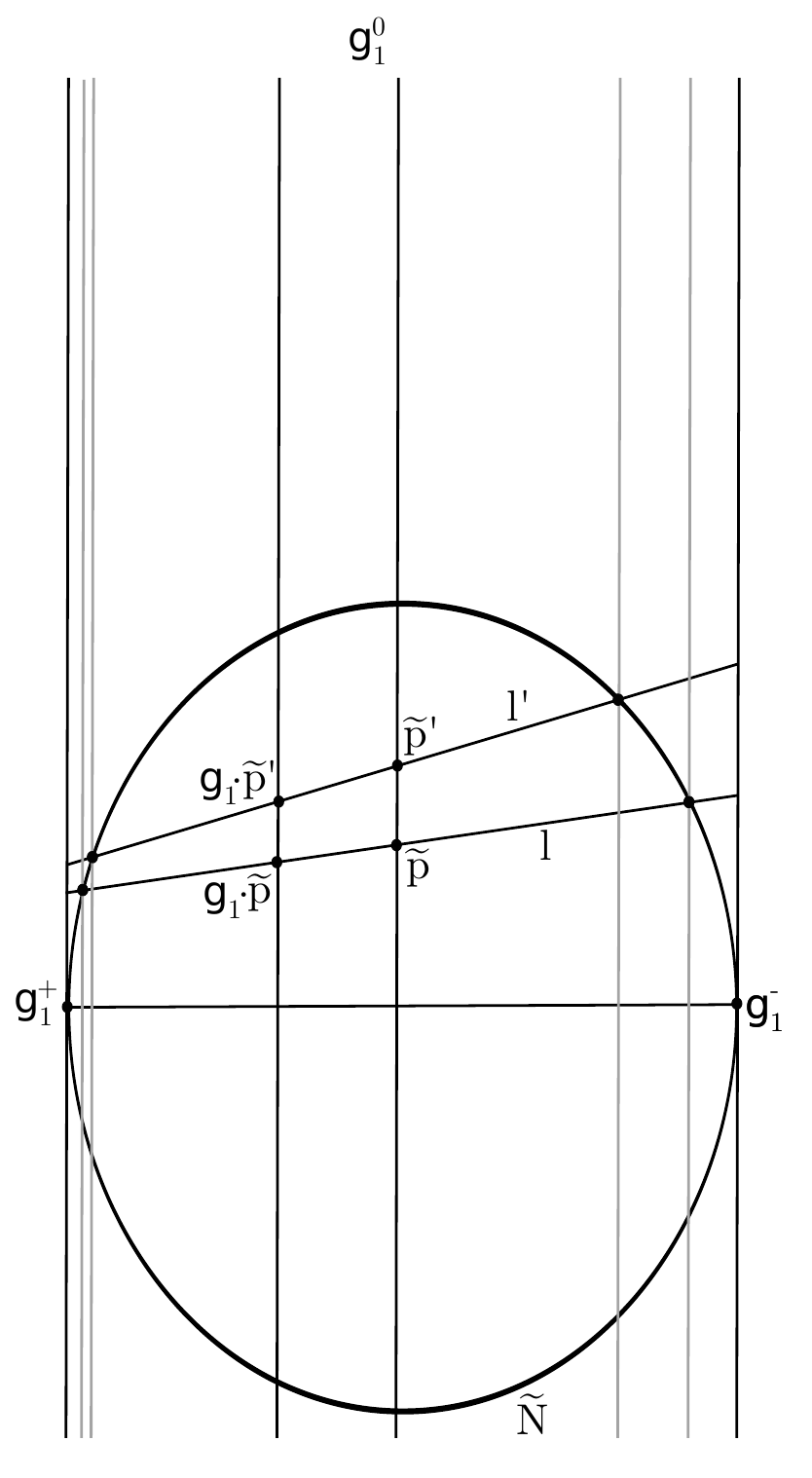}
\caption{$Hd_N(g_1\cdot \ptd,\ptd)<Hd_N(g_1\cdot \ptd',\ptd')$}
\label{enlargingrp}
\end{figure}

Thus, there are at least two group elements $g_1$, $g_2$ in $\pi_1(\Mmf)$ such that $g_1\neq g_2\neq g_1^{-1}$ and	
\[Hd_\Mmf(g_1\cdot\ptd,\ptd)=Hd_\Mmf(g_1^{-1}\cdot\ptd,\ptd)=Hd_\Mmf(g_2\cdot\ptd,\ptd)=Hd_\Mmf(g_2^{-1}\cdot\ptd,\ptd)=2r_p.\]
Let $\xtd_1$, $\xtd_1'$, $\xtd_2$, $\xtd_2'$ be the midpoints of the line segments in $\Omega$ between $\ptd$ and $g_1\cdot\ptd$, $g_1^{-1}\cdot\ptd$, $g_2\cdot\ptd$, $g_2^{-1}\cdot\ptd$ respectively. Then for $i=1,2$, let $\ltd_i$, $\ltd_i'$ be the line segments in $\Omega$ between $\ptd$ and $\xtd_i$, $\xtd_i'$ respectively. Define $x_i:=\Pi(\xtd_i)=\Pi(\xtd_i')$, $l_i:=\Pi(\ltd_i)$ and $l_i':=\Pi(\ltd_i')$ for $i=1,2$, where $\Pi:\Omega\to \Mmf$ is the covering map. It is clear that $l_1$, $l_1'$, $l_2$, $l_2'$ are pairwise distinct and that (1), (2) and (3) hold. To get (4), simply note that the closed curves $l_1\cup l_1'$ and $l_2\cup l_2'$ correspond to $g_1$ and $g_2$ in $\pi_1(\Mmf)$ respectively. Since $g_2^{-1}\neq g_1\neq g_2$ and $g_1\neq\id\neq g_2$, the two closed curves cannot be homotopic relative to $p$, and neither of them are homotopically trivial.
\end{proof}

In the proof of Corollary \ref{finalcor}, we will also make use of the following well-known result by Benz\'ecri \cite{Bz1}.

\begin{thm}\label{benzecri}
(Benz\'ecri) Define 
\[\Epc:=\{(x,\Omega):\Omega\subset\Rbbb\Pbbb^n\text{ is open and properly convex }, x\in\Omega\}\]
and equip $\Epc$ with the Hausdorff topology. Then the natural action of $PGL(n+1,\Rbbb)$ on $\Epc$ is cocompact.
\end{thm}

\begin{cor}\label{finalcor}
Let $M$ be a closed surface. Let $\Ipc:\Cmf(M)\to\Rbbb^+$ be the function that maps each $\Mpc\in\Cmf(M)$ to the maximal injectivity radius over all points in $\Mpc$, and let $\Apc:\Cmf(M)\to\Rbbb^+$ be the function that maps each $\Mpc\in\Cmf(M)$ to the Busemann area of $\Mpc$. If $\{\Mpc_j\}_{j=1}^\infty$ is a Goldman sequence in $\Cmf(M)$, then 
\begin{enumerate}[(1)]
\item $\displaystyle\lim_{j\to\infty}\Ipc(\Mpc_j)=\infty$
\item $\displaystyle\lim_{j\to\infty}\Apc(\Mpc_j)=\infty$.
\end{enumerate}
\end{cor}

\begin{proof}
Proof of (1). Let $\Mpc$ be any marked closed convex $\Rbbb\Pbbb^2$ surface and let $p$, $l_i$ and $l_i'$ for $i=1,2$ be as defined in Lemma \ref{growingball}. Let $\eta_i:[0,1]\to l_i\cup l_i'$ be a parameterization such that $\eta_i(0)=\eta_i(1)=p$. Since $l_1\cup l_1'$ and $l_2\cup l_2'$ are not homotopic relative $p$ and neither of them are homotopically trivial, either $\eta_1\cdot\eta_2$ or $\eta_1\cdot\eta_2^{-1}$ is typical, so $4r_p\geq\Lpc(\Mpc)$. Since this is true for all $\Mpc\in\Cmf(M)$, (1) of Theorem \ref{mainthm1} implies that $\displaystyle\lim_{j\to\infty}\Ipc(\Mpc_j)=\infty$.

Proof of (2). Let $p_j$ be a point in $\Mpc_j$ where the injectivity radius is maximized. It is sufficient to show that $\displaystyle\lim_{j\to\infty}\nu_{N_j}(B_{r_{p_j}}(p_j))=\infty$. Choose a lift $\ptd_j$ of $p_j$ in $\Omega_{\Mpc_j}$. By Theorem \ref{benzecri}, we can assume that the sequence $(\ptd_j,\Omega_{\Mpc_j})$ converges to some $(x,\Omega)$ in $\Epc$. By (1), we know that $\displaystyle\lim_{j\to\infty}r_{p_j}=\infty$. Hence, the sequence  $(\ptd_j,B_{r_{p_j}}(\ptd_j))$ also converges to $(x,\Omega)$, so 
\[\lim_{j\to\infty}\nu_{N_j}(B_{r_{p_j}}(p_j))=\lim_{j\to\infty}\nu_{N_j}(B_{r_{p_j}}(\ptd_j))=\nu_{\Omega}(\Omega)=\infty.\]
\end{proof}

The next corollary gives a positive answer to a question posted by Crampon and Marquis (Question 13 of \cite{Ques}). Let $M$ be a closed surface. They asked if there is, for any number $\alpha\in[0,1]$, a diverging sequence $\{\Mpc_j\}_{j=1}^\infty$ in $\Cmf(M)$ such that $\displaystyle\lim_{j\to\infty}h_{top}(\Mpc_j)=\alpha$. Previously, Nie proved in \cite{Ni1} that there exists a diverging sequence $\{\Mpc_j\}_{j=1}^\infty$ in $\Cmf(M)$ such that $\displaystyle\lim_{j\to\infty}h_{top}(\Mpc_j)=0$. 

\begin{cor}\label{cramponquestion}
Suppose that $M$ is a closed surface. For any number $\alpha\in[0,1]$, there is a diverging sequence $\{\Mpc_j\}_{j=1}^\infty$ in $\Cmf(M)$ such that $\displaystyle\lim_{j\to\infty}h_{top}(\Mpc_j)=\alpha$.
\end{cor}

\begin{proof}
Consider any diverging sequence $\{\Mpc_j'\}$ in $\Tmf(M)\subset\Cmf(M)$ corresponding to pinching the curves in the pants decomposition of $M$, i.e. the largest and smallest eigenvalues of the holonomy about the closed curves in the pants decomposition converge to $1$ as $j$ approaches $\infty$. Let $(s_j^{(i)},r_j^{(i)})$ be the internal parameters for $f_i^\#(\Mpc_j')$. By Section 4 of Goldman \cite{Go1}, we know that $s_j^{(i)}=1$ for all $i\in\{1,\dots,2g-2+n\}$ and for all $j$.

Now, for each $j$, consider the one parameter family $s\mapsto\Mpc_{j,s}'$ in $\Cmf(M)$, where all the boundary invariants and twist-bulge parameters for $\Mpc_{j,s}'$ are the same as that of $\Mpc_j'$, but the internal parameters of each $f_i^\#(\Mpc_{j,s}')$ is $(s,r_j^{(i)})$. Observe that $h_{top}(\Mpc_{j,1}')=1$. Also, Theorem \ref{mainthm1} implies that $\displaystyle\lim_{s\to\infty}h_{top}(\Mpc_{j,s}')=0$. Since $h_{top}$ is continuous (see Section 3 of Sambarino \cite{Sa1} or Proposition 8.3 of Crampon \cite{Cr2}), we know that there is some $s_j>1$ such that $h_{top}(\Mpc_{j,s_j}')=\alpha$. Let $\Mpc_j=\Mpc_{j,s_j}'$. Then the sequence $\{\Mpc_j\}_{j=1}^\infty$ is diverging because the curves in the pants decomposition of $M$ are getting pinched, and $\displaystyle\lim_{j\to\infty}h_{top}(\Mpc_j)=\alpha$.
\end{proof}

The rest of this paper will be devoted to proving Theorem \ref{mainthm1}.

\subsection{Decomposition of closed geodesics}\label{setup}

We want to find a way to decompose every closed geodesic on a marked convex $\Rbbb\Pbbb^2$ surface into line segments, so that we can bound the length of each line segment from below by a number that depends on the Goldman parameters. We will give a rough description of the idea behind this decomposition before formally describing it. Let $\Mpc\in\Cmf(M)$ be a marked convex $\Rbbb\Pbbb^2$ surface. Since we have chosen a pants decomposition $\Pmc$ on $M$, the marking induces a pants decomposition on $\Mpc$, which we also denote by $\Pmc$. Any closed geodesic $\eta$ on $\Mpc$ can be thought of as a union of two types of line segments. The first type are line segments that ``wind around" collar neighborhoods of the simple closed curves in $\Pmc$ (see Figure \ref{winding}) and the second type are line segments that ``go between" these collar neighborhoods (see Figure \ref{gobetween}). 

\begin{figure}
\includegraphics[scale=0.5]{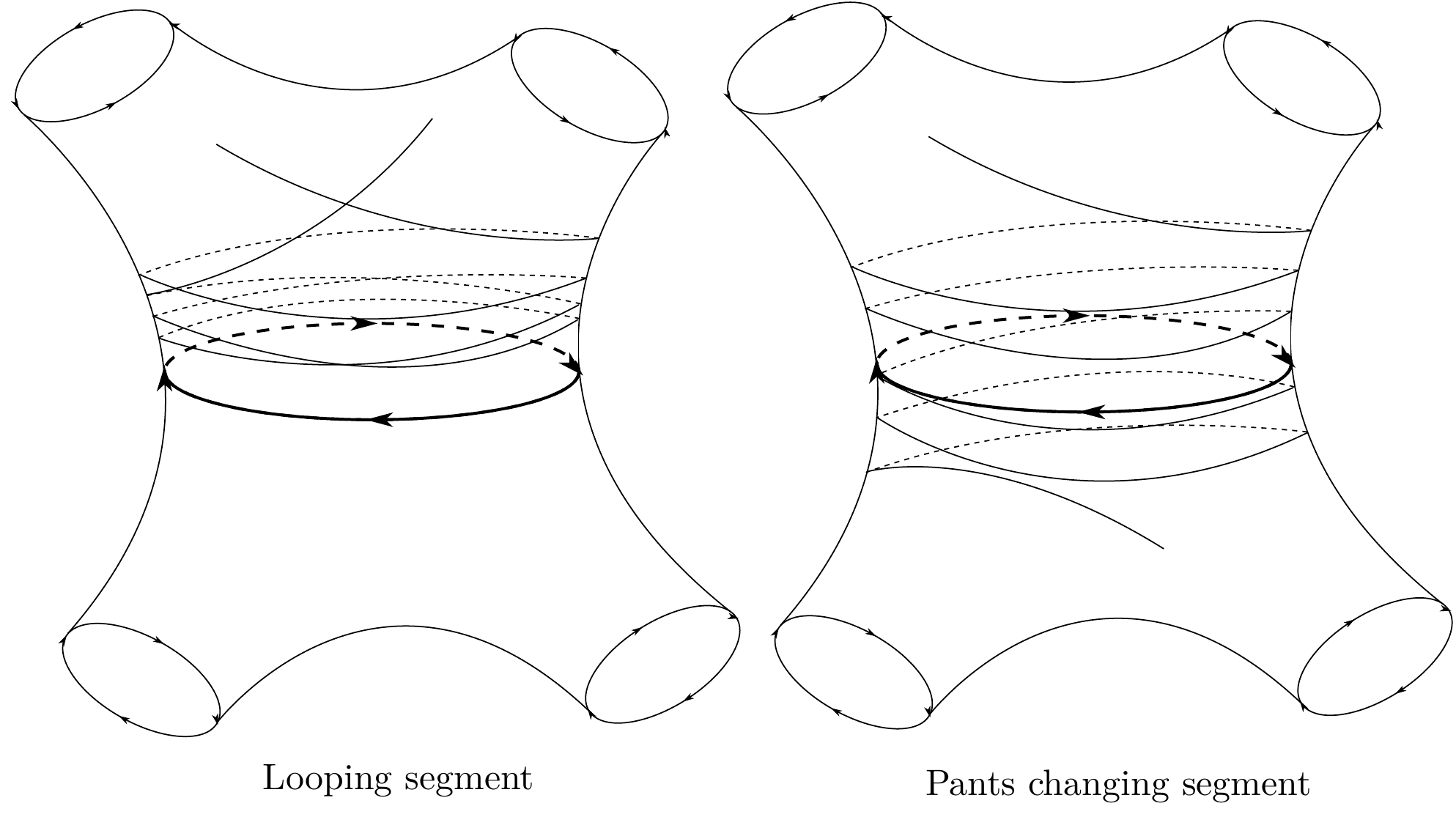}
\caption{Segments that ``wind around" collar neighborhoods}
\label{winding}
\end{figure}

\begin{figure}
\includegraphics[scale=0.6]{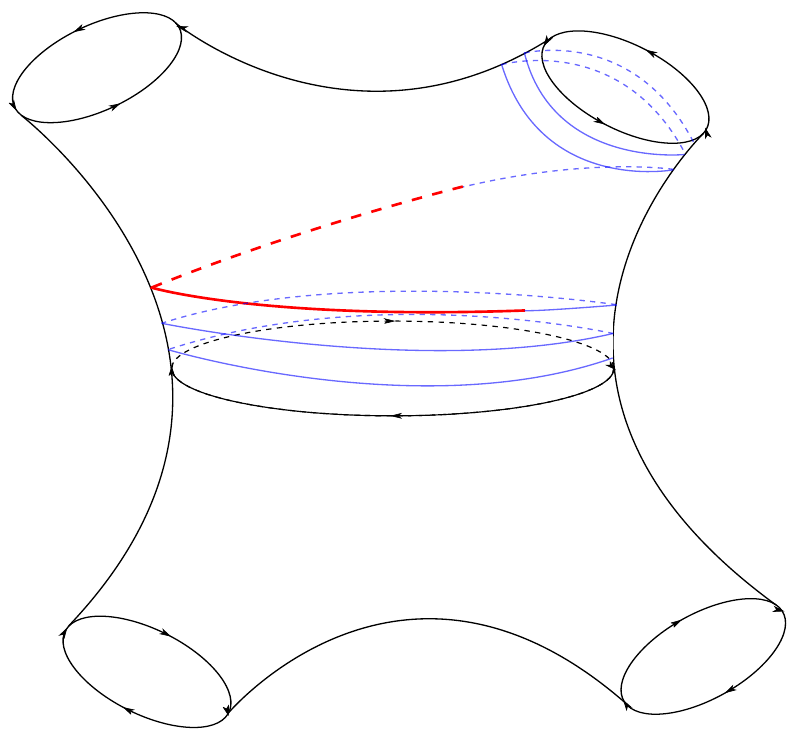}
\caption{Segment that ``goes between" collar neighborhoods}
\label{gobetween}
\end{figure}

The length of the first type of line segments can be bounded below by a multiple (depending on the number of times they ``wind around" the collar neighborhood) of the infimum of the lengths of the simple closed curves in $\Pmc$. On the other hand, we will later show that the length of the second type of line segments can be bounded below by some combination of the logarithm of the thirty cross ratios described in Section \ref{areparameterization}. 

To formally describe how to decompose a closed geodesic into these types of line segments, it is convenient to use an ideal triangulation $\Umc_\Mpc$, which we will now define. For each $i=1,\dots,2g-2+n$, let $\Ppc^{(i)}:=f_i^\#(\Mpc)\in\Cmf(P)$ (see Section \ref{thehilbertmetric}). We can then construct the ideal triangulation $\Vmc^{(i)}$ on $\Ppc^{(i)}$ corresponding to $A_0$, $B_0$, $C_0$ for each $i\in\{1,\dots,2g-2+n\}$. (See Definition \ref{idealtriangulation}.), which lifts to a lamination $\widetilde{\Vmc}^{(i)}$ on $\Omega_{\Ppc^{(i)}}$. (This is partially drawn in Figure \ref{goldman}.) Define the lamination 
\[\Umc_\Mpc:=\Pmc\cup\bigcup_{i=1}^{2g-2+n}\Vmc^{(i)}\]
on $\Mpc$, which lifts to a lamination $\widetilde{\Umc}_\Mpc$ on $\Omega_\Mpc$. We call every element in $\widetilde{\Umc}_\Mpc$ or $\Umc_\Mpc$ an \emph{edge} of $\widetilde{\Umc}_\Mpc$ or $\Umc_\Mpc$ respectively. Also, we say two edges in $\Umc_\Mpc$ \emph{share a common vertex} if they have lifts to $\widetilde{\Umc}_\Mpc$ that share a common vertex.

Before we proceed to describe the decomposition of closed geodesics into line segments, we will first take a closer look at the ideal triangulation and some related structure. Observe that the three line segments in $\displaystyle\Omega_{\Ppc^{(i)}}$ with endpoints $a$ and $b$, $b$ and $c$, $c$ and $a$ (see Figure \ref{goldman}) are lifts of the three leaves of $\Ppc^{(i)}$. Moreover, for each boundary component $\gamma$ of $\Ppc^{(i)}$, there are exactly two leaves in $\Vmc^{(i)}$ that accumulate to $\gamma$.

Using this, we can construct the following:
\begin{itemize}
\item ($\Gamma_q$ and $\Omega_q^{(i)}$) Let $q$ be any endpoint of any leaf in $\widetilde{\Vmc}^{(i)}$, and define $\Gamma_q$ to be the stabilizer of $q$ in $\pi_1(M)$. Let $\Delta_0'$ and $\Delta_1'$ be two open triangles in $\displaystyle\Omega_{\Ppc^{(i)}}\setminus\widetilde{\Vmc}^{(i)}$ that share an edge and have $q$ as a common vertex, and let $\Delta_j$ be the closure of $\Delta_j'$ in $\Omega_{\Ppc^{(i)}}$ for $j=0,1$. Then define $\Omega_q^{(i)}$ to be the interior of $\Gamma_q\cdot(\Delta_0\cup\Delta_1)$. (See Figure \ref{constructions}.)

\begin{figure}
\includegraphics[scale=0.7]{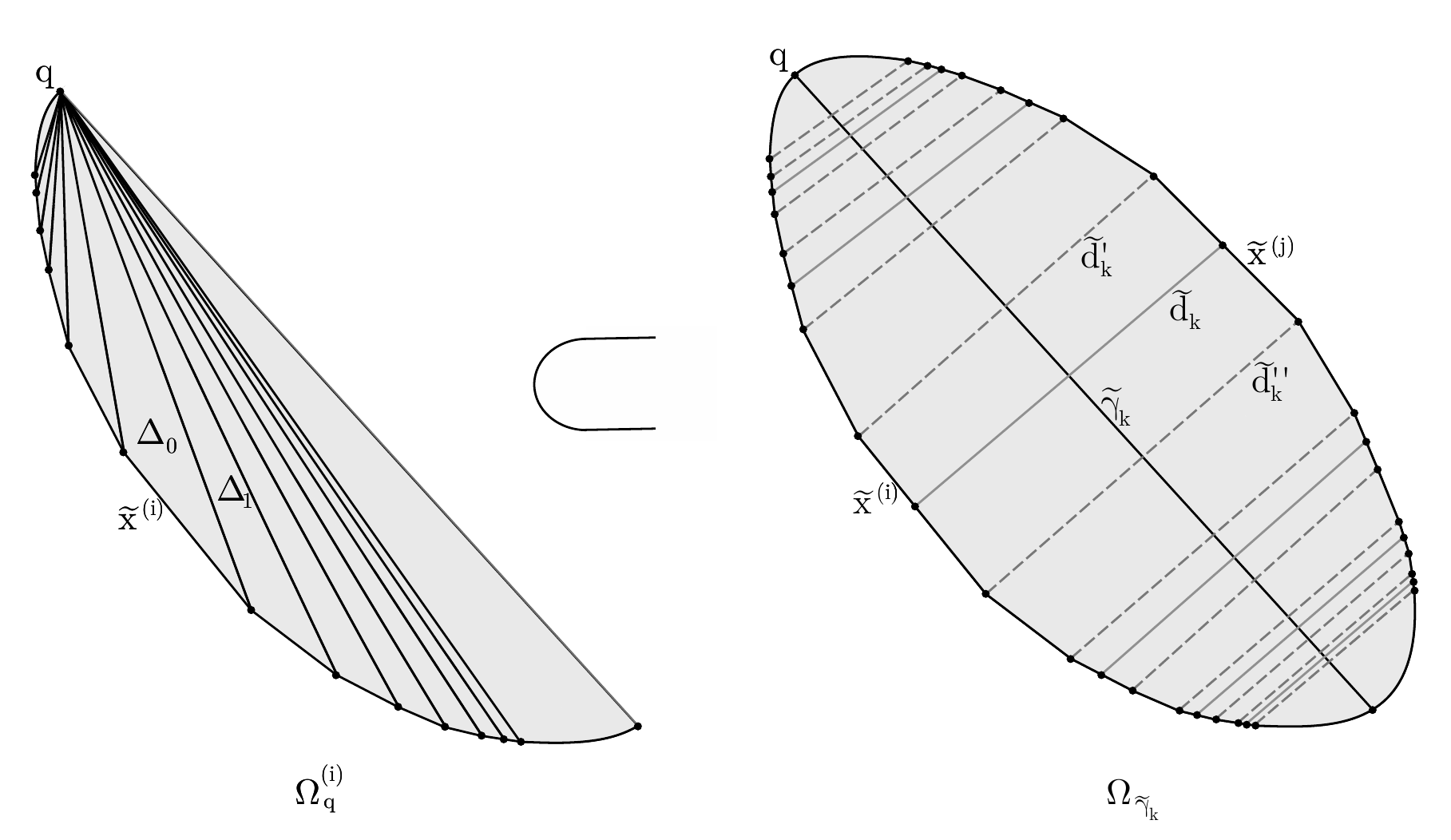}
\caption{Pictorial description of $\Omega_q^{(i)}$, $\Omega_{\widetilde{\gamma}_k}$, $\dtd_k$, $\dtd_k'$, $\dtd_k''$.}
\label{constructions}
\end{figure}

\item ($d_k$, $d_k'$ and $d_k''$) Choose $\gamma_k\in\Pmc$ that is not a boundary component of $\Mpc$, and let $\Ppc^{(i)},\Ppc^{(j)}\subset\Mpc$ be the two pairs of pants given by $\Pmc$ that contain $\gamma_k$. In $\Ppc^{(i)}$, exactly one of the three leaves of $\Vmc^{(i)}$, call it $x^{(i)}$, does not accumulate to $\gamma_k$. Similarly, let $x^{(j)}$ be the unique leaf of $\Vmc^{(j)}$ that does not accumulate to $\gamma_k$. Then, among the open line segments in $\Ppc^{(i)}\cup \Ppc^{(j)}$ that intersect $\gamma_k$ and have one endpoint in $x^{(i)}$ and one endpoint in  $x^{(j)}$, choose a length minimizing one and call it $d_k$ (see Figure \ref{dk}). Let $\dtd_k$ be a lift of $d_k$ to $\Omega_\Mpc$, and it is clear that the endpoints of $\dtd_k$ lie on lifts $\xtd^{(i)}$ and $\xtd^{(j)}$ of $x^{(i)}$ and $x^{(j)}$ respectively. Let $\dtd'_k$ and $\dtd''_k$ be the line segments in $\Omega_{\Ppc^{(i)}}\cup\Omega_{\Ppc^{(j)}}$ joining an endpoint of $\xtd^{(i)}$ to an endpoint of $\xtd^{(j)}$ such that $\dtd'_k\cap\dtd_k=\emptyset=\dtd''_k\cap\dtd_k$. (See Figure \ref{constructions}.) Then let $d'_k:=\Pi(\dtd_k')$ and $d''_k:=\Pi(\dtd''_k)$, where $\Pi:\Omega_\Mpc\to \Mpc$ is the covering map. (See Figure \ref{dkdk}.) 

\item ($\Gamma_{\widetilde{\gamma}_k}$ and $\Omega_{\widetilde{\gamma}_k}$) Choose $\gamma_k\in\Pmc$ that is not a boundary component of $\Mpc$ and let $\widetilde{\gamma}_k$ be a lift of $\gamma_k$ to $\Omega_\Mpc$. Let $\Gamma_{\widetilde{\gamma}_k}$ be the stabilizer in $\pi_1(M)$ of $\widetilde{\gamma}_k$, and let $x^{(i)}$, $x^{(j)}$ be the lines in $\Umc_\Mpc$ that contain the endpoints of $d_k$. Choose a lift $\dtd_k$ of $d_k$ that intersects $\widetilde{\gamma}_k$ and let $\xtd^{(i)}$, $\xtd^{(j)}$ be the lifts of $x^{(i)}$, $x^{(j)}$ that contain the endpoints of $\dtd_k$. Then define $\Omega_{\widetilde{\gamma}_k}$ (see Figure \ref{constructions}) to be the open convex subset of $\Omega_\Mpc$ bounded by the $\Gamma_{\widetilde{\gamma}_k}$ translates of $\xtd^{(i)}$ and $\xtd^{(j)}$. 
\end{itemize}

\begin{figure}
\includegraphics[scale=0.25]{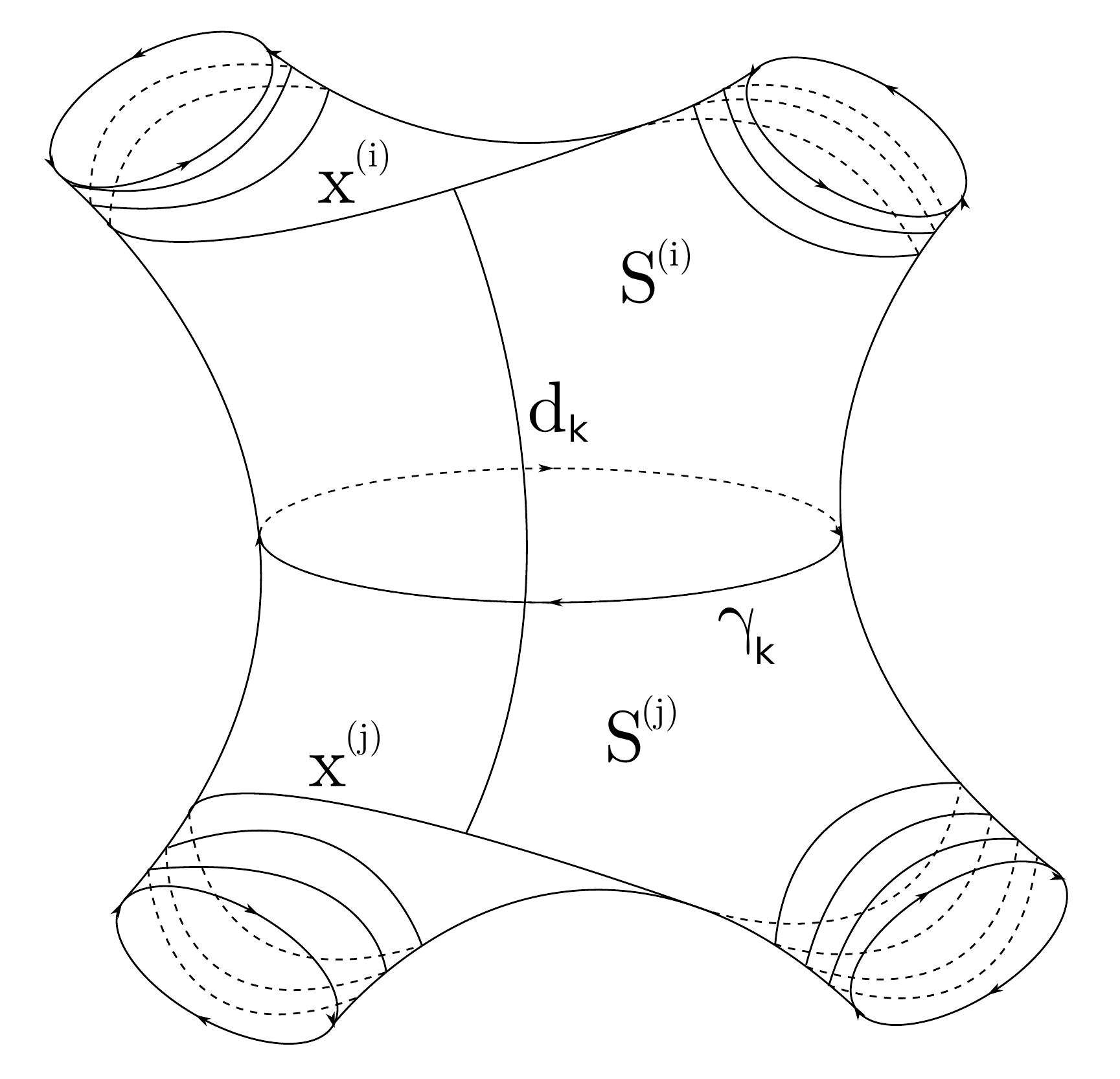}
\caption{Pictorial description of $d_k$.}
\label{dk}
\end{figure}

\begin{figure}
\includegraphics[scale=0.25]{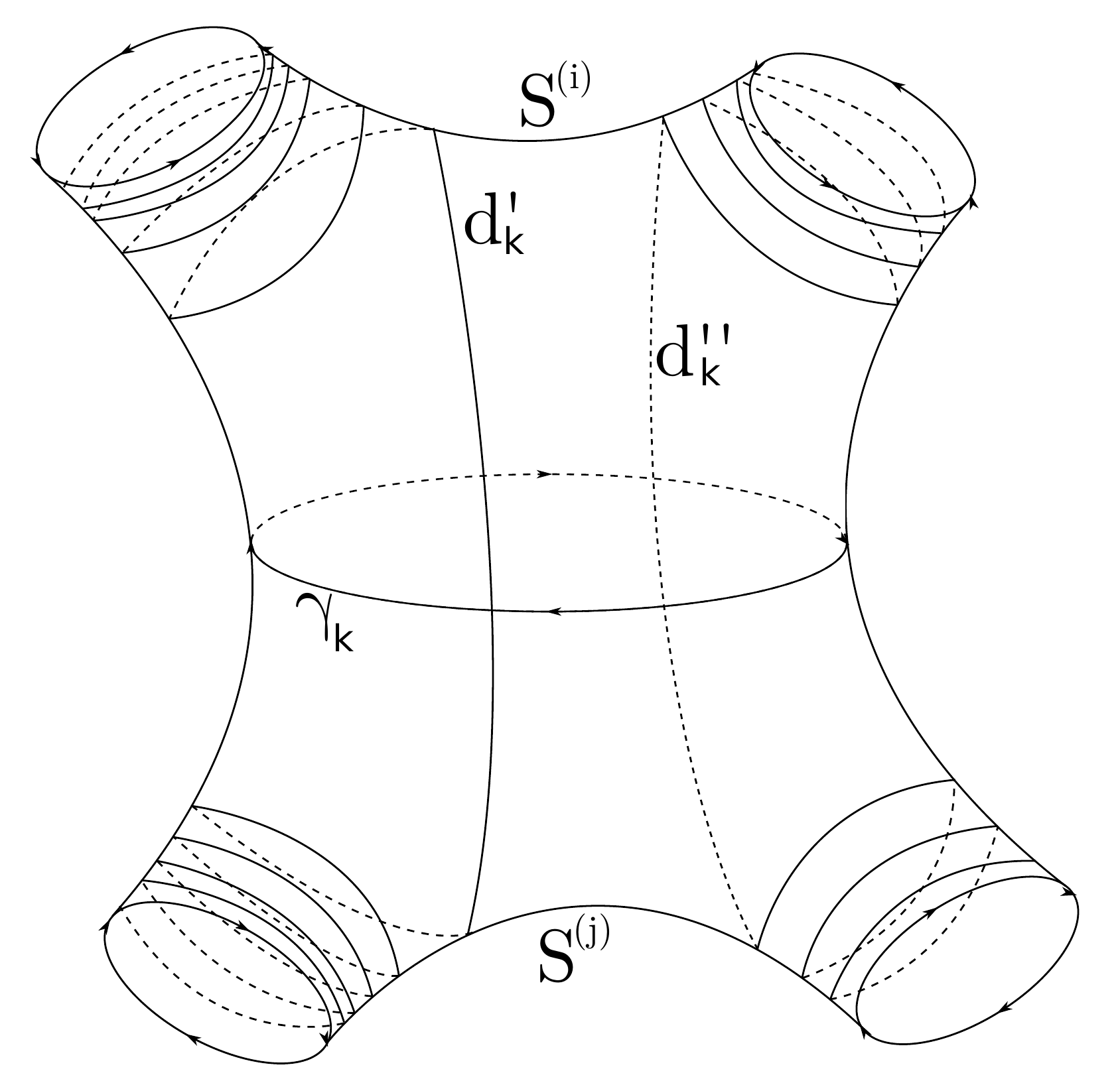}
\caption{Pictorial description of $d_k'$ and $d_k''$.}
\label{dkdk}
\end{figure}

\begin{remark}\label{thetwogammas}
Let $\Ppc^{(i)}$ and $\Ppc^{(j)}$ be the two pairs of pants given by $\Pmc$ that contain $\gamma_k$, and choose a lift $\widetilde{\gamma}_k$ of $\gamma_k$ in $\Omega_\Mpc$. By conjugating, we can assume that $\widetilde{\gamma}_k$ is the axis of both $h_{\Ppc^{(i)}}(X)$ and $h_{\Ppc^{(j)}}(Y)$ for some $X,Y\in\{A_0,B_0,C_0\}$. Denote by $q^{(i)}$ and $q^{(j)}$ the repelling fixed points of $h_{\Ppc^{(i)}}(X)$ and $h_{\Ppc^{(j)}}(Y)$ respectively. Then $\Gamma_{q^{(i)}}=\Gamma_{q^{(j)}}=\Gamma_{\widetilde{\gamma}_k}$ and $\Omega_{q^{(i)}}^{(i)}\cup\Omega_{q^{(j)}}^{(j)}\cup\widetilde{\gamma}_k=\Omega_{\widetilde{\gamma}_k}$.
\end{remark}

\begin{remark}\label{dkdk'dk''}
Observe that $\Omega_{\widetilde{\gamma}_k}$ contains $\widetilde{\gamma}_k$ and all the $\Gamma_{\widetilde{\gamma}_k}$ translates of $d_k$, $d_k'$ and $d_k''$. Moreover, if $X$ is a generator of $\Gamma_{\widetilde{\gamma}_k}$, then for any $j\in\Zbbb$, the subset of $\Omega_{\widetilde{\gamma}_k}$ that is bounded between $X^j\cdot d_k$ and $X^{j+1}\cdot d_k$ (resp. $X^j\cdot d'_k$ and $X^{j+1}\cdot d'_k$; $X^j\cdot d''_k$ and $X^{j+1}\cdot d''_k$) contains exactly one $\Gamma_{\widetilde{\gamma}_k}$ translate of $d'_k$ and one $\Gamma_{\widetilde{\gamma}_k}$ translate of $d''_k$ (resp. $d_k$ and $d''_k$; $d_k$ and $d'_k$).
\end{remark}

\begin{remark}\label{gammaonomega}
Observe that $\Delta_0\cup\Delta_1$ is a fundamental domain of the action of $\pi_1(\Ppc^{(i)})$ on $\Omega_{\Ppc^{(i)}}$. In particular, for any $X\in\pi_1(M)\setminus\Gamma_q$, we have $(X\cdot\Omega_q^{(i)})\cap\Omega_q^{(i)}=\emptyset$. Also we can easily describe the action of $\Gamma_q$ on $\Omega_q^{(i)}$. Let $\Delta_i'$ and $\Delta_i''$ be the two triangles adjacent to $\Delta_{1-i}$ that have $q$ as a vertex and let $\Delta_{1-i}'$ and $\Delta_{1-i}''$ be the other triangles that have $q$ as a vertex and are adjacent to $\Delta_i'$ and $\Delta_i''$ respectively. If $Y$, $Y^{-1}$ are the two possible generators of $\Gamma_q$, then $\{Y\cdot\Delta_{1-i},Y^{-1}\cdot\Delta_{1-i}\}=\{\Delta_{1-i}',\Delta_{1-i}''\}$.
\end{remark}

Next, we will describe how to decompose any typical oriented closed geodesic $\eta$ in $\Mpc$ using the ideal triangulation $\Umc_\Mpc$. Choose a parameterization for $\eta$ and define the following. (Here, we abuse notation by denoting the image of a curve by the curve itself.)
\begin{itemize}
\item Let $\displaystyle I'_\eta:=\eta^{-1}(\Umc_\Mpc\cap\eta)\subset S^1$. These are the points in $S^1$ that are mapped via $\eta$ to points of intersection of the curve $\eta$ with $\Umc_\Mpc$.
\item For any $q\in I'_\eta$, define $u_q$ to be the unique line in $\Umc_\Mpc$ that contains $\eta(q)$. Also, define $\uora_q$ to be the line $u_q$ equipped with the orientation so that $\eta$ passes from the left to the right of $\uora_q$ at $q$.
\item Let $\displaystyle D_\eta:=\eta^{-1}(\Pmc\cap\eta)\subset I'_\eta$. These are the points in $S^1$ that are mapped via $\eta$ to points where $\eta$ intersects the curves in $\Pmc$.
\end{itemize}

The orientation of $\eta$ induces a cyclic order on $I'_\eta$. Note that the only elements in $I'_\eta$ that do not have a successor or a predecessor in this cyclic order are the points in $D_\eta$. Thus, we can also define the following.
\begin{itemize}
\item Let $\suc=\suc_\eta:I'_\eta\setminus D_\eta\to I'_\eta\setminus D_\eta$  be the bijection which takes any point in $I'_\eta\setminus D_\eta$ to its successor. 

\item Define $I_\eta:=I'_\eta\setminus(D_\eta\cup E^+_\eta\cup E^-_\eta)$ of $S^1$, where 
\[\begin{array}{rcl}
E^+_\eta&:=&\{q\in I'_\eta:u_{\suc^k(q)}\text{ share a common vertex }\forall k\in\Zbbb_{\geq-2}\},\\
E^-_\eta&:=&\{q\in I'_\eta:u_{\suc^k(q)}\text{ share a common vertex }\forall k\in\Zbbb_{\leq2}\}.
\end{array}\]
\end{itemize}

The compactness of $\eta$ implies that $I_\eta$ is a finite set. Also, the cyclic order on $I_\eta'$ induces a cyclic order on $I_\eta$. 

\begin{notation}
For any pair of distinct points $p,q$ in $S^1$, define $(p,q)$ (resp. [p,q]) to be the open (resp. closed) subinterval of $S^1$ in the clockwise direction from $p$ to $q$.
\end{notation}

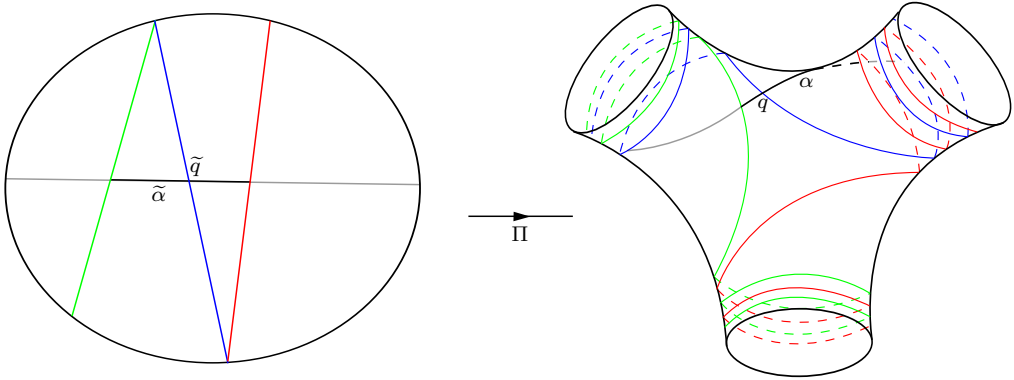
\begin{figure}
\includegraphics[scale=0.7]{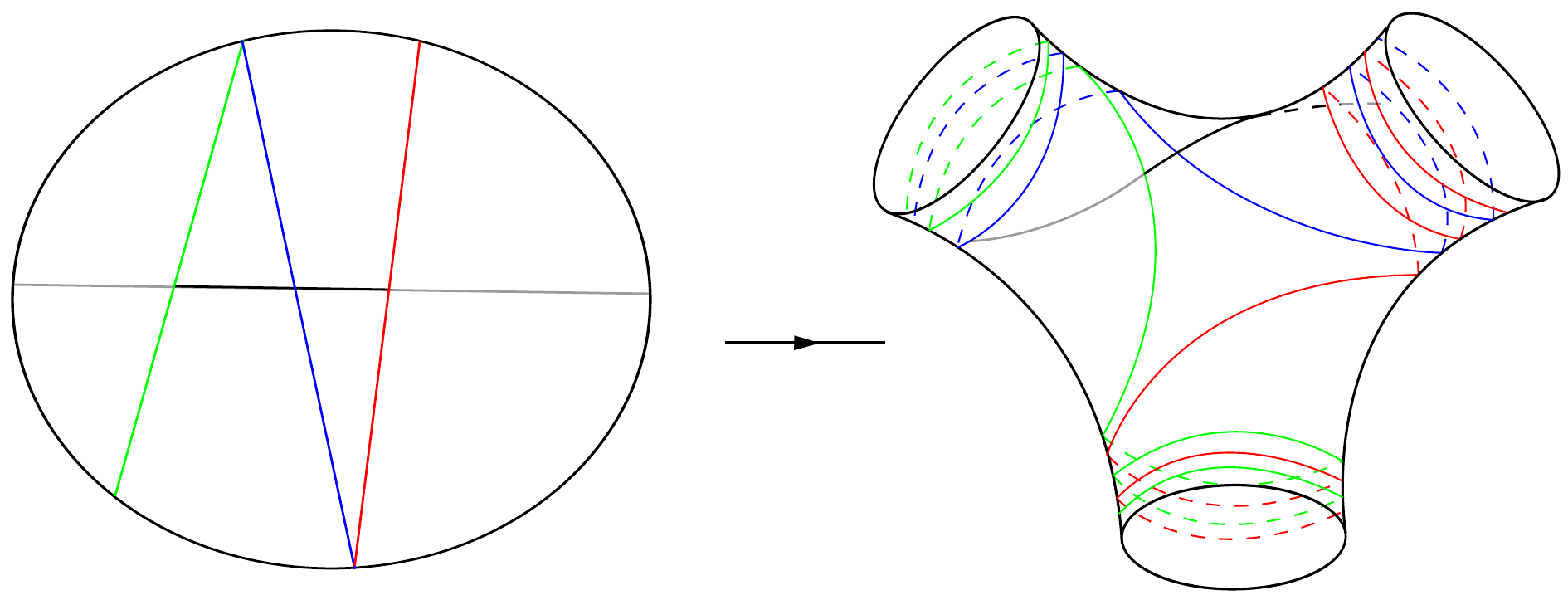}
\put (-377, 102){\makebox[0.7\textwidth][r]{$q$ }}
\put (-360, 110){\makebox[0.7\textwidth][r]{$\alpha$ }}
\put (-590, 78){\makebox[0.7\textwidth][r]{$\qtd$ }}
\put (-603, 67){\makebox[0.7\textwidth][r]{$\widetilde{\alpha}$ }}
\put (-467, 52){\makebox[0.7\textwidth][r]{$\Pi$ }}
\caption{$q$ is a crossing point and $\alpha$ is a crossing segment}
\label{crossing segment}
\end{figure}

\begin{definition}
A point $q$ in $I_\eta$ is called a \emph{crossing point} if the three edges $u_{\suc^{-1}(q)}$, $u_q$ and $u_{\suc(q)}$ do not share any common vertices. The closed subsegment $\alpha:=\eta|_{[\suc^{-1}(q),\suc(q)]}$ is called a \emph{crossing segment} corresponding to $u_q$ (see Figure \ref{crossing segment}), and the triple 
\[(\uora_{\suc^{-1}(q)},\uora_q,\uora_{\suc(q)})\]
is called a \emph{crossing triple} for $\Mpc$. 
\end{definition}

One should think of the crossing segments as line segments that ``go between" collar neighborhoods of the simple closed curves in $\Pmc$. Since $\Pmc$ decomposes $\Mpc$ into $2g-2+n$ pairs of pants, and each $\Ppc^{(i)}=f_i^\#(\Mpc)$ has twelve crossing triples, we have the following lemma.

\begin{lem}\label{crossingtriple} 
Let $M$ be a closed surface of genus $g$ with $n$ disjoint open discs removed. Then any $\Mpc\in\Cmf(M)$ has $24g-24+12n$ crossing triples. 
\end{lem}

Since the crossing points for a fixed parameterized closed curve $\eta$ lie in $S^1$, they have a natural cyclic order. So, by choosing a crossing point $p_1$ of $\eta$, we can enumerate the other crossing points of $\eta$ according to the cyclic order. Let $\{p_{m+1}=p_1,p_2,p_3,\dots,p_m\}$ be the set of crossing points of $\eta$ enumerated as described. 

\begin{definition}
The pair $(p_i,p_{i+1})$ is called a \emph{pants changing pair} if for all $k\in\Zbbb_{\geq 1}$, $suc^k(p_i)\neq p_{i+1}$. The \emph{pants changing segment} corresponding to $(p_i,p_{i+1})$ is the closed subsegment $\eta|_{[\suc^{-1}(p_i),\suc(p_{i+1})]}$. (See Figure \ref{pants changing segment}.)
\end{definition}

\begin{definition}
The pair $(p_i, p_{i+1})$ is called a \emph{looping pair} if there is some $k\in\Zbbb_{\geq 1}$ such that $\suc^k(p_i)=p_{i+1}$. The \emph{looping segment} corresponding to $(p_i, p_{i+1})$ is the closed subsegment $\eta|_{[\suc^{-1}(p_i),\suc(p_{i+1})]}$. (See Figure \ref{looping segment}.)
\end{definition}

One should think of the pants changing segments and looping segments as line segments that ``wind around" the collar neighborhoods of the simple closed curves in $\Pmc$. The pants changing segments wind around while moving between pairs of pants, while the looping segments wind around while staying in the same pair of pants.

\begin{figure}
\includegraphics[scale=0.65]{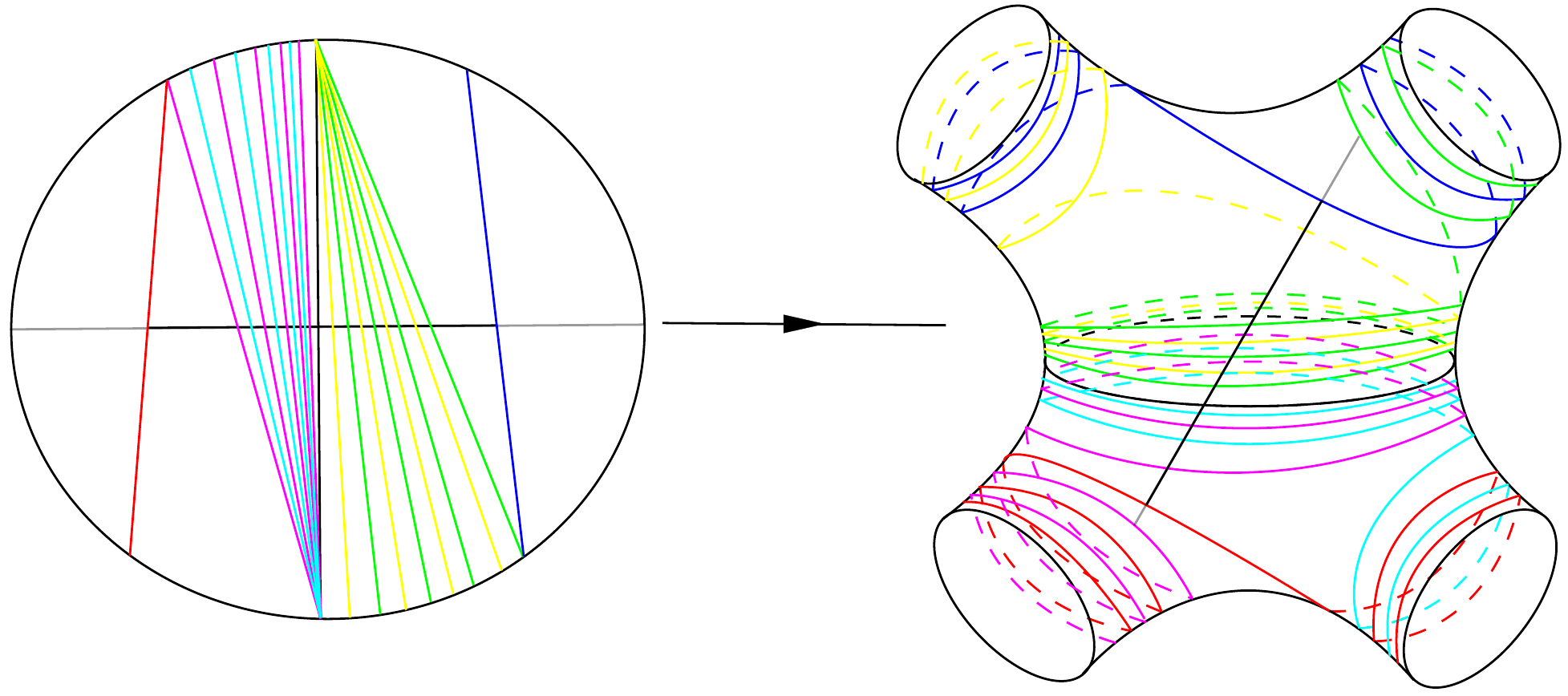}
\put (-371, 45){\makebox[0.7\textwidth][r]{$p_i$ }}
\put (-322, 117){\makebox[0.7\textwidth][r]{$p_{i+1}$ }}
\put (-351, 100){\makebox[0.7\textwidth][r]{$\beta$ }}
\put (-617, 89){\makebox[0.7\textwidth][r]{$\ptd_i$ }}
\put (-517, 90){\makebox[0.7\textwidth][r]{$\ptd_{i+1}$ }}
\put (-605, 77){\makebox[0.7\textwidth][r]{$\widetilde{\beta}$ }}
\put (-460, 77){\makebox[0.7\textwidth][r]{$\Pi$ }}
\caption{$(p_i,p_{i+1})$ is a pants changing pair, $\beta$ is a pants changing segment}
\label{pants changing segment}
\end{figure}

\begin{figure}
\includegraphics[scale=0.7]{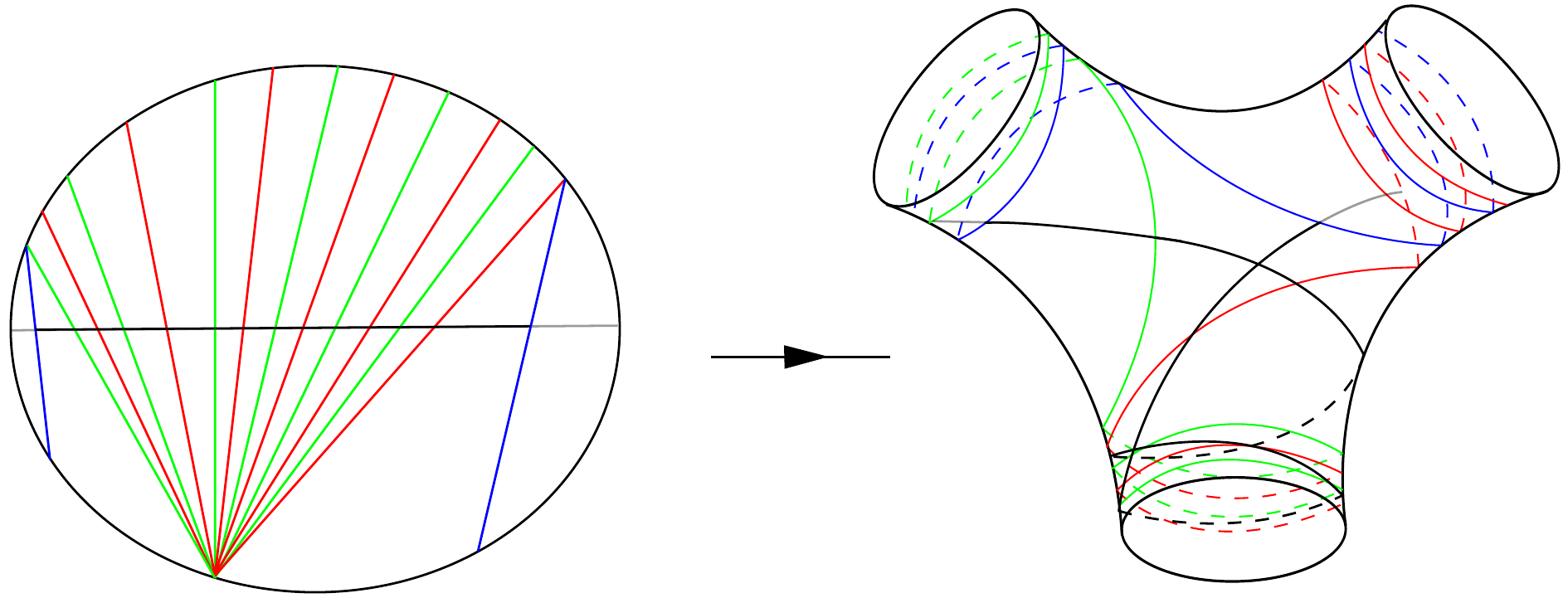}
\put (-420, 87){\makebox[0.7\textwidth][r]{$p_i$ }}
\put (-337, 97){\makebox[0.7\textwidth][r]{$p_{i+1}$ }}
\put (-340, 60){\makebox[0.7\textwidth][r]{$\beta$ }}
\put (-648, 58){\makebox[0.7\textwidth][r]{$\ptd_i$ }}
\put (-520, 58){\makebox[0.7\textwidth][r]{$\ptd_{i+1}$ }}
\put (-550, 56){\makebox[0.7\textwidth][r]{$\widetilde{\beta}$ }}
\put (-468, 49){\makebox[0.7\textwidth][r]{$\Pi$ }}
\caption{$(p_i,p_{i+1})$ is a looping pair, $\beta$ is a looping segment}
\label{looping segment}
\end{figure}
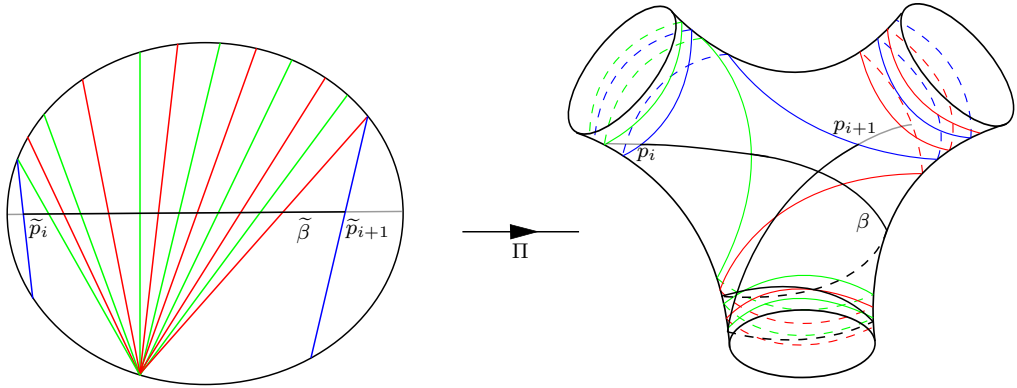

On $S^1$, one can also visualize the pre-images (under $\eta$) of the crossing segments as neighborhoods of the crossing points, while the pre-images of the pants changing segments and looping segments contain the intervals between subsequent pairs of crossing points. Note that even though we chose a parameterization of $\eta$ to make the above definitions, the pants changing, looping, crossing segments and the crossing triples are in fact independent of the choice of parameterization. 

If $\beta$ is a pants changing segment, then $\beta$ intersects exactly one $\gamma_k$ in $\Pmc$. Define
\begin{equation}\label{Fi1}
F_\beta:=\beta^{-1}(\beta\cap(d'_k\cup d''_k)).
\end{equation}
For any $q\in F_\beta$, let $v_q$ be the line in $\{d'_k,d''_k\}$ that contains $\eta(q)$ and let $\vora_q$ be the line $v_q$ equipped with the orientation so that $\eta$ (equipped with its orientation) passes from the left to the right of $\vora_q$ at $q$. The natural order on $F_\beta$ induced by the orientation of $\eta$ allows us to define the finite sequence $\overrightarrow{Y}_\beta=\{\vora_q:q\in F_\beta\}$.

If $\beta$ is a looping segment, define 
\begin{equation}\label{Fi2}
F_\beta:=\beta^{-1}((\beta\setminus\partial\beta)\cap\Upc_\Mpc).
\end{equation}
Note that the edges in $\{u_p:p\in F_\beta\}$ share a common vertex, or equivalently, accumulate to a common boundary component $\gamma$ of $S$. Moreover, the $\uora_p$ are oriented so that either all of them point towards $\gamma$ or all of them point away from $\gamma$. Since $F_\beta$ has a natural order induced by the orientation of $\eta$, we can define the finite sequence $\overrightarrow{Y}_\beta:=\{\uora_p:p\in F_\beta\}$.

In either case, $\overrightarrow{Y}_\beta$ is designed to keep track of the ``amount of twisting" $\eta$ does while moving between crossing points. Let $Y_\beta$ be the finite sequence that is obtained from $\overrightarrow{Y}_\beta$ by forgetting the orientation of the terms of $\overrightarrow{Y}_\beta$. Again, observe $\overrightarrow{Y}_\beta$ and $Y_\beta$ are independent of the choice of parameterization for $\eta$. 

\subsection{Combinatorial descriptions of closed geodesics.}\label{combinatorialdescriptionsclosedgeodesics}
The decomposition described in Section \ref{setup} begs the following question. If we are given the cyclic sequence of crossing triples for some $\eta\in\Tpc_\Mpc$ (recall that $\Tpc_\Mpc$ is the set of typical closed geodesics in $\Mpc$), together with the looping or pants changing segments between every pair of crossing points for $\eta$, can we recover $\eta$? The answer to this question, as we will see later, is yes. However, the way this question is currently posed is a little awkward because while the cyclic sequence of crossing triples for $\eta$ is a combinatorial object (there are finitely many possibilities for each entry of this cyclic sequence by Lemma \ref{crossingtriple}), the cyclic sequence of looping or pants changing segments between every pair of crossing points for $\eta$ is not. Hence, we want to replace the latter with something more combinatorial in nature. 

It turns out that there are two rather natural ways to do so. These give two different combinatorial descriptions of $\eta$, which we call $\phi(\eta)$ and $\psi(\eta)$ respectively. The goal of this subsection is thus to formally define $\phi(\eta)$ and $\psi(\eta)$, and answer the above question about $\phi(\eta)$ and $\psi(\eta)$, i.e. can we recover $\eta$ from $\phi(\eta)$ or $\psi(\eta)$, and if not, how much information do we lose by describing $\eta$ using $\phi(\eta)$ or $\psi(\eta)$? We shall start with $\phi(\eta)$.

\begin{definition}
Let $\eta\in\Tpc_\Mpc$ and let $\{p_{m+1}=p_1,\dots,p_m\}$ be the cyclic sequence of crossing points along $\eta$. Also, let $\beta_i$ be the looping or pants changing segment associated to the pair $(p_i,p_{i+1})$ and let $\overrightarrow{Y}_i=\overrightarrow{Y}_{\beta_i}$. Then define \emph{$\phi(\eta)$} be the cyclic sequence $\{\phi_{m+1}=\phi_1,\dots,\phi_m\}$, where each $\phi_i$ is the tuple $(\uora_{\suc^{-1}(p_i)},\uora_{p_i},\uora_{\suc(p_i)},\overrightarrow{Y}_i)$, 
\end{definition}

The next proposition gives a positive answer to the question asked above, i.e. we can recover $\eta$ from $\phi(\eta)$.

\begin{prop}\label{phietaetaequivalence}
If $\eta,\eta'\in\Tpc_\Mpc$ are such that $\phi(\eta)=\phi(\eta')$, then $\eta=\eta'$.
\end{prop}

Before we begin the formal proof which is rather technical, let us first see why this proposition is morally true. If we look in $\Omega_\Mpc$, the condition that $\phi(\eta)=\phi(\eta')$ should imply that the lifts of $\eta$ and $\eta'$ are two lines that pass through the same triangles of the ideal triangulation $\widetilde{\Umc}_\Mpc$. This means that their endpoints in $\partial\Omega_\Mpc$ are the same, so they have to be equal.

\begin{proof}
The general strategy is to show that if $\phi(\eta)=\phi(\eta')$, then $\eta$ is homotopic to $\eta'$. This allows us to conclude that $\eta=\eta'$ by (2) of Proposition \ref{properties}. 

Enumerate the crossing points of $\eta$ and $\eta'$ by $\{p_{m+1}=p_1,\dots,p_m\}$ and $\{p_{m+1}'=p_1',\dots,p_m'\}$ respectively, so that if $\{\phi_1,\dots,\phi_m\}$ and $\{\phi_1',\dots,\phi_m'\}$ are the induced enumeration of the terms in $\phi(\eta)$ and $\phi(\eta')$ respectively, then $\phi_i=\phi_i'$ for all $i\in\{1,\dots,m\}$. Let $\alpha_i$ and $\alpha_i'$ be the crossing segments corresponding to $p_i$ and $p_i'$ respectively. Also, let $\beta_i$ and $\beta_i'$ be the pants changing segment or looping segment corresponding to the pair $(p_i,p_{i+1})$ and $(p_i',p_{i+1}')$ respectively. 

Since $u_{\suc^{-1}(p_i)}=u_{\suc^{-1}(p_i')}$, $u_{p_i}=u_{p_i'}$ and $u_{\suc(p_i)}=u_{\suc(p_i')}$ for all $i\in\{1,\dots,m\}$, we can define the following. For any $i\in\{1,\dots,m\}$, let 
\[\epsilon_{\suc^{-1}(p_i)},\epsilon_{p_i},\epsilon_{\suc(p_i)}:[0,1]\to \Mpc\]
be the subsegments of $u_{\suc^{-1}(p_i)}$, $u_{p_i}$, $u_{\suc(p_i)}$ respectively such that 
\[\begin{array}{lll}
\epsilon_{\suc^{-1}(p_i)}(0)=\eta(\suc^{-1}(p_i)), &\epsilon_{p_i}(0)=\eta(p_i),&\epsilon_{\suc(p_i)}(0)=\eta(\suc(p_i)),\\
\epsilon_{\suc^{-1}(p_i)}(1)=\eta'(\suc^{-1}(p_i')), & \epsilon_{p_i}(1)=\eta'(p_i'),& \epsilon_{\suc(p_i)}(1)=\eta'(\suc(p_i')).
\end{array}\]

For all $i\in\{1,\dots,m\}$, let $\delta_{i,-}:=\eta|_{[\suc^{-1}(p_i),p_i]}$, $\delta_{i,+}:=\eta|_{[p_i,\suc(p_i)]}$, $\delta_{i,-}':=\eta'|_{[\suc^{-1}(p_i'),p_i']}$ and $\delta_{i,+}':=\eta'|_{[p_i',\suc(p_i')]}$. Note that the line segments $\delta_{i,-}$,$\delta_{i,-}'$, $\epsilon_{\suc^{-1}(p_i)}$ and $\epsilon_{p_i}$ all lie in a triangle that is the closure of a connected component of $\displaystyle\Mpc\setminus\Upc_\Mpc$. The same is also true for the line segments  $\delta_{i,+}$,$\delta_{i,+}'$, $\epsilon_{p_i}$ and $\epsilon_{\suc(p_i)}$. Thus, $\delta_{i,-}$ is homotopic relative endpoints to $\epsilon_{\suc^{-1}(p_i)}\cdot\delta_{i,-}'\cdot\epsilon_{p_i}^{-1}$ and $\delta_{i,+}$ is homotopic relative endpoints to $\epsilon_{p_i}\cdot\delta_{i,+}'\cdot\epsilon_{\suc(p_i)}^{-1}$. Since $\alpha_i=\delta_{i,-}\cdot\delta_{i,+}$ and $\alpha_i'=\delta_{i,-}'\cdot\delta_{i,+}'$, we have that $\alpha_i$ is homotopic relative endpoints to $\epsilon_{\suc^{-1}(p_i)}\cdot\alpha_i'\cdot\epsilon_{\suc(p_i)}^{-1}$. 

Also, since $Y_i=Y_i'$, we know that $\beta_i$ is a looping segment (resp. pants changing segment) if and only if $\beta_i'$ is a looping segment (resp. pants changing segment). Define $F_i$ as in (\ref{Fi1}) if $\beta_i$ is a pants changing segment or as in (\ref{Fi2}) if $\beta_i$ is a looping segment, and let $F_i'$ be the corresponding object for $\beta_i'$. Enumerate $F_i=\{q_1,\dots,q_{l_i}\}$ and $F_i'=\{q_1',\dots,q_{l_i}'\}$ according to the order induced on $F_i$ and $F_i'$ by the orientations of $\eta$ and $\eta'$ respectively. Define $q_0:=\suc^{-1}(p_i)$, $q_{l_i+1}:=\suc(p_{i+1})$, $q_0':=\suc^{-1}(p_i')$, $q_{l_i+1}':=\suc(p_{i+1}')$, and if $\beta_i$ is a pants changing segment, define $v_{q_0}:=u_{\suc^{-1}(p_i)}$, $v_{q_{l_i+1}}:=u_{\suc(p_{i+1})}$, $v_{q_0'}:=u_{\suc^{-1}(p_i')}$, $v_{q_{l_i+1}'}:=u_{\suc(p_{i+1}')}$. Note that for all $j\in\{0,\dots,l_i+1\}$, $u_{q_j}=u_{q_j'}$ if $\beta_i$ is a looping segment and $v_{q_j}=v_{q_j'}$ if $\beta_i$ is a pants changing segment. Thus, for $j\in\{0,\dots,l_i+1\}$, we can define $\epsilon_{q_j}:[0,1]\to N$ to be the subsegment of $u_{q_j}$ or $v_{q_j}$ such that $\epsilon_{q_j}(0)=\eta(q_j)$ and $\epsilon_{q_j}(1)=\eta'(q_j')$. Also, for $j\in\{0,\dots,l_i\}$, define $\delta_j:=\eta|_{[q_j,q_{j+1}]}$ and $\delta_j':=\eta'|_{[q_j',q_{j+1}']}$. 

If $\beta_i$ is a looping segment, then for the same reasons as above, $\delta_j$ is homotopic relative endpoints to $\epsilon_{q_j}\cdot\delta_j'\cdot\epsilon_{q_{j+1}}^{-1}$. In the case that $\beta_i$ is a pants changing segment, let $\gamma_k$ be the unique closed geodesic in $\Pmc$ that intersects $\beta_i$. Then $\gamma_k$ is also the unique closed geodesic that intersects $\beta_i'$. Let $\Ppc^{(1)},\Ppc^{(2)}\subset\Mpc$ be the two pairs of pants given by the pants decomposition $\Pmc$ that contain $\gamma_k$, and let $x^{(1)}$, $x^{(2)}$ be the leaves in $\Vmc^{(1)}$, $\Vmc^{(2)}$ respectively that do not accumulate to $\gamma_k$. For simplicity, we will assume that $\Ppc^{(1)}\neq \Ppc^{(2)}$; the proof in the case when $\Ppc^{(1)}=\Ppc^{(2)}$ is similar. For all $j\in\{0,\dots,l_i\}$, note that $\delta_j$,$\delta_j'$, $\epsilon_{q_j}$ and $\epsilon_{q_{j+1}}$ all lie in the closure of one of the two quadrilaterals in $\Mpc$ bounded by the lines $d_k'$, $d_k''$, $x^{(1)}$, $x^{(2)}$. This means that $\delta_j$ is homotopic relative endpoints to $\epsilon_{q_j}\cdot\delta_j'\cdot\epsilon_{q_{j+1}}^{-1}$. In either case, since $\beta_i$ is the concatenation of the $\delta_j$ and $\beta_i'$ is the concatenation of the $\delta_j'$ in the obvious order, we see that $\beta_i$ is homotopic relative endpoints to $\epsilon_{q_0}\cdot\beta_i'\cdot\epsilon_{q_{l_i+1}}^{-1}=\epsilon_{\suc^{-1}(p_i)}\cdot\beta_i'\cdot\epsilon_{\suc(p_{i+1})}^{-1}$.

Finally, note that $\eta$ is the cyclic concatenation $\alpha_1^{-1}\cdot\beta_1\cdot\alpha_2^{-1}\cdot\beta_2\cdot\dots\cdot\alpha_m^{-1}\cdot\beta_m$ and $\eta'$ is the cyclic concatenation $\alpha_1'^{-1}\cdot\beta_1'\cdot\alpha_2'^{-1}\cdot\beta_2'\cdot\dots\cdot\alpha_m'^{-1}\cdot\beta_m'$. By what we proved above, the cyclic concatenation $\alpha_1'^{-1}\cdot\beta_1'\cdot\alpha_2'^{-1}\cdot\beta_2'\cdot\dots\cdot\alpha_m'^{-1}\cdot\beta_m'$ is homotopic to the cyclic concatenation $\alpha_1^{-1}\cdot\beta_1\cdot\alpha_2^{-1}\cdot\beta_2\cdot\dots\cdot\alpha_m^{-1}\cdot\beta_m$, so $\eta$ is homotopic to $\eta'$. Since $\eta$ and $\eta'$ are closed geodesics, we deduce from (2) of Proposition \ref{properties} that $\eta=\eta'$.
\end{proof}

For any $\eta\in\Tpc_\Mpc$, let $\Xmf_\eta$, $\Ymf_\eta$ and $\Zmf_\eta$ be the set of all crossing segments, pants changing segments and looping segments of $\eta$ respectively. It is easy to see that the typicalness of $\eta$ implies $\Xmf_\eta$ and $\Ymf_\eta\cup\Zmf_\eta$ are nonempty and finite. From $\eta$, we can also obtain a second piece of combinatorial data $\psi(\eta)$, which we will now define. Let $\#:\Ymf\cup\Zmf\to\Zbbb_{\geq 0}$ be the function defined as follows. If $\beta$ is a pants changing segment of $\eta$ intersecting $\gamma_k$,
\[\#(\beta):=\left\{\begin{array}{ll}
0&\text{if }d_k\text{ is not transverse to }\beta\text{ or }|\beta\cap d_k|=0\\ 
|\beta^{-1}(\beta\cap d_k)|-1&\text{otherwise}\\
\end{array}\right.,\]
and if $\beta$ is a looping segment of $\eta$, 
\[\#(\beta):=\text{the number of self intersections of }\beta.\]

\begin{definition}
Let $\eta\in\Tpc_\Mpc$ and let $\{p_{m+1}=p_1,\dots,p_m\}$ be the cyclic sequence of crossing points along $\eta$. Also, let $\beta_i$ be the looping or pants changing segment associated to the pair $(p_i,p_{i+1})$.Then define \emph{$\psi(\eta)$} be the cyclic sequence $\{\chi_1,\dots,\chi_m\}$, where each $\chi_i$ is the tuple $(\uora_{\suc^{-1}(p_i)},\uora_{p_i},\uora_{\suc(p_i)},\#(\beta_i))$ and the cyclic order here is induced by the cyclic order on the set of crossing points $\{p_1,\dots,p_m\}$.
\end{definition}

Unlike $\phi(\eta)$, $\psi(\eta)$ does not give a complete combinatorial description of $\eta$, i.e. there are distinct typical closed geodesics $\eta$ and $\eta'$ in $\Mpc$ such that $\psi(\eta)=\psi(\eta')$. For example, if $\eta\in\Tpc_\Mpc$ has a pants changing segment $\beta$ intersecting $\gamma_k$ such that $|\beta\cap d_k|=1$, then we can do a Dehn twist about $\gamma_k$ to obtain a new curve $\eta'\in\Tpc_\Mpc$ such that the corresponding pants changing segment $\beta'$ of $\eta'$ does not intersect $d_k$. Note that in this case, $\#(\beta)=0=\#(\beta')$, so $\psi(\eta)=\psi(\eta')$, but $\eta\neq\eta'$.

However, $\psi(\eta)$ is useful because we can obtain a lower bound on the length of $\eta$ in terms of the Goldman parameters for $\Mpc$ and other data which depends only on $\psi(\eta)$. We will demonstrate how to do this in Section \ref{lowerboundlengths}. 

In the remainder of this subsection, we will show that even though we cannot recover $\eta$ from $\psi(\eta)$, describing $\eta$ using $\psi(\eta)$ only loses us a ``bounded amount of information" in the following sense. Let $\Fpc_\Mpc:=\{\phi(\eta):\eta\in\Tpc_\Mpc\}$ and $\Gpc_\Mpc:=\{\psi(\eta):\eta\in\Tpc_\Mpc\}$. Then we have the following two maps:
\[\phi:\Tpc_\Mpc\to\Fpc_\Mpc\text{ given by }\phi:\eta\mapsto\phi(\eta),\]
\[\psi:\Tpc_\Mpc\to\Gpc_\Mpc\text{ given by }\psi:\eta\mapsto\psi(\eta).\] 
As mentioned above, $\phi$ is a bijection but $\psi$ is not. However, we can obtain an upper bound on the size $\psi^{-1}(\psi(\eta))$ which depends only on the length of the cyclic sequence $\psi(\eta)$. We will now construct this upper bound.

\begin{lem}\label{pantschanging}
Let $\beta$ be a pants changing segment of $\eta\in\Tpc_\Mpc$, and let $\gamma_k$ be the unique closed geodesic in $\Pmc$ that intersects $\beta$. Then 
\[\Big|\big|\beta^{-1}(\beta\cap d'_k)\big|-\big|\beta^{-1}(\beta\cap d''_k)\big|\Big|\leq 1,\] 
and
\[\Big|\#(\beta)-\big|\beta^{-1}(\beta\cap d'_k)\big|\Big|,\Big|\#(\beta)-\big|\beta^{-1}(\beta\cap d''_k)\big|\Big|\leq 2.\]
\end{lem}

\begin{proof}
Suppose first that $\beta$ intersects $d_k$ transversely. Then $|\beta\cap d_k|$ is finite, and by Remark \ref{dkdk'dk''}, the three quantities 
\[\Big|\big|\beta^{-1}(\beta\cap d'_k)\big|-\big|\beta^{-1}(\beta\cap d''_k)\big|\Big|, \Big|\big|\beta^{-1}(\beta\cap d_k)\big|-\big|\beta^{-1}(\beta\cap d'_k)\big|\Big|,\Big|\big|\beta^{-1}(\beta\cap d_k)\big|-\big|\beta^{-1}(\beta\cap d''_k)\big|\Big|\] 
are all at most $1$. The inequalities in the lemma then follow from the definition of $\#(\beta)$.
In the case when $\beta$ does not intersect $d_k$ transversely, we have that $d_k\subset\beta$, so $|\beta\cap d'_k|=|\beta\cap d''_k|=\#(\beta)=0$. It is then clear that the required inequalties hold in this case as well.
\end{proof}

\begin{lem}\label{looping}
Let $\beta$ be a looping segment of $\eta$. Then 
\[|F_\beta|-2\leq2\cdot\#(\beta)\leq|F_\beta|,\]
where $F_\beta$ is defined by (\ref{Fi2}).
\end{lem}

\begin{proof}
Consider any lift $\widetilde{\beta}$ of $\beta$ in $\Omega_\Mpc$. Since $\beta$ is a looping segment, $\widetilde{\beta}$ intersects the lines in $\widetilde{\Umc}_\Mpc$ finitely many times, and all the lines in $\widetilde{\Umc}_\Mpc$ that intersect $\widetilde{\beta}\setminus\partial\widetilde{\beta}$ share a common endpoint, $q$. This implies that $\beta$ lies in a single pair of pants $\Ppc^{(j)}$ given by $\Pmc$, and that $\widetilde{\beta}\setminus\partial\widetilde{\beta}$ lies in $\Omega_q^{(j)}$. Moreover, by Remark \ref{gammaonomega}, any other lift $\widetilde{\beta}'$ of $\beta$ to $\Ntd$ which intersects $\widetilde{\beta}$ must also lie in $\Omega_q^{(j)}$. In fact, $\widetilde{\beta}'=X^k\cdot\widetilde{\beta}$ for some $k\in\Zbbb$, where $X$ is a generator of $\Gamma_q$. Thus, the number of self intersections of $\beta$ is the number of positive integers $k$ such that $\widetilde{\beta}\cap X^k\cdot\widetilde{\beta}$ is nonempty.

It follows from the definition of a looping segment that $|F_\beta|\geq 2$. Also, by the description of the action of $\Gamma_q$ on $\Omega_q^{(j)}$ given in Remark \ref{gammaonomega}, it is clear that if $|F_\beta|=2k+1$ for $k\in\Zbbb_{\geq 1}$, then $X^j\cdot\widetilde{\beta}$ intersects $\widetilde{\beta}$ for $j=\{\pm1,\dots,\pm k\}$. This implies that $\#(\beta)=k$, so the inequalities in the lemma hold. 

On the other hand, if $|F_\beta|=2k$ for $k\in\Zbbb_{\geq 1}$, then there are two possible cases. Let $q_1$ and $q_2$ be the two endpoints of $\widetilde{\beta}$ so that $X^k\cdot q_1$ and $q_2$ lie on the same line in $\widetilde{\Upc}$. Let $\gamma$ be the axis of $X$ and let $\omega:=\partial\Omega_q^{(j)}\setminus\gamma$. It is clear that $q_1$, $q_2$ and $X^k\cdot q_1$ lie on $\omega$. If $q_2$ lies between $q_1$ and $X^k\cdot q_1$ on $\omega$, then $X^j\cdot\widetilde{\beta}$ intersects $\widetilde{\beta}$ for $j=\{\pm1,\dots,\pm (k-1)\}$. (See Figure \ref{keven}.) However, if $X^k\cdot q_1$ lies between $q_1$ and $q_2$ on $\omega$, then $X^j\cdot\widetilde{\beta}$ intersects $\widetilde{\beta}$ for $j=\{\pm1,\dots,\pm k\}$. In either case, $\#(\beta)$ is either $k$ or $k-1$, so the required inequalities still hold.
\end{proof}

\begin{figure}
\includegraphics[scale=0.6]{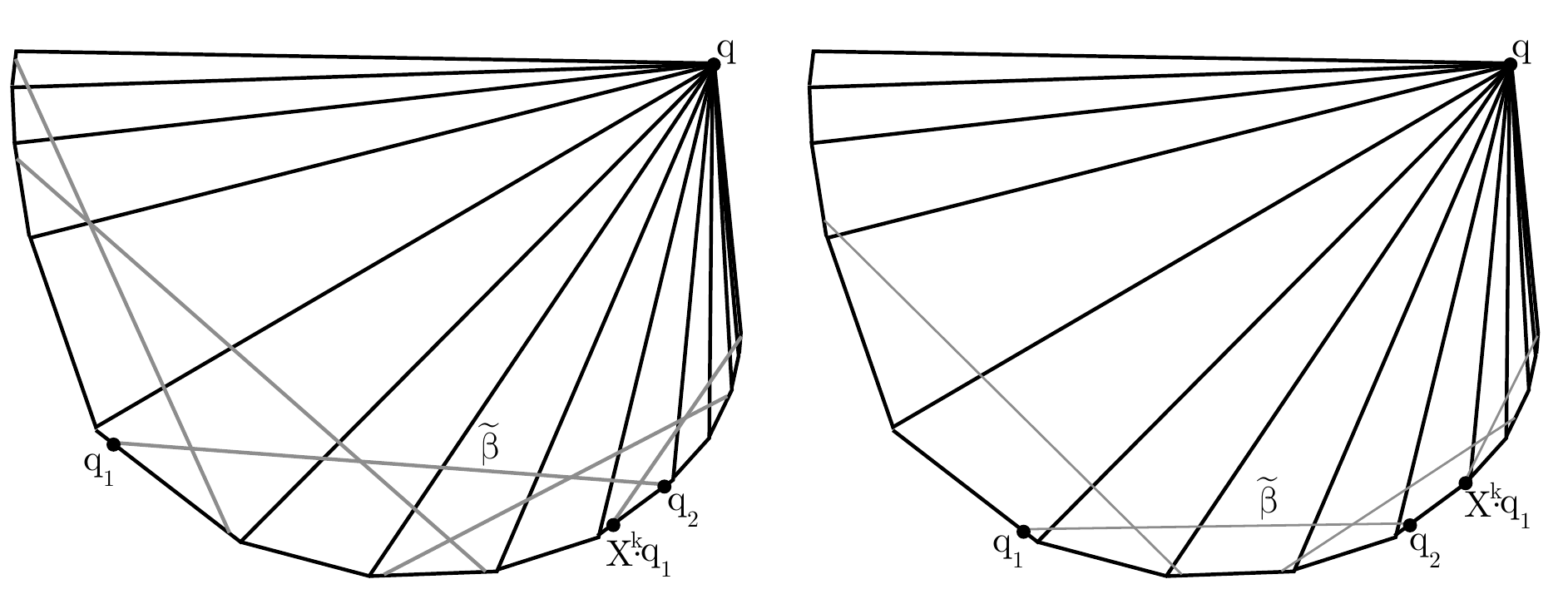}
\caption{$\#(\beta)$ and $|F_\beta|$.}
\label{keven}
\end{figure}

The above two lemmas describe the discrepancy between the information contained in $\phi(\eta)$ and $\psi(\eta)$. With these, we can prove the main proposition of this subsection.

\begin{prop}\label{phitopsi}
If $\eta\in\Tpc_\Mpc$ has $m$ crossing points, then any $\eta'\in\Tpc_\Mpc$ such that $\psi(\eta)=\psi(\eta')$ also has $m$ crossing points. Moreover, there are at most $18^m$ closed geodesics in $\Tpc_\Mpc$ that have the same image as $\eta$ under the map $\psi$.
\end{prop}

\begin{proof}
The first statement follows immediately from the definition of $\psi(\eta)$. To prove the second statement, we will pick any sequence $\chi$ in $\Gpc_\Mpc$ and reconstruct $\phi(\eta)$ for all typical closed curves $\eta$ in $N$ such that $\psi(\eta)=\chi$. Then we show that for each $\chi$, the number of possible $\phi(\eta)$ that we can construct is at most $18^m$. The fact that $\phi$ is a bijection will then imply the proposition.

Choose any $\chi=(\chi_1,\dots,\chi_m)$ in $\Gpc_\Mpc$, where each $\chi_i$ is the tuple $(\uora^-_i,\uora^0_i,\uora^+_i,a_i)$. Here, $\uora_i^+$, $\uora_i^0$, $\uora_i^-$ are lines in $\Upc_\Mpc$ equipped with an orientation, and $a_i\in\Zbbb_{\geq 0}$. Suppose $\eta\in\Tpc_\Mpc$ such that $\psi(\eta)=\chi$. We can assume without loss of generality that $\phi(\eta)=\{\phi_{m+1}=\phi_1,\dots,\phi_m\}$, with $\phi_i=(\uora_i^+,\uora_i^0,\uora_i^-,\overrightarrow{Y}_i)$, i.e. for all $i\in\{1,\dots,m\}$, $\uora_{\suc^{-1}(p_i)}=\uora^-_i$, $\uora_{p_i}=\uora^0_i$ and $\uora_{\suc(p_i)}=\uora^+_i$. 

First, note that $\chi$ contains sufficient information to determine if a pair $(p_i,p_{i+1})$ is a looping pair or a pants changing pair: $(p_i,p_{i+1})$ is a looping pair if and only if there is some pair of pants $\Ppc^{(1)}$ for $\Mpc$ such that $u_i^0$ and $u_{i+1}^0$ are lines in $\Vmc^{(1)}$, and there are lifts of the lines $u_i^0$, $u_i^+$, $u_{i+1}^-$, $u_{i+1}^0$ to $\Umc_\Mpc$ that share a common vertex. Furthermore, if $(p_i,p_{i+1})$ is a pants changing pair and $\beta_i$ is the corresponding pants changing segment, then $\chi$ determines the unique $\gamma_k$ in $\Pmc$ that intersects $\beta_i$. Indeed, if $\Ppc^{(1)},\Ppc^{(2)}\subset\Mpc$ are the pairs of pants given by $\Pmc$ containing $u^0_i$ and $u^0_{i+1}$ respectively (it might be that $\Ppc^{(1)}=\Ppc^{(2)}$), then $\gamma_k$ is the common boundary component of $\Ppc^{(1)}$ and $\Ppc^{(2)}$ that is bounded away from both $u^-_i$ in $\Ppc^{(1)}$ and $u^+_{i+1}$ in $\Ppc^{(2)}$.

Suppose now that $(p_i,p_{i+1})$ is a pants changing pair and $\gamma_k$ is the unique curve in $\Pmc$ obtained in the previous paragraph. If $\beta_i$ is the pants changing segment of $\eta$ corresponding to the pants changing pair $(p_i,p_{i+1})$, then $\#(\beta_i)=a_i$. Hence, Lemma \ref{pantschanging} implies that $Y_i$ has to be a finite sequence that alternates between $d_k'$ and $d_k''$, so that if $n'$ and $n''$ are the number of times $d_k'$ and $d_k''$ respectively occur in $Y_i$, then $|n'-n''|\leq 1$, $|a_i-n'|\leq 2$ and $|a_i-n''|\leq 2$. One can easily verify that there are at most eighteen possibilities for $Y_i$. 

Moreover, the orientation of $\uora_{\suc^{-1}(p_i)}$ and the sequence $Y_i$ determines $\overrightarrow{Y}_i$ uniquely. Explicitly, let $\Ppc^{(1)}$ be the pair of pants for $\Mpc$ that contains $u_{\suc^{-1}(p_i)}$ and let $x'$ and $x''$ be the two boundary components of $\Ppc^{(1)}$ that $\uora_{\suc^{-1}(p_j)}$ accumulates to, such that $\uora_{\suc^{-1}(p_i)}$ points from $x'$ to $x''$. Also, let $\gamma_k$ be the third boundary component of $\Ppc^{(1)}$ that is not $x'$ or $x''$. Assume without loss of generality that $d_k'$ accumulates to $x'$ and $d_k''$ accumulates to $x''$. If the first term of $Y_i$ is $d_k''$, then we orientate $d_k'$ towards $x'$ and $d_k''$ towards $x''$ to obtain $\overrightarrow{Y}_i$. If the first term of $Y_i$ is $d_k'$, then we orientate $d_k'$ away from $x'$ and $d_k''$ away from $x''$ to obtain $\overrightarrow{Y}_i$. These orientations are forced onto us by the requirement that there is a geodesic that passes from the left to right across each oriented line in $\overrightarrow{Y}_i$. In particular, if we know $\psi(\eta)$, then there are only at most eighteen different possibilities for $\overrightarrow{Y}_i$ for each $i$.

On the other hand, if $(p_i,p_{i+1})$ is a looping pair, then $\#(\beta_i)=a_i$, so Lemma \ref{looping}, tells us that $Y_i$ has to be a sequence that alternates between $u^0_i$ and $u^+_i$, with either $2a_i+2$, $2a_i+1$ or $2a_i$ terms. As in the previous paragraph, given the orientation on $\uora_{p_i}$ and the sequence $Y_i$, there is a unique way to orient the terms of $Y_i$ to recover $\overrightarrow{Y}_i$. Explicitly, let $\Ppc^{(1)}$ be the pair of pants given by $\Pmc$ containing $\beta_i$, and let $\gamma_k$ be the boundary component of $\Ppc^{(1)}$ that every line in $Y_i$ accumulates to. If $\uora_{p_i}$ is oriented towards $\gamma_k$, then we orient every term in $Y_i$ towards $\gamma_k$, and if $\uora_{p_i}$ is oriented away from $\gamma_k$, then we orient every term in $Y_i$ away from $\gamma_k$. In particular, this tells us that there are at most three different possibilities for $\overrightarrow{Y}_i$.

Putting all of these together, we get that if $\chi$ has $m_1$ pants changing pairs and $m_2$ looping pairs, then
\[|\{\phi(\eta):\eta\in\Tpc_\Mpc,\psi(\eta)=\chi\}|\leq 18^{m_1}\cdot 3^{m_2}\leq 18^{m_1+m_2}=18^m.\]
Since $\phi$ is a bijection, we are done.
\end{proof}

\subsection{Lower bounds on lengths and proof of (1) of Theorem \ref{mainthm1}.}\label{lowerboundlengths}

In this subsection, we will show that we can obtain a lower bound for the length of any $\eta\in\Tpc_\Mpc$ in terms of the Goldman coordinates and combinatorial data encoded by $\psi(\eta)$. As a consequence, we prove (1) of Theorem \ref{mainthm1}. Note that to prove (1) of Theorem \ref{mainthm1}, we  actually only need to obtain a lower bound for crossing segments and show that this lower bound goes to $\infty$ as we deform along Goldman sequences. However, we need estimates of the lengths of the pants changing segment and looping segments to prove (2). 

First, we will give the length lower bound for pants changing segments and looping segments. Given a pair of boundary invariants $(\lambda,\tau)$ for a marked convex $\Rbbb\Pbbb^2$ surface $\Mpc\in\Cmf(M_{g,n})$, let $\gamma$ be the simple closed curve in $\Mpc$ corresponding to $(\lambda,\tau)$. We know that the length of $\gamma$ in the Hilbert metric is given by (see Equation \ref{lengthofc})
\[l_\Mpc(\gamma)=\log\Bigg(\frac{\tau+\sqrt{\tau^2-\frac{4}{\lambda}}}{2\lambda}\Bigg).\]
Thus, we can define the function 
\[L:\Cmf(M_{g,n})\to\Rbbb^+\]
given by $L(\Mpc)=\frac{1}{10}\min\{l_\Mpc(\gamma):\gamma\in\Pmc\}$.

We want to bound the lengths of the pants changing segments and looping segments of $\eta$ from below by a constant multiple of $L(\Mpc)$, where the constant depends only on $\psi(\eta)$. But first, we need the following lemma.

\begin{lem}\label{matchingloops}
Let $\beta$ be a looping segment of $\eta\in\Tpc_\Mpc$. Let $\widetilde{\beta}$ be a lift of $\beta$ to $\Omega_\Mpc$, and let $\{r_1,\dots,r_l\}$ be the points in $\widetilde{\beta}$ that also lie on other lifts of $\beta$, enumerated according to the orientation on $\beta$. Then $l$ is even and the covering map $\Pi:\Omega_\Mpc\to\Mpc$ satisfies $\Pi(r_j)=\Pi(r_{l+1-j})$ for any $j\in\{1,\dots,l\}$.  
\end{lem}

\begin{proof}
Let $\Ppc^{(i)}$ be the pair of pants given by $\Pmc$ that contains $\beta$, let $\widetilde{\beta}$ be a lift of $\beta$ to $\Omega_\Mpc$ and let $\Omega_{\Ppc^{(i)}}$ be the lift of $\Ppc^{(i)}$ to $\Omega_\Mpc$ that contains $\widetilde{\beta}$. Since $\beta$ is a looping segment, all the lines in $\widetilde{\Umc}_\Mpc$ that intersect $\widetilde{\beta}\setminus\partial\widetilde{\beta}$ share a common endpoint in $\partial\Omega_\Mpc$. Let this common endpoint be $q$, and it is clear that $\widetilde{\beta}\setminus\partial\widetilde{\beta}$ lies in $\Omega_q^{(i)}\subset\Omega_{\Ppc^{(i)}}$. Since $q$ is the image of either $a^{(i)}$, $b^{(i)}$ or $c^{(i)}$ under a deck transformation, we can assume without loss of generality that $q=a^{(i)}$, $b^{(i)}$ or $c^{(i)}$. We will do the proof for $q=a^{(i)}$; the other two cases are similar.

Let $A^{(i)}=h_\Mpc\circ(f_i)_*(A_0)$ ($A_0$ and $f_i$ are defined at the start of Section \ref{maincontent}). The properness of $\Omega_\Mpc$ and the fact that it is invariant under $A^{(i)}$ implies that $\Omega_{\Ppc^{(i)}}$ is contained in a triangle $T$ whose vertices are the three fixed points of $A^{(i)}$. Via a projective transformation, we can assume that the attracting, repelling and third fixed point of $A^{(i)}$ are $[1,0,0]$, $[0,0,1]$ and $[0,1,0]$ respectively, and $T=\{[x,y,z]\in\Rbbb\Pbbb^2:x,y,z>0\}$. Goldman computed in Section 1.10 of \cite{Go1} that the $A^{(i)}$ orbit of any $[x_0,y_0,z_0]\in T$ lies on the curve 
\[\omega_{[x_0,y_0,z_0]}:=\bigg\{[x,y,z]\in T:\bigg(\frac{x}{x_0}\bigg)^{\log(\frac{\nu}{\mu})}\bigg(\frac{z}{z_0}\bigg)^{\log(\frac{\mu}{\lambda})}=\bigg(\frac{y}{y_0}\bigg)^{\log(\frac{\nu}{\lambda})}\bigg\},\]
where $\lambda$, $\mu$, $\nu$ are the eigenvalues of $A^{(i)}$ for $[1,0,0]$, $[0,1,0]$ and $[0,0,1]$ respectively. One can easily verify that the curves $\omega_{[x_0,y_0,z_0]}$ for $[x_0,y_0,z_0]\in T$ foliate $T$, and that any line intersects each of these curves at most twice. 

If $\Pi(r_{j_1})=\Pi(r_{j_2})$, then there is some deck transformation $X$ such that $X\cdot r_{j_1}=r_{j_2}$. However, since $r_{j_1}$ and $r_{j_2}$ lie in $\Omega_{a^{(i)}}$, we know $X$ is of the form $X=(A^{(i)})^k$ for some integer $k$. The previous paragraph then proves that for any self-intersection $p$ of $\beta$, there are exactly two points in $\{r_1,\dots,r_l\}$ that are mapped to $p$ by $\Pi$, so $l$ is even.

Now, suppose that there is some $j\in\{1,\dots l\}$ such that $\Pi(r_j)\neq\Pi(r_{l+1-j})$. By the pigeon hole principle, one of the following must hold:
\begin{enumerate}[(1)]
\item There is some $j'<j$ and some $j''<l+1-j$ such that $\Pi(r_{j'})=\Pi(r_{j''})$.
\item There is some $j'>j$ and some $j''>l+1-j$ such that $\Pi(r_{j'})=\Pi(r_{j''})$.
\end{enumerate}
Either way, this implies that the curves $\omega_{r_j}$ and $\omega_{r_{j'}}$ must intersect, contradiction.
\end{proof}

With this, we can obtain the required lower bound stated in the next proposition. In the case when $\Mpc$ lies in the Fuchsian locus, this proposition follows from the fact that the projection of $\Hbbb^2$ onto a geodesic is $1$-Lipschitz.

\begin{prop}\label{betalowerbound}
If $\beta$ is a pants changing segment or looping segment for a typical oriented closed geodesic $\eta$ in a marked convex $\Rbbb\Pbbb^2$ surface $\Mpc\in\Cmf(M_{g,n})$, then 
\[l_\Mpc(\beta)\geq 5L(\Mpc)\cdot\#(\beta).\]
\end{prop}

\begin{proof}
First, consider the case when $\beta$ is a pants changing segment. Let $\Ppc^{(1)}$ and $\Ppc^{(2)}$ be the two pairs of pants given by $\Pmc$ such that $\beta\subset \Ppc^{(1)}\cup \Ppc^{(2)}$ and let $\gamma_k$ be the unique element in $\Pmc$ that intersects $\beta$. Also, let $x_1$ and $x_2$ be the lines in the ideal triangulation of $\Ppc^{(1)}$ and $\Ppc^{(2)}$ respectively that do not accumulate to $\gamma_k$. If $\beta$ intersects $d_k$ non-transversely or $\beta$ intersects $d_k$ at most once, then $\#(\beta)=0$ and the required lower bound clearly holds. 

If $\beta$ intersects $d_k$ transversely and at least twice, then $\#(\beta)\geq 1$. Enumerate $\beta^{-1}(\beta\cap d_k)=\{q_1,\dots,q_{\#(\beta)+1}\}$ according to the total order induced by the orientation on $\beta$. Let $\mu_j:=\eta|_{[q_j,q_{j+1}]}$ and let $\nu_j:[q_j,q_{j+1}]\to\Mpc$ be the subsegment of $d_k$ with $\nu_j(q_j)=\eta(q_j)$ and $\nu_j(q_{j+1})=\eta(q_{j+1})$. Since $\mu_j$ and $\nu_j$ are distinct line segments sharing the same endpoints, they cannot be homotopic relative endpoints by (3) of Proposition \ref{properties}. Hence, $\mu_j\cdot\nu_j^{-1}$ is a non-contractible simple closed curve in the topological annulus $(S^{(1)}\cup S^{(2)})\setminus(x_1\cup x_2)$, so it is homotopic to $\gamma_k$. This implies that 
\[l_\Mpc(\mu_j)+l_\Mpc(\nu_j)=l_\Mpc(\mu_j\cdot\nu_j^{-1})\geq l_\Mpc(\gamma_k)\geq 10L(\Mpc).\]
On the other hand, since $d_k$ is a distance minimizing line segment between $x_1$ and $x_2$, we know that $l_\Mpc(\mu_j)\geq l_\Mpc(\nu_j)$. Thus, $l_\Mpc(\mu_j)\geq 5L(\Mpc)$ for all $j\in\{1,\dots,\#(\beta)\}$, so
\[l_\Mpc(\beta)\geq\sum_{j=1}^{\#(\beta)}l_\Mpc(\mu_j)\geq 5L(\Mpc)\cdot\#(\beta).\]

Next, consider the case when $\beta$ is a looping segment. Enumerate the set
\[\beta^{-1}(\{\text{self-intersections of }\beta\})=\{q_1,\dots,q_l\}\]
according to the orientation on $\eta$. The proof of Lemma \ref{matchingloops} gives us that $l=2\cdot\#(\beta)$. Define $\delta_j:=\beta|_{[q_j,q_{j+1}]}$ for $j\in\{1,\dots,l-1\}$

By Lemma \ref{matchingloops}, we know that $\beta(q_j)=\beta(q_{l+1-j})$ for all $j\in\{1,\dots,l\}$, so for $j\neq\frac{l}{2}$, $\delta_j$ and $\delta_{l-j}$ are two line segments that share their endpoints, but do not intersect anywhere else. In particular, $\delta_j\cdot\delta_{l-j}^{-1}$ is a non-contractible simple closed curve. Let $\Ppc^{(i)}$ be the pair of pants given by $\Pmc$ that contains $\beta$, and note that (2) of Proposition \ref{properties} implies that $\delta_j\cdot\delta_{l-j}^{-1}$ has to be homotopic to a boundary component of $\Ppc^{(i)}$. By using this same argument on the line segment $\delta_{\frac{l}{2}}$ and the trivial line segment that is just the endpoint of $\delta_{\frac{l}{2}}$, we see that $\delta_{\frac{l}{2}}$ is also homotopic to a boundary component of $\Ppc^{(i)}$. Thus, $l_\Mpc(\delta_j)+l_\Mpc(\delta_{l-j})\geq 10L(\Mpc)$ for $j\neq \frac{l}{2}$ and $l_\Mpc(\delta_{\frac{l}{2}})\geq 10L(\Mpc)$. This implies that 
\[l_\Mpc(\beta)\geq\sum_{j=1}^{l-1}l_\Mpc(\delta_j)\geq\frac{l}{2}\cdot 10L(\Mpc)\geq 5L(\Mpc)\cdot\#(\beta).\]
\end{proof}

Now, we want to find the corresponding length lower bound for the crossing segments of $\eta$. To do so, we first need to classify the crossing segments of $\eta$ into two types. Let $p_i$ be a crossing point of $\eta$ and let $\alpha_i$ be the corresponding crossing segment. Then let $T_i^-$ be the triangle in $\Mpc$ bounded by $u_{\suc^{-1}(p_i)}$, $u_{p_i}$ and $\alpha_i$, and let $T_i^+$ be the triangle in $\Mpc$ bounded by $u_{\suc(p_i)}$, $u_{p_i}$ and $\alpha_i$. 

\begin{definition}
Consider any crossing segment $\alpha_i$ of $\eta\in\Tpc_\Mpc$, oriented according to the cyclic orientation on $\eta$. Then $\alpha_i$ is said to be of \emph{Z-type} if $T_i^-$ is to the left of $\alpha_i$ and $T_i^+$ is to the right of $\alpha_i$. It is said to be of \emph{S-type} if $T_i^-$ is to the right of $\alpha_i$ and $T_i^+$ is to the left of $\alpha_i$. (See Figure \ref{SZ}.)
\end{definition}

\begin{figure}
\includegraphics[scale=0.6]{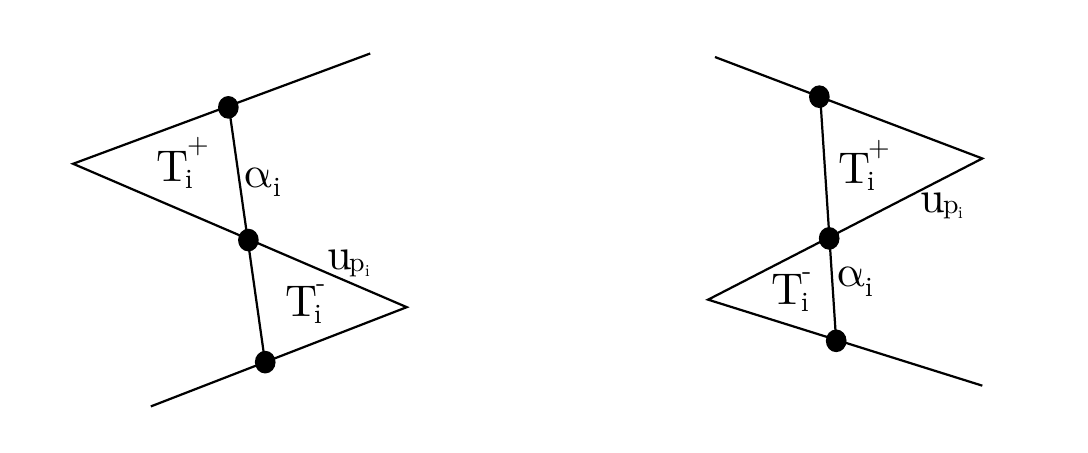}
\caption{S-type (on the left) and Z-type (on the right) crossing segments.}
\label{SZ}
\end{figure}

It is clear that every crossing segment is either of S-type or Z-type, and reversing orientation preserves the type. Observe also that if $(p_i,p_{i+1})$ is a looping pair, then the crossing segment $\alpha_i$ for $p_i$ is of S-type if and only if the crossing segment $\alpha_{i+1}$ for $p_{i+1}$ is of Z-type.

For the rest of the paper, we will simplify notation by writing $Cr_{x,y}(\Ppc^{(i)})$ as $Cr_{x,y}^{(i)}$. (See Section \ref{areparameterization}.) Define the function 
\[K:\Cmf(M_{g,n})\to\Rbbb^+\]
given by $K(\Mpc)=\frac{1}{48}\min\{\log(X^{(i)}\cdot Y^{(i)}):1\leq i\leq 2g-2+n\}$, where $X^{(i)}$ is the minimum of the following six numbers:
\begin{eqnarray}\label{list1}
Cr_{f,d}^{(i)}Cr_{c,f}^{(i)},\,\,\, Cr_{f,d}^{(i)}Cr_{b,f}^{(i)},\,\,\,Cr_{d,e}^{(i)}Cr_{a,d}^{(i)},\,\,\,Cr_{d,e}^{(i)}Cr_{c,d}^{(i)},\,\,\, Cr_{e,f}^{(i)}Cr_{b,e}^{(i)},\,\,\, Cr_{e,f}^{(i)}Cr_{a,e}^{(i)},
\end{eqnarray}
and $Y^{(i)}$ is the minimum of the following six numbers:
\begin{eqnarray}\label{list2}
Cr_{d,f}^{(i)}Cr_{b,d}^{(i)},\,\,\, Cr_{d,f}^{(i)}Cr_{a,d}^{(i)},\,\,\,Cr_{e,d}^{(i)}Cr_{c,e}^{(i)},\,\,\, Cr_{e,d}^{(i)}Cr_{b,e}^{(i)},\,\,\,Cr_{f,e}^{(i)}Cr_{a,f}^{(i)},\,\,\, Cr_{f,e}^{(i)}Cr_{c,f}^{(i)}.
\end{eqnarray}

The next proposition provide length lower bounds for the crossing segments of $\eta\in\Tpc_\Mpc$.

\begin{prop}\label{alphalowerbound}
Let $\eta\in\Tpc_\Mpc$ and let $\{p_{m+1}=p_1,p_2,\dots,p_m\}$ be the crossing points of $\eta$, ordered cyclically according to the orientation of $\eta$. Also, let $\alpha_j$ be the crossing segment corresponding to $p_j$ and let $\beta_j$ be the looping or pants changing segment corresponding to $(p_j,p_{j+1})$.
\begin{enumerate}[(1)]
\item If $(p_j,p_{j+1})$ is a looping pair, then $l_\Mpc(\alpha_j)+l_\Mpc(\alpha_{j+1})\geq 24K(\Mpc)$.
\item If $\eta$ is contained in some pair of pants $\Ppc^{(i)}$ given by $\Pmc$, then $l_\Mpc(\eta)\geq 12K(\Mpc)$.
\item Suppose that $(p_{j-1},p_j)$ and $(p_j,p_{j+1})$ are both pants changing pairs. Let $q_{j-1}$ and $q_j$ be the unique points in $D_\eta$ (see Section \ref{setup}) such that $q_{j-1}\in\eta^{-1}(\beta_{j-1})$ and $q_j\in\eta^{-1}(\beta_j)$. Define $\hat{\alpha}_j$ to be the subsegment $\eta|_{I}$, where $I\subset S^1$ is the closed subsegment containing $\eta^{-1}(\alpha_j)$, with endpoints $q_{j-1}$ and $q_j$. If $K(\Mpc)>l_\Mpc(\gamma)$ for all $\gamma\in\Pmc$, then $l_\Mpc(\hat{\alpha}_j)\geq 5K(\Mpc).$
\end{enumerate}
\end{prop}

\begin{proof}
In this proof, we will use the following notation. Let $\Omega$ be a properly convex subset of $\Rbbb\Pbbb^2$ and let $a,b$ be any two distinct points in the closure $\overline{\Omega}$ of $\Omega$ in $\Rbbb\Pbbb^2$. Let $[a,b]_\Omega$ denote the closed line segment in $\overline{\Omega}$ with endpoints $a$, $b$. If $a,b$ are two distinct points in $\partial\Omega$, choose an oriented affine chart $U$ containing the closure of $\Omega$, and define $(a,b)_{\partial\Omega}$ to be the oriented open subsegment of $\partial\Omega$ that goes from $a$ to $b$ in the clockwise direction in $U$ along $\partial\Omega$.

Proof of (1). Assume without loss of generality that $\alpha_j$ is of S-type and $\alpha_{j+1}$ is of Z-type. Let $\Ppc^{(i)}$ be the pair of pants given by $\Pmc$ containing $\alpha_j$ and $\alpha_{j+1}$, and consider the lamination $\widetilde{\Vmc}^{(i)}$ of $\Omega_{\Ppc^{(i)}}$ (drawn partially in Figure \ref{goldman}). There is a deck transformation that sends $u_{p_j}$ to either $[c,a]_{\Omega_\Mpc}$, $[a,b]_{\Omega_\Mpc}$ or $[b,c]_{\Omega_\Mpc}$, so one of the following must hold because $\alpha_j$ is of S-type:
\begin{enumerate}[(i)]
\item There is a lift $\widetilde{\alpha}_j$ of $\alpha_j$ with endpoints in $[a,e]_{\Omega_\Mpc}$ and $[b,c]_{\Omega_\Mpc}$. 
\item There is a lift $\widetilde{\alpha}_j$ of $\alpha_j$ with endpoints in $[b,f]_{\Omega_\Mpc}$ and $[c,a]_{\Omega_\Mpc}$. 
\item There is a lift $\widetilde{\alpha}_j$ of $\alpha_j$ with endpoints in $[c,d]_{\Omega_\Mpc}$ and $[a,b]_{\Omega_\Mpc}$. 
\end{enumerate}

\begin{figure}
\includegraphics[scale=0.5]{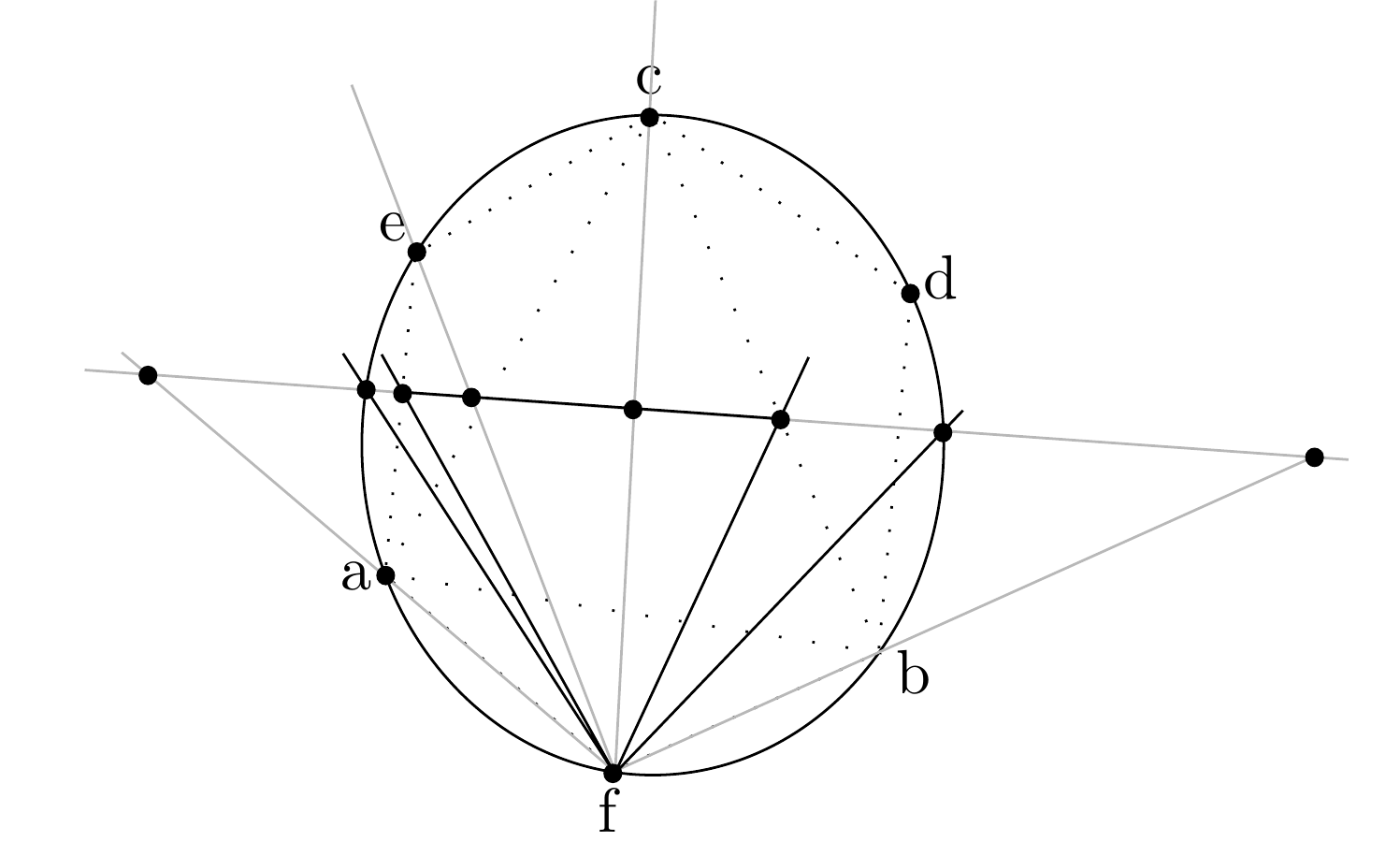}
\caption{$l_{\Omega_\Mpc}(\widetilde{\alpha}_j)\geq\log(Cr_{f,d}^{(i)})$.}
\label{estimate1}
\end{figure}

Suppose first that (i) holds. Then by applying Proposition \ref{crossratioinequality}, we see that $l_{\Mpc}(\alpha_j)=l_{\Omega_\Mpc}(\widetilde{\alpha}_j)\geq\log(Cr_{f,d}^{(i)})$. (See Figure \ref{estimate1}.) Also, it is clear that the line segment in $\Omega_\Mpc$ containing the segment $\widetilde{\alpha}_j$ and with endpoints in $\partial\Omega_\Mpc$ has one endpoint $x$ in $(a,e)_{\partial\Omega_\Mpc}$ and one endpoint $y$ in $(c,d)_{\partial\Omega_\Mpc}\cup(d,b)_{\partial\Omega_\Mpc}$. Using Proposition \ref{crossratioinequality} again, we see that if $y$ lies in $(d,b)_{\partial\Omega_\Mpc}$, then $l_\Mpc(\alpha_j)\geq\log(Cr_{c,f}^{(i)})$, and if $y$ lies in $(c,d)_{\partial\Omega_\Mpc}$, then $l_\Mpc(\alpha_j)\geq\log(Cr_{b,f}^{(i)})$. (See Figure \ref{estimate2} and Figure \ref{estimate3}.) Thus, 
\begin{eqnarray*}
l_\Mpc(\alpha_j)&\geq&\min\{\max\{\log(Cr_{f,d}^{(i)}),\log(Cr_{c,f}^{(i)})\},\max\{\log(Cr_{f,d}^{(i)}),\log(Cr_{b,f}^{(i)})\}\}\\
&\geq&\min\bigg\{\frac{\log(Cr_{f,d}^{(i)})+\log(Cr_{c,f}^{(i)})}{2},\frac{\log(Cr_{f,d}^{(i)})+\log(Cr_{b,f}^{(i)})}{2}\bigg\}\\
&\geq&\frac{1}{2}\log(\min\{Cr_{f,d}^{(i)}Cr_{c,f}^{(i)}\,,\,Cr_{f,d}^{(i)}Cr_{b,f}^{(i)}\})
\end{eqnarray*}
The same argument shows that if (ii) holds, then 
\begin{equation*}
l_\Mpc(\alpha_j)\geq\frac{1}{2}\log(\min\{Cr_{d,e}^{(i)}Cr_{a,d}^{(i)}\,,\,Cr_{d,e}^{(i)}Cr_{c,d}^{(i)}\}).
\end{equation*}
and if (iii) holds, then 
\begin{equation*}
l_\Mpc(\alpha_j)\geq\frac{1}{2}\log(\min\{Cr_{e,f}^{(i)}Cr_{b,e}^{(i)}\,,\,Cr_{e,f}^{(i)}Cr_{a,e}^{(i)}\}),
\end{equation*}
Hence, $l_\Mpc(\alpha_j)\geq\frac{1}{2}\log(X^{(i)})$, where $X^{(i)}$ is the minimum of the six numbers listed in (\ref{list1}).

\begin{figure}
\includegraphics[scale=0.5]{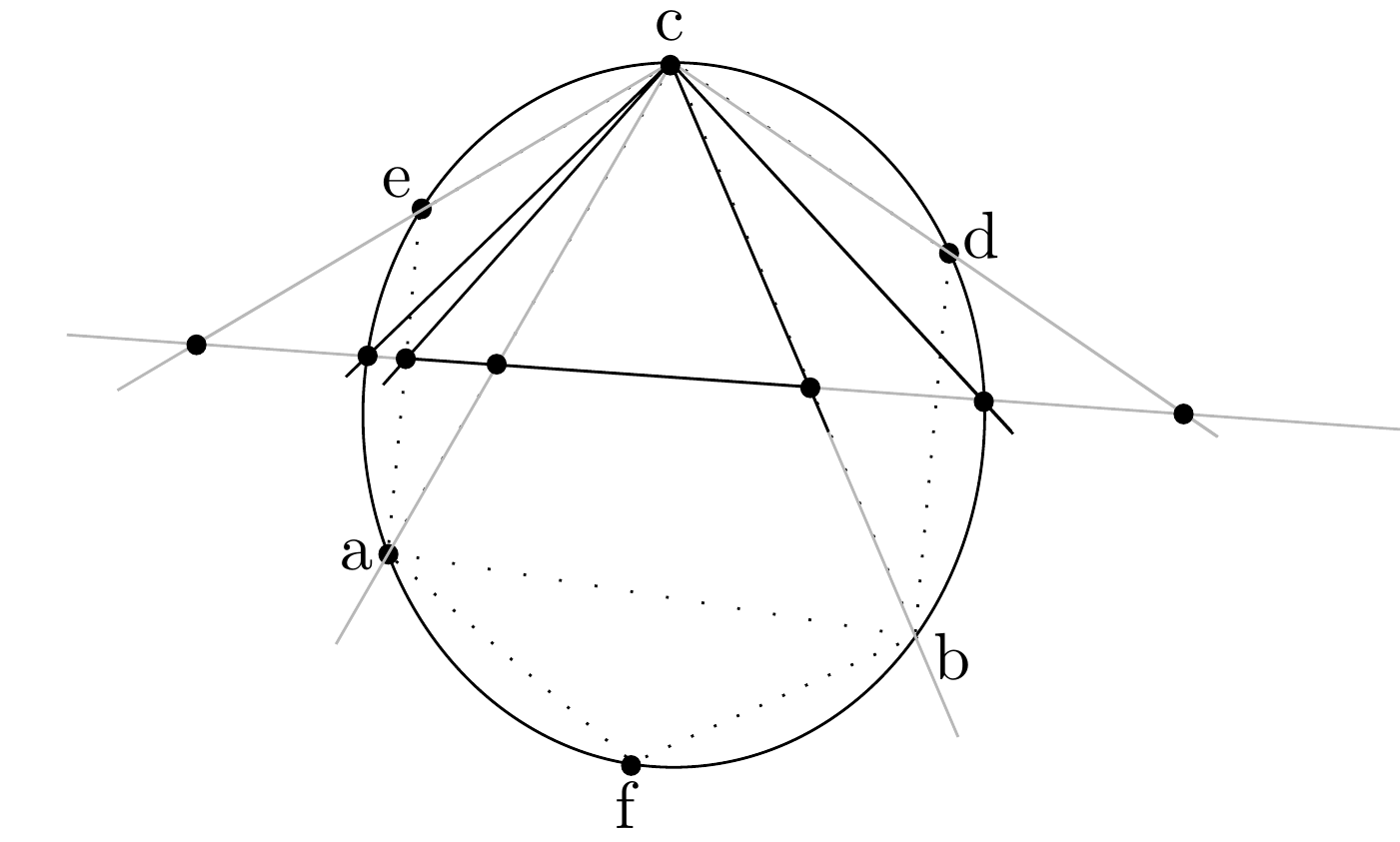}
\caption{$l_{\Omega_\Mpc}(\widetilde{\alpha}_j)\geq\log(Cr_{c,f}^{(i)})$.}
\label{estimate2}
\end{figure}

\begin{figure}
\includegraphics[scale=0.5]{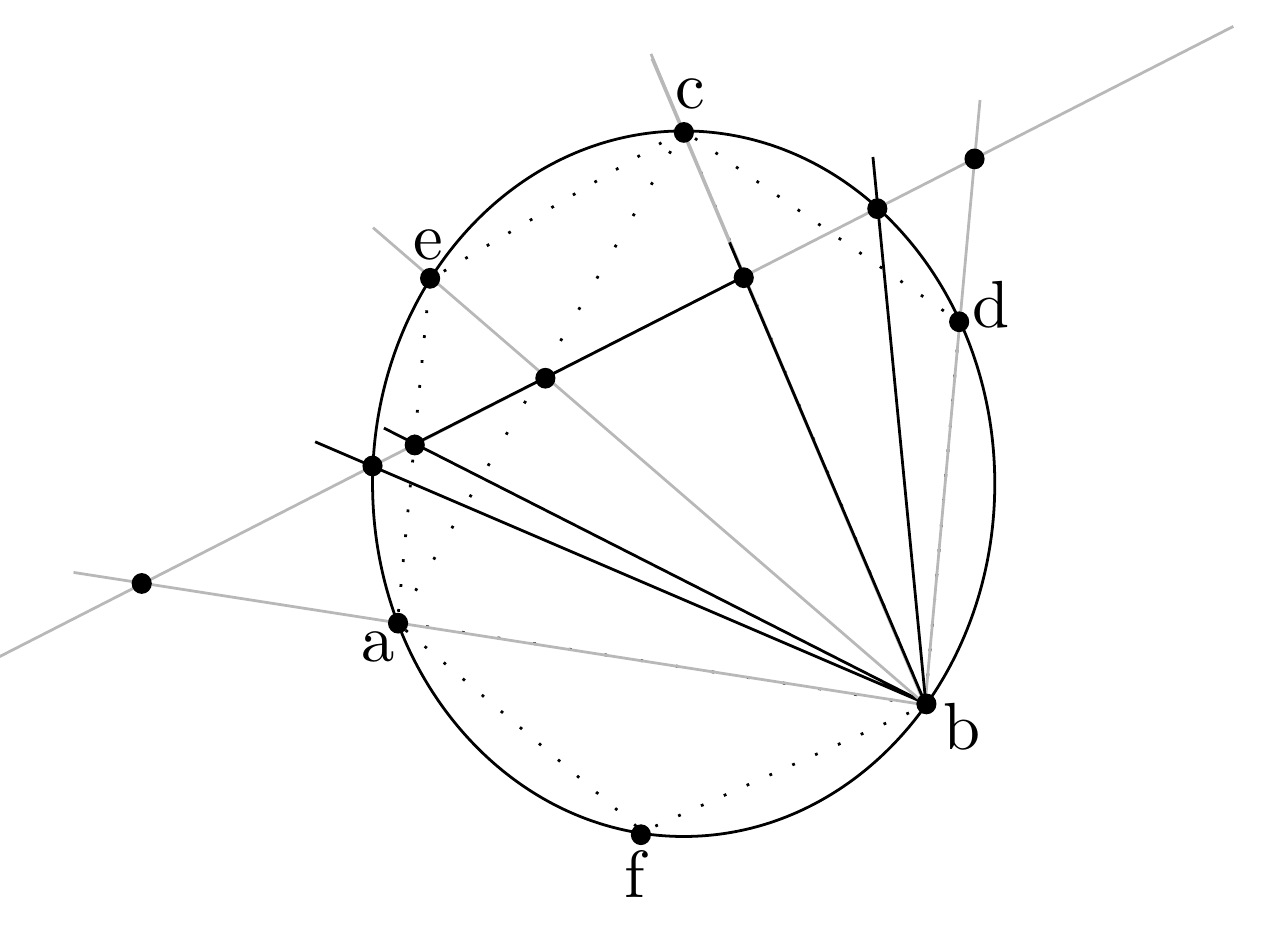}
\caption{$l_{\Omega_\Mpc}(\widetilde{\alpha}_j)\geq\log(Cr_{b,f}^{(i)})$.}
\label{estimate3}
\end{figure}

Similarly, since there is a deck transformation that sends $u_{p_{j+1}}$ to either $[c,a]_{\Omega_\Mpc}$, $[a,b]_{\Omega_\Mpc}$ or $[b,c]_{\Omega_\Mpc}$ and $\alpha_{j+1}$ is of Z-type, one of the following must hold:
\begin{enumerate}[(I)]
\item There is a lift $\widetilde{\alpha}_{j+1}$ of $\alpha_{j+1}$ with endpoints in $[c,e]_{\Omega_\Mpc}$ and $[b,a]_{\Omega_\Mpc}$. 
\item There is a lift $\widetilde{\alpha}_{j+1}$ of $\alpha_{j+1}$ with endpoints in $[a,f]_{\Omega_\Mpc}$ and $[c,b]_{\Omega_\Mpc}$. 
\item There is a lift $\widetilde{\alpha}_{j+1}$ of $\alpha_{j+1}$ with endpoints in $[b,d]_{\Omega_\Mpc}$ and $[a,c]_{\Omega_\Mpc}$. 
\end{enumerate}
Using the same arguments as in the previous paragraph, we can show that if (I), (II) or (III) holds, then
\begin{equation*}
l_\Mpc(\alpha_{j+1})\geq\frac{1}{2}\log(\min\{Cr_{d,f}^{(i)}Cr_{b,d}^{(i)},Cr_{d,f}^{(i)}Cr_{a,d}^{(i)}\}),
\end{equation*} 
\begin{equation*}
l_\Mpc(\alpha_{j+1})\geq\frac{1}{2}\log(\min\{Cr_{e,d}^{(i)}Cr_{c,e}^{(i)},Cr_{e,d}^{(i)}Cr_{b,e}^{(i)}\})
\end{equation*} 
or
\begin{equation*}
l_\Mpc(\alpha_{j+1})\geq\frac{1}{2}\log(\min\{Cr_{f,e}^{(i)}Cr_{a,f}^{(i)},Cr_{f,e}^{(i)}Cr_{c,f}^{(i)}\})
\end{equation*} 
respectively. Thus,
$l_\Mpc(\alpha_{j+1})\geq\frac{1}{2}\log(Y^{(i)})$, where $Y^{(i)}$ is the minimum of the six numbers listed in (\ref{list2}).

Putting our lower bounds for the lengths of $\alpha_j$ and $\alpha_{j+1}$ together, we get
\[l_\Mpc(\alpha_j)+l_\Mpc(\alpha_{j+1})\geq\frac{1}{2}\log(X^{(i)}\cdot Y^{(i)})\geq24K(\Mpc).\]

Proof of (2). Since $\eta$ is contained in $\Ppc^{(i)}$, we know that $\eta$ has at least two crossing points. Thus, 
\[l_\Mpc(\eta)\geq\frac{1}{2}(l_\Mpc(\alpha_1)+l_\Mpc(\alpha_2))\geq 12K(\Mpc).\]

Proof of (3). Let $\gamma_{j-1}$ and $\gamma_j$ be the two curves in $\Pmc$ that contain $q_{j-1}$ and $q_j$ respectively, parameterized by $\gamma_k:[0,1]\to N$, $\gamma_k(0)=\gamma_k(1)=q_k$ for $k=j-1,j$. Note that either $\gamma_{j-1}\cdot\hat{\alpha}_j\cdot\gamma_j\cdot(\hat{\alpha}_j)^{-1}$ or $\gamma_{j-1}\cdot\hat{\alpha}_j\cdot\gamma_j^{-1}\cdot(\hat{\alpha}_j)^{-1}$ is typical and their geodesic representatives both lie in some pair of pants given by $\Pmc$. Either way, we can apply (2) to get 
\[l_\Mpc(\gamma_{j-1})+l_\Mpc(\gamma_j)+2l_\Mpc(\hat{\alpha}_j)\geq 12K(\Mpc),\]
which implies that $l_\Mpc(\hat{\alpha}_j)\geq 6K(\Mpc)-\frac{1}{2}(l_\Mpc(\gamma_{j-1})+l_\Mpc(\gamma_j))\geq 5K(\Mpc)$.
\end{proof}

We can finally give the lower bound on the length of any $\eta\in\Tpc_\Mpc$, as promised at the start of this subsection.

\begin{thm}\label{mainestimate}
Let $\eta\in\Tpc_\Mpc$. If $\Mpc$ is such that $K(\Mpc)>l_\Mpc(\gamma)$ for any $\gamma\in\Pmc$, then
\[l_\Mpc(\eta)\geq B(\psi(\eta)):=m\cdot K(\Mpc)+\sum_{\beta\in\Ymf_\eta\cup\Zmf_\eta}\#(\beta)\cdot L(\Mpc)\]
where $m$ is the number of crossing points for $\eta$. 
\end{thm}

\begin{proof}
Let $\{p_{m+1}=p_1,\dots,p_m\}$ be the crossing points for $\eta$, ordered according to the orientation on $\eta$. Define
\begin{eqnarray*}
\Amf&:=&\{j:(p_{j-1},p_j)\text{ and }(p_j,p_{j+1})\text{ are both pants changing pairs}\},\\
\Bmf&:=&\{j:\text{either }(p_{j-1},p_j)\text{ or }(p_j,p_{j+1})\text{ is a looping pair}\},
\end{eqnarray*}
and note that $\Amf$ and $\Bmf$ are disjoint and $\Amf\cup\Bmf=\{1,\dots,m\}$.

Let $J$ be the set of interiors of all domains for the pants changing segments of $\eta$, the looping segments of $\eta$, the subsegments $\hat{\alpha}_j$ of $\eta$ for $j\in\Amf$, and the subsegments of $\alpha_j$ for $\eta$ if $j\in\Bmf$. One can easily verify that $J$ is an open cover for $S^1$ and that every point in $S^1$ is contained in at most five different elements of $J$. This implies that 
\begin{equation}\label{lowerbound1}
5l_\Mpc(\eta)\geq\sum_{\beta\in\Ymf_\eta\cup\Zmf_\eta}l_\Mpc(\beta)+\sum_{j\in\Amf} l_\Mpc(\hat{\alpha}_j)+\sum_{j\in\Bmf} l_\Mpc(\alpha_j).
\end{equation}

By Proposition \ref{betalowerbound}, we know
\begin{equation}\label{lowerbound2}
\sum_{\beta\in\Ymf_\eta\cup\Zmf_\eta}l_\Mpc(\beta)\geq 5\sum_{\beta\in\Ymf_\eta\cup\Zmf_\eta}\#(\beta)\cdot L(\Mpc),
\end{equation} 
and (3) of Proposition \ref{alphalowerbound} implies 
\begin{equation}\label{lowerbound3}
\sum_{j\in\Amf}l_\Mpc(\hat{\alpha}_j)\geq 5|\Amf|K(\Mpc).
\end{equation}
Also, if we define $\Bmf':=\{j\in\Bmf:j+1\text{ is also in }\Bmf\}$, observe that $\displaystyle|\Bmf'|\geq\frac{1}{2}|\Bmf|$. Hence, (1) of Proposition \ref{alphalowerbound} implies
\begin{eqnarray}\label{lowerbound4}
\sum_{j\in\Bmf}l_\Mpc(\alpha_j)&\geq&\frac{1}{2}\sum_{j\in\Bmf'}(l_\Mpc(\alpha_j)+l_\Mpc(\alpha_{j+1}))\nonumber\\
&\geq&12\sum_{j\in\Bmf'}K(\Mpc)\\
&\geq &5|\Bmf| K(\Mpc).\nonumber
\end{eqnarray}
Combining Equations (\ref{lowerbound3}) and (\ref{lowerbound4}) gives  
\begin{equation}\label{lowerbound5}
\sum_{j\in\Amf}l_\Mpc(\hat{\alpha}_j)+\sum_{j\in\Bmf}l_\Mpc(\alpha_j)\geq 5|\Amf|K(\Mpc)+5|\Bmf| K(\Mpc)= 5mK(\Mpc).
\end{equation}
The theorem then follows from the inequalities (\ref{lowerbound1}), (\ref{lowerbound2}) and (\ref{lowerbound5}).
\end{proof}

The next lemma is the main computation in this subsection. This lemma, together with Theorem \ref{mainestimate}, implies (1) of Theorem \ref{mainthm1}.

\begin{lem}\label{Kgoestoinfty}
For any Goldman sequence $\{\Mpc_j\}_{j=1}^\infty$ in $\Cmf(M)$, we have
\[\lim_{j\to\infty}K(\Mpc_j)=\infty.\]
\end{lem}

\begin{proof}
To prove the lemma, it is sufficient to show that for any Goldman sequence $\Ppc_j\in\Cmf(P)$, we have that 
\[\lim_{j\to\infty}X_j\cdot Y_j=\infty,\]
where $X_j$ is the minimum of the following six numbers 
\begin{eqnarray*}
Cr_{f,d}(\Ppc_j)Cr_{c,f}(\Ppc_j), Cr_{f,d}(\Ppc_j)Cr_{b,f}(\Ppc_j),\\
Cr_{d,e}(\Ppc_j)Cr_{a,d}(\Ppc_j), Cr_{d,e}(\Ppc_j)Cr_{c,d}(\Ppc_j),\\
Cr_{e,f}(\Ppc_j)Cr_{b,e}(\Ppc_j), Cr_{e,f}(\Ppc_j)Cr_{a,e}(\Ppc_j),
\end{eqnarray*}
and $Y_j$ is the minimum of the following six numbers
\begin{eqnarray*}
Cr_{d,f}(\Ppc_j)Cr_{b,d}(\Ppc_j),Cr_{d,f}(\Ppc_j)Cr_{a,d}(\Ppc_j),\\
Cr_{e,d}(\Ppc_j)Cr_{c,e}(\Ppc_j),Cr_{e,d}(\Ppc_j)Cr_{b,e}(\Ppc_j),\\
Cr_{f,e}(\Ppc_j)Cr_{a,f}(\Ppc_j),Cr_{f,e}(\Ppc_j)Cr_{c,f}(\Ppc_j).
\end{eqnarray*}

We can compute (see Appendix \ref{crossratiocomputationsandformulas}) each $Cr_{x,y}(\Ppc_j)$ explicitly in terms of 
\[\rho_1:=Cr_{a,d}(\Ppc_j),\ \rho_2:=Cr_{b,e}(\Ppc_j),\ \rho_3:=Cr_{c,f}(\Ppc_j)\]
and the internal parameter $r$ for $\Ppc_j$ (see Section \ref{areparameterization}). This gives us
\begin{eqnarray*}
X_j&=&\min(\{\frac{\rho_{i-1}(r+\rho_i\rho_{i+1})}{\rho_{i+1}(\rho_i-1)}:i=1,2,3\}\cup\{\frac{(\rho_i+r)(r+\rho_i\rho_{i+1})}{\rho_i\rho_{i+1}(\rho_i-1)}:i=1,2,3\})\\
Y_j&=&\min(\{\frac{\rho_{k-1}(\rho_k+r)(r+\rho_k\rho_{k+1})}{r^2(\rho_{k-1}-1)}:k=1,2,3\}\cup\{\frac{\rho_k\rho_{k-1}(\rho_k+r)}{r(\rho_{k-1}-1)}:k=1,2,3\}),
\end{eqnarray*}
where arithmetic in the subscripts is done in $\Zbbb_3$.
Thus, 
\[X_j\cdot Y_j=\min\{\frac{\rho_{k-1}(\rho_k+r)(r+\rho_i\rho_{i+1})}{r\rho_{i+1}(\rho_{k-1}-1)(\rho_i-1)}\cdot Z_{i,k}:i,k=1,2,3\}\]
where 
\begin{eqnarray*}
Z_{i,k}&:=&\min\{\frac{\rho_{i-1}(r+\rho_k\rho_{k+1})}{r},\frac{(\rho_i+r)(r+\rho_k\rho_{k+1})}{r\rho_i},\rho_{i-1}\rho_k,\frac{\rho_k(\rho_i+r)}{\rho_i}\}\\
&\geq&\min\{\rho_{i-1},\rho_k,\frac{\rho_k\rho_{k+1}}{\rho_i}\}.
\end{eqnarray*}

Let
\begin{eqnarray}\label{bignumbers}
n_1&:=&\frac{\rho_{k-1}\rho_{i-1}(\rho_k+r)(r+\rho_i\rho_{i+1})}{r\rho_{i+1}(\rho_{k-1}-1)(\rho_i-1)},\nonumber\\
n_2&:=&\frac{\rho_{k-1}\rho_k(\rho_k+r)(r+\rho_i\rho_{i+1})}{r\rho_{i+1}(\rho_{k-1}-1)(\rho_i-1)},\\
n_3&:=&\frac{\rho_{k-1}\rho_k\rho_{k+1}(\rho_k+r)(r+\rho_i\rho_{i+1})}{r\rho_i\rho_{i+1}(\rho_{k-1}-1)(\rho_i-1)}.\nonumber
\end{eqnarray}
It is now sufficient to show that for all $i,k=1,2,3$, the limits as $j$ approaches infinity of $n_1$, $n_2$ and $n_3$ are infinity. (Note that $r$ and the $\rho_i$ depend on $j$.)

Let 
\[((\lambda_A,\tau_A),(\lambda_B,\tau_B),(\lambda_C,\tau_C),s,r)\]
be the Goldman parameters for $\Ppc_j$. Since $\{\Ppc_j\}_{j=1}^\infty$ is a Goldman sequence, every subsequence of $\{\Ppc_j\}_{j=1}^\infty$ has a further subsequence, also denoted $\{\Ppc_j\}_{j=1}^\infty$, such that either $\displaystyle\lim_{j\to\infty}s=\infty$, $\displaystyle\lim_{j\to\infty}s=0$, $\displaystyle\lim_{j\to \infty}r=\infty$ or $\displaystyle\lim_{j\to \infty}r=0$. We need to show that along each of these subsequences, $n_1$, $n_2$ and $n_3$ grow to infinity.

Suppose first that $\displaystyle\lim_{j\to\infty}s=\infty$ or $\displaystyle\lim_{j\to\infty}s=0$. The condition that the lengths of the boundaries curves of $\Ppc_j$ are bounded away from $0$ and $\infty$ imply that $\displaystyle\lim_{j\to\infty}\rho_i=\infty$ or $\displaystyle\lim_{j\to\infty}\rho_i=1$ respectively for $i=1,2,3$. Since
\[n_1\geq \frac{\rho_{k-1}\rho_{i-1}\rho_i\rho_{i+1}r}{r\rho_{i+1}(\rho_{k-1}-1)(\rho_i-1)}\geq\max\{\rho_{i-1},\frac{1}{\rho_i-1}\},\]
it is clear that if $\displaystyle\lim_{j\to\infty}s=\infty$ or $\displaystyle\lim_{j\to\infty}s=0$, then $\displaystyle\lim_{j\to\infty} n_1=\infty$.

In the case when $\displaystyle\lim_{j\to\infty}s$ is neither infinite nor zero, the condition that the lengths of the boundaries curves of $\Ppc_j$ are bounded away from $0$ and $\infty$ implies that $\rho_1$, $\rho_2$ and $\rho_3$ are bounded away from $1$ and $\infty$. Thus, if $\displaystyle\lim_{j\to \infty}r=\infty$, the inequality
\[n_1\geq \frac{\rho_{k-1}\rho_{i-1}r^2}{r\rho_{i+1}(\rho_{k-1}-1)(\rho_i-1)}\geq\frac{\rho_{i-1}r}{\rho_{i+1}(\rho_i-1)}\]
implies that $\displaystyle\lim_{j\to\infty}n_1=\infty$. On the other hand, if $\displaystyle\lim_{j\to \infty}r=0$, the inequality 
\[n_1\geq \frac{\rho_{k-1}\rho_{i-1}\rho_k\rho_i\rho_{i+1}}{r\rho_{i+1}(\rho_{k-1}-1)(\rho_i-1)}\geq\frac{1}{r}\]
implies that $\displaystyle\lim_{j\to\infty}n_1=\infty$.

We have thus shown that for any Goldman sequence $\{\Ppc_j\}_{j=1}^\infty$, $\displaystyle\lim_{j\to\infty}n_1=\infty$. Using similar arguments, we can likewise show that $\displaystyle\lim_{j\to\infty}n_2=\infty$ and $\displaystyle\lim_{j\to\infty}n_3=\infty$.
\end{proof}

\subsection{Bounding the entropy and proof of (2) of Theorem \ref{mainthm1}}\label{boundingentropy}
The work in Section \ref{lowerboundlengths} gives us a lower bound on the length of all typical closed geodesics $\eta\in\Tpc_\Mpc$, which depends only on the data $\psi(\eta)$ and grows to infinity along Goldman sequences. Moreover, by Section \ref{combinatorialdescriptionsclosedgeodesics}, we also have an upper bound on the number of typical closed geodesics with the same $\psi(\eta)$.  We will now show that these bounds are strong enough to show that the topological entropy converges to zero along Goldman sequences, thereby proving (2) of Theorem \ref{mainthm1}. 

Suppose that $\Mpc\in\Cmf(M)$ is such that $K(\Mpc)>l_\Mpc(\gamma)$ for all $\gamma\in\Pmc$. By Theorem \ref{mainestimate}, we know that if $l_\Mpc(\eta)\leq T$, then the number of crossing points of $\eta$ is at most $\displaystyle\bigg\lfloor\frac{T}{K(\Mpc)}\bigg\rfloor$. Applying Proposition \ref{phitopsi}, we have that for $T\gg K(\Mpc)$, 
\begin{equation}\label{inequal1}
|\{\eta\in\Tpc_\Mpc:l_\Mpc(\eta)\leq T\}|\leq 18^{\lfloor\frac{T}{K(\Mpc)}\rfloor}\cdot|\{\psi(\eta)\in\Gpc_\Mpc:B(\psi(\eta))\leq T\}|.
\end{equation} 
(Recall that $\displaystyle B(\psi(\eta)):=m\cdot K(\Mpc)+\sum_{\beta\in\Ymf_\eta\cup\Zmf_\eta}\#(\beta)\cdot L(\Mpc)$.)

If $\eta\in\Tpc_\Mpc$ has $m$ crossing points, then Lemma \ref{crossingtriple} implies that there are at most $(24g-24+12n)^m$ possibilities for what the cyclic sequence of crossing triples of $\eta$ can be. Also, if $B(\psi(\eta))\leq T$, and $\eta$ has $m$ crossing points, then Theorem \ref{mainestimate} implies that 
\[\sum_{\beta\in\Ymf_\eta\cup\Zmf_\eta}\#(\beta)\leq \frac{T-m\cdot K(\Mpc)}{L(\Mpc)}.\] 
Thus, we have the inequality
\begin{align}\label{inequal2}
|\{\psi(\eta)\in\Gpc_\Mpc:B(\psi(\eta))\leq T\}|&\leq\sum_{m=1}^{\lfloor\frac{T}{K(\Mpc)}\rfloor}(24g-24+12n)^m\cdot {\lfloor\frac{T-mK(\Mpc)}{L(\Mpc)}\rfloor+m\choose m}\nonumber\\
&\leq(24g-24+12n)^{\lfloor \frac{T}{K(\Mpc)}\rfloor}\sum_{m=1}^{\lfloor \frac{T}{K(\Mpc)}\rfloor}{\lfloor\frac{T-mK(\Mpc)}{L(\Mpc)}\rfloor+m\choose m}.
\end{align}

Now, let $\Cpc_\Mpc$ be the set of oriented closed geodesics in $\Mpc$. If we assume that $T\gg K(\Mpc)$, then  (\ref{inequal1}) and (\ref{inequal2}) imply that
\begin{align*}
|\{\eta\in\Cpc_\Mpc:l_\Mpc(\eta)\leq T\}|&\leq|\{\eta\in\Tpc_\Mpc:l_\Mpc(\eta)\leq T\}|+2\cdot\Big\lfloor\frac{T}{L(\Mpc)}\Big\rfloor\cdot(3g-3+2n)\\
&\leq (432g-432+216n)^{\lfloor \frac{T}{K(\Mpc)}\rfloor}\sum_{m=1}^{\lfloor \frac{T}{K(\Mpc)}\rfloor}{\lfloor\frac{T-mK(\Mpc)}{L(\Mpc)}\rfloor+m\choose m}\\
&\hspace{0.5cm}+2\cdot\Big\lfloor\frac{T}{L(\Mpc)}\Big\rfloor\cdot(3g-3+2n)\\
&\leq 2\cdot(432g-432+216n)^{\lfloor \frac{T}{K(\Mpc)}\rfloor}\sum_{m=1}^{\lfloor \frac{T}{K(\Mpc)}\rfloor}{\lfloor\frac{T-mK(\Mpc)}{L(\Mpc)}\rfloor+m\choose m}\\
&\leq 2\cdot(432g-432+216n)^{\lfloor \frac{T}{K(\Mpc)}\rfloor}\cdot\Big\lfloor \frac{T}{K(\Mpc)}\Big\rfloor\cdot{\lfloor\frac{T-QK(\Mpc)}{L(\Mpc)}\rfloor+Q\choose Q},
\end{align*}
where $\displaystyle Q\in\bigg\{1,\dots,\bigg\lfloor \frac{T}{K(\Mpc)}\bigg\rfloor\bigg\}$ is an integer such that $\displaystyle{\lfloor\frac{T-QK(\Mpc)}{L(\Mpc)}\rfloor+Q\choose Q}\geq {\lfloor\frac{T-mK(\Mpc)}{L(\Mpc)}\rfloor+m\choose m}$ for all $\displaystyle m\in\bigg\{1,\dots,\bigg\lfloor \frac{T}{K(\Mpc)}\bigg\rfloor\bigg\}$. Thus,
\begin{align*}
\frac{\log|\{\eta\in\Cpc_\Mpc:l_\Mpc(\eta)\leq T\}|}{T}&\leq\frac{\log(2)+\lfloor\frac{T}{K(\Mpc)}\rfloor\log(432g-432+216n)+\log(\lfloor \frac{T}{K(\Mpc)}\rfloor)}{T}\\
&\hspace{0.5cm}+\frac{1}{T}\log{\lfloor\frac{T-QK(\Mpc)}{L(\Mpc)}\rfloor+Q\choose Q},
\end{align*}
so by taking the limit supremum as $T$ approaches infinity on both sides, we get that 
\begin{equation*}
h_{top}(\Mpc)\leq\frac{\log(432g-432+216n)}{K(\Mpc)}+\limsup_{T\to\infty}\frac{1}{T}\log{\lfloor\frac{T-QK(\Mpc)}{L(\Mpc)}\rfloor+Q\choose Q}
\end{equation*}
for any $\Mpc\in\Cmf(M)$ such that $K(\Mpc)>l_\Mpc(\gamma)$ for all $\gamma\in\Pmc$. 

Consider a Goldman sequence $\{\Mpc_i\}_{i=1}^\infty$ in $\Cmf(M)$.  By Lemma \ref{Kgoestoinfty}, which tells us $\displaystyle\lim_{i\to\infty}K(\Mpc_i)=\infty$, we have that 
\[\limsup_{i\to\infty}h_{top}(\Mpc_i)\leq\limsup_{i\to\infty}\limsup_{T\to\infty}\frac{1}{T}\log{\lfloor\frac{T-QK(\Mpc_i)}{L(\Mpc_i)}\rfloor+Q\choose Q}.\]

Recall that there are constants $0<C_1<C_2<\infty$ such that $C_1<L(\Mpc_i)<C_2$. Hence, to prove (2) of Theorem \ref{mainthm1}, it is thus sufficient to show the following technical proposition, which we prove in Appendix \ref{proof of proposition technical}.

\begin{prop} \label{technical}
For any constant $L$ and any increasing sequences $\{K_i\}_{i=1}^\infty$ and $\{T_j\}_{j=1}^\infty$ such that $\displaystyle\lim_{i\to\infty}K_i=\infty=\displaystyle\lim_{j\to\infty}T_j$, we have that 
\begin{equation*}
\lim_{i\to\infty}\limsup_{j\to\infty}\frac{1}{T_j}\log{\lfloor\frac{T_j-Q_{i,j}K_i}{L}\rfloor+Q_{i,j}\choose Q_{i,j}}=0,
\end{equation*}
where $Q_{i,j}$ is a number in $\displaystyle\bigg\{1,\dots,\bigg\lfloor\frac{T_j}{K_i}\bigg\rfloor\bigg\}$ such that $\displaystyle{\lfloor\frac{T_j-Q_{i,j}K_i}{L}\rfloor+Q_{i,j}\choose Q_{i,j}}\geq{\lfloor\frac{T_j-mK_i}{L}\rfloor+m\choose m}$ for all $\displaystyle m\in\bigg\{1,\dots,\bigg\lfloor\frac{T_j}{K_i}\bigg\rfloor\bigg\}$.
\end{prop}

\appendix

\section{Cross ratio computations and formulas}\label{crossratiocomputationsandformulas}

Here, we will list the formulas for some of the thirty cross ratios $Cr_{x,y}=Cr_{x,y}(\Ppc)$ as defined at the start of Section \ref{areparameterization}, namely those we used in the proof of our result. The formulas will be listed in both the Goldman parameters $(R,s,t)$ defined in Section \ref{goldmanparameters} and the new parameters $(R,s,r)$ defined in Section \ref{areparameterization}. Also, we will demonstrate the computation to obtain one of these, namely $Cr_{d,e}$. The rest of the cross ratios are computed using the same algorithm. To simplify the formulas, we will write $\rho_i^R(S)$ as $\rho_i$ for $i=1,2,3$.

We compute $Cr_{d,e}$ using Definition \ref{cross ratio def}. Since the cross ratio is invariant under the action of $SL(3,\Rbbb)$ we can choose a normalization so that 
\[a=\left[\begin{array}{c} 1\\ 0\\ 0 \end{array} \right], 
b=\left[ \begin{array}{c} 0\\ 1\\ 0 \end{array} \right], 
c=\left[ \begin{array}{c} 0\\ 0\\ 1 \end{array} \right], 
f=\left[ \begin{array}{c} 2\\ 2\\ -1\end{array} \right].\]
Then in the Goldman parameters, we have 
\[d=\left[ \begin{array}{c} -1\\ \frac{\rho_3}{t}\\ \frac{\rho_2}{2} \end{array} \right],
e=\left[ \begin{array}{c} t\\ -1\\ \frac{\rho_1}{2} \end{array} \right].\]
(This was computed in Section 4 of Goldman \cite{Go1}). By the definition of the cross ratio, we can then compute
\begin{eqnarray*}Cr_{d,e}&=&\frac{\left[ \begin{array}{c} -1\\ \frac{\rho_3}{t}\\ \frac{\rho_2}{2} \end{array} \right]\wedge\left[ \begin{array}{c} 0\\ 1\\ 0 \end{array} \right]\wedge\left[\begin{array}{c} 1\\ 0\\ 0 \end{array} \right]}{\left[ \begin{array}{c} -1\\ \frac{\rho_3}{t}\\ \frac{\rho_2}{2} \end{array} \right]\wedge\left[ \begin{array}{c} 0\\ 1\\ 0 \end{array} \right]\wedge\left[ \begin{array}{c} 2\\ 2\\ -1\end{array} \right]}\cdot\frac{\left[ \begin{array}{c} -1\\ \frac{\rho_3}{t}\\ \frac{\rho_2}{2} \end{array} \right]\wedge\left[ \begin{array}{c} 0\\ 0\\ 1 \end{array} \right]\wedge\left[ \begin{array}{c} 2\\ 2\\ -1\end{array} \right]}{\left[ \begin{array}{c} -1\\ \frac{\rho_3}{t}\\ \frac{\rho_2}{2} \end{array} \right]\wedge\left[ \begin{array}{c} 0\\ 0\\ 1 \end{array} \right]\wedge\left[\begin{array}{c} 1\\ 0\\ 0 \end{array} \right]}\\
&=&\frac{\det\left[ \begin{array}{ccc} -1&0&1\\ \frac{\rho_3}{t}&1&0\\ \frac{\rho_2}{2}&0&0 \end{array} \right]}{\det\left[ \begin{array}{ccc} -1&0&2\\ \frac{\rho_3}{t}&1&2\\ \frac{\rho_2}{2}&0&-1\end{array} \right]}\cdot\frac{\det\left[ \begin{array}{ccc} -1&0&2\\ \frac{\rho_3}{t}&0&2\\ \frac{\rho_2}{2}&1&-1 \end{array} \right]}{\det\left[ \begin{array}{ccc} -1&0&1\\ \frac{\rho_3}{t}&0&0\\ \frac{\rho_2}{2}&1&0 \end{array} \right]}\\
&=&\frac{-\frac{\rho_2}{2}}{1-\rho_2}\cdot\frac{2+\frac{2\rho_3}{t}}{\frac{\rho_3}{t}}\\
&=&\frac{\rho_2(t+\rho_3)}{\rho_3(\rho_2-1)}.
\end{eqnarray*}

Now, we will give the list of formulas for the cross ratios we used in the paper.
\begin{equation*}
\begin{array}{rllcrll}
Cr_{a,f}&=1+\frac{\rho_3\rho_1}{t\rho_2}&=1+\frac{\rho_3\rho_1}{r}&&Cr_{d,e}&=\frac{\rho_2(t+\rho_3)}{\rho_3(\rho_2-1)}&=\frac{r+\rho_2\rho_3}{\rho_3(\rho_2-1)}\\
Cr_{a,d}&=\rho_1&=\rho_1&&Cr_{d,f}&=\frac{\rho_3(\rho_2t+\rho_1)}{t\rho_2(\rho_3-1)}&=\frac{\rho_3\rho_1+r\rho_3}{r(\rho_3-1)}\\
Cr_{a,e}&=1+\frac{t\rho_2}{\rho_3}&=1+\frac{r}{\rho_3}&&Cr_{c,e}&=1+\frac{\rho_3}{t}&=1+\frac{\rho_2\rho_3}{r}\\
Cr_{f,d}&=\frac{t+\rho_1}{-1+\rho_1}&=\frac{r+\rho_1\rho_2}{\rho_2(\rho_1-1)}&&Cr_{c,f}&=\rho_3&=\rho_3\\
Cr_{f,e}&=\frac{\rho_3+t\rho_2}{t(\rho_2-1)}&=\frac{\rho_2\rho_3+r\rho_2}{r(\rho_2-1)}&&Cr_{c,d}&=1+t&=1+\frac{r}{\rho_2}\\
Cr_{b,d}&=1+\frac{\rho_1}{t}&=1+\frac{\rho_1\rho_2}{r}&&Cr_{e,f}&=\frac{\rho_2t+\rho_3\rho_1}{\rho_1(\rho_3-1)}&=\frac{r+\rho_3\rho_1}{\rho_1(\rho_3-1)}\\
Cr_{b,e}&=\rho_2&=\rho_2&&Cr_{e,d}&=\frac{\rho_1(1+t)}{t(\rho_1-1)}&=\frac{\rho_1(\rho_2+r)}{r(\rho_1-1)}\\
Cr_{b,f}&=1+\frac{t\rho_2}{\rho_1}&=1+\frac{r}{\rho_1}&&&&
\end{array}
\end{equation*}

\section{Proof of Proposition \ref{technical}}\label{proof of proposition technical}
In this appendix, we will prove Proposition \ref{technical}, which we restate here for the reader's convenience.

\begin{prop}[(Proposition \ref{technical})]\label{technical2}
For any constant $L$ and any increasing sequences $\{K_i\}_{i=1}^\infty$ and $\{T_j\}_{j=1}^\infty$ such that $\displaystyle\lim_{i\to\infty}K_i=\infty=\displaystyle\lim_{j\to\infty}T_j$, we have that 
\begin{equation*}
\lim_{i\to\infty}\limsup_{j\to\infty}\frac{1}{T_j}\log{\lfloor\frac{T_j-Q_{i,j}K_i}{L}\rfloor+Q_{i,j}\choose Q_{i,j}}=0,
\end{equation*}
where $Q_{i,j}$ is a number in $\displaystyle\bigg\{1,\dots,\bigg\lfloor\frac{T_j}{K_i}\bigg\rfloor\bigg\}$ such that $\displaystyle{\lfloor\frac{T_j-Q_{i,j}K_i}{L}\rfloor+Q_{i,j}\choose Q_{i,j}}\geq{\lfloor\frac{T_j-mK_i}{L}\rfloor+m\choose m}$ for all $\displaystyle m\in\bigg\{1,\dots,\bigg\lfloor\frac{T_j}{K_i}\bigg\rfloor\bigg\}$.
\end{prop} 

First, we will fix $K_i$ to be $K\gg L$ and compute 
\begin{equation}\label{limitbound1}
\lim_{j\to\infty}\frac{1}{T_j}\log{\lfloor\frac{T_j-Q_jK}{L}\rfloor+Q_j\choose Q_j}, 
\end{equation}
where $Q_j$ is a number in $\displaystyle\bigg\{1,\dots,\bigg\lfloor\frac{T_j}{K}\bigg\rfloor\bigg\}$ such that $\displaystyle{\lfloor\frac{T_j-Q_jK}{L}\rfloor+Q_j\choose Q_j}\geq{\lfloor\frac{T_j-mK}{L}\rfloor+m\choose m}$ for all $m\in\displaystyle\bigg\{1,\dots,\bigg\lfloor\frac{T_j}{K}\bigg\rfloor\bigg\}$.
The main tool to compute (\ref{limitbound1}) is the asymptotic equality commonly known as Stirling's Formula, which we state here.

\begin{thm}[(Stirling's Formula)]
$\displaystyle n!\sim\bigg(\frac{n}{e}\bigg)^n\sqrt{2\pi n}$, i.e. $\displaystyle\lim_{n\to\infty}\frac{n!}{(\frac{n}{e})^n\sqrt{2\pi n}}=1$.
\end{thm}

For fixed $K$ and $L$, we will at times denote $\displaystyle\bigg\lfloor\frac{T_j-Q_jK}{L}\bigg\rfloor$ by $F_j$ to simplify notation. To use Stirling's formula, we need to know how $F_j$ and $Q_j$ vary with $j$. Hence the following lemma.

\begin{lem}\label{stirlingconditions}
Let $K,L$ be fixed, with $K\gg L$. Then the following hold:
\begin{enumerate}[(1)]
\item $\displaystyle\lim_{j\to\infty}F_j=\infty$.
\item $\displaystyle\lim_{j\to\infty}Q_j=\infty$.
\item $\displaystyle0\leq\liminf_{j\to\infty}\frac{Q_j}{F_j}\leq\limsup_{j\to\infty}\frac{Q_j}{F_j}\leq1$. 
\item There exists $\alpha>0$ such that $\displaystyle\alpha\leq\liminf_{j\to\infty}\frac{Q_j}{T_j}\leq\limsup_{j\to\infty}\frac{Q_j}{T_j}\leq\frac{1}{L+K}$. 
\end{enumerate}
\end{lem}

\begin{proof}
Observe that for sufficiently large $j$ (so that $T_j\gg K$),
\begin{align*}
1&\leq{\lfloor\frac{T_j-Q_jK}{L}\rfloor+Q_j\choose Q_j}\Bigg/{\lfloor\frac{T_j-(Q_j+1)K}{L}\rfloor+Q_j+1\choose Q_j+1}\\
&=\frac{(\lfloor\frac{T_j-Q_jK}{L}\rfloor+Q_j)(\lfloor\frac{T_j-Q_jK}{L}\rfloor+Q_j-1)\dots(\lfloor\frac{T_j-(Q_j+1)K}{L}\rfloor+Q_j+2)(Q_j+1)}{(\lfloor\frac{T_j-Q_jK}{L}\rfloor)(\lfloor\frac{T_j-Q_jK}{L}\rfloor-1)\cdots(\lfloor\frac{T_j-(Q_j+1)K}{L}\rfloor+1)}
\end{align*}
which implies 
\begin{align}\label{Qbound1}
\frac{Q_j+1}{F_j}=\frac{Q_j+1}{\lfloor\frac{T_j-Q_jK}{L}\rfloor}&\geq\frac{(\lfloor\frac{T_j-Q_jK}{L}\rfloor-1)\cdots(\lfloor\frac{T_j-(Q_j+1)K}{L}\rfloor+1)}{(\lfloor\frac{T_j-Q_jK}{L}\rfloor+Q_j)\dots(\lfloor\frac{T_j-(Q_j+1)K}{L}\rfloor+Q_j+2)}.
\end{align}

Similarly, for sufficiently large $j$,
\begin{align*}
1&\geq{\lfloor\frac{T_j-(Q_j-1)K}{L}\rfloor+Q_j-1\choose Q_j-1}\Bigg/{\lfloor\frac{T_j-Q_jK}{L}\rfloor+Q_j\choose Q_j}&\\
&=\frac{(\lfloor\frac{T_j-(Q_j-1)K}{L}\rfloor+Q_j-1)(\lfloor\frac{T_j-(Q_j-1)K}{L}\rfloor+Q_j-2)\dots(\lfloor\frac{T_j-Q_jK}{L}\rfloor+Q_j+1)(Q_j)}{(\lfloor\frac{T_j-(Q_j-1)K}{L}\rfloor)(\lfloor\frac{T_j-(Q_j-1)K}{L}\rfloor-1)\cdots(\lfloor\frac{T_j-Q_jK}{L}\rfloor+1)}&\nonumber
\end{align*}
which implies 
\begin{align}\label{Qbound2}
\frac{Q_j}{F_j+1}=\frac{Q_j}{\lfloor\frac{T_j-Q_jK}{L}\rfloor+1}&\leq\frac{(\lfloor\frac{T_j-(Q_j-1)K}{L}\rfloor)\cdots(\lfloor\frac{T_j-Q_jK}{L}\rfloor+2)}{(\lfloor\frac{T_j-(Q_j-1)K}{L}\rfloor+Q_j-1)\dots(\lfloor\frac{T_j-Q_jK}{L}\rfloor+Q_j+1)}\leq 1.
\end{align}

Proof of (1). Suppose for contradiction that $\displaystyle\liminf_{j\to\infty}F_j<\infty$. This implies that $\displaystyle\limsup_{j\to\infty}Q_j=\infty$, so we have
\begin{equation*}
\limsup_{j\to\infty}\frac{Q_j}{F_j+1}=\infty.
\end{equation*}
However, that contradicts (\ref{Qbound2}).  

Proof of (2). Suppose for contradiction that $\displaystyle\liminf_{j\to\infty}Q_j<\infty$. By choosing a subsequence, we can assume that $\displaystyle\lim_{j\to\infty}Q_j<\infty$, then $\displaystyle\lim_{j\to\infty} F_j=\infty$, so we have
\begin{equation*}
\lim_{j\to\infty}\frac{F_j}{Q_j+1}=\infty.
\end{equation*}
However, if $\displaystyle\lim_{j\to\infty}Q_j<\infty$ and $\displaystyle\lim_{j\to\infty} F_j=\infty$, then the right hand side of the inequality (\ref{Qbound1}) converges to $1$ as $j\to\infty$. This then implies that 
\begin{equation*}
\lim_{j\to\infty}\frac{F_j}{Q_j+1}\leq 1
\end{equation*}
which is a contradiction. 

Proof of (3). This follows immediately from (1), (2) and the inequality (\ref{Qbound2}).

Proof of (4). By (\ref{Qbound2}), we know that $\displaystyle\frac{Q_j}{F_j+1}\leq 1$, so $\displaystyle\frac{Q_j}{T_j}\leq\frac{1}{L+K}\bigg(1+\frac{L}{T_j}\bigg)$. Since $\displaystyle\lim_{j\to\infty}T_j=\infty$, this proves $\displaystyle\limsup_{j\to\infty}\frac{Q_j}{T_j}\leq\frac{1}{L+K}$. It is clear that $\displaystyle\liminf_{j\to\infty}\frac{Q_j}{T_j}\geq 0$, so suppose for contradiction that $\displaystyle\liminf_{j\to\infty}\frac{Q_j}{T_j}=0$. By taking a subsequence, we can assume that $\displaystyle\lim_{j\to\infty}\frac{Q_j}{T_j}=0$, which implies that $\displaystyle\lim_{j\to\infty}\frac{Q_j}{F_j}=0$.

Since we know (1) and (2) hold, (\ref{Qbound1}) and (\ref{Qbound2}) imply that
\begin{align*}
&\hspace{.5cm}\limsup_{j\to\infty}\frac{(\lfloor\frac{T_j-Q_jK}{L}\rfloor-1)\cdots(\lfloor\frac{T_j-(Q_j+1)K}{L}\rfloor+1)}{(\lfloor\frac{T_j-Q_jK}{L}\rfloor+Q_j)(\lfloor\frac{T_j-Q_jK}{L}\rfloor+Q_j-1)\dots(\lfloor\frac{T_j-(Q_j+1)K}{L}\rfloor+Q_j+2)}\\
&\leq\limsup_{j\to\infty}\frac{Q_j+1}{F_j}\\
&=\limsup_{j\to\infty}\frac{Q_j}{F_j}\\
&=\limsup_{j\to\infty}\frac{Q_j}{F_j+1}\\
&\leq\limsup_{j\to\infty}\frac{(\lfloor\frac{T_j-(Q_j-1)K}{L}\rfloor)(\lfloor\frac{T_j-(Q_j-1)K}{L}\rfloor-1)\cdots(\lfloor\frac{T_j-Q_jK}{L}\rfloor+2)}{(\lfloor\frac{T_j-(Q_j-1)K}{L}\rfloor+Q_j-1)(\lfloor\frac{T_j-(Q_j-1)K}{L}\rfloor+Q_j-2)\dots(\lfloor\frac{T_j-Q_jK}{L}\rfloor+Q_j+1)}\\
&=\limsup_{j\to\infty}\frac{(\lfloor\frac{T_j-Q_jK}{L}\rfloor-1)\cdots(\lfloor\frac{T_j-(Q_j+1)K}{L}\rfloor+1)}{(\lfloor\frac{T_j-Q_jK}{L}\rfloor+Q_j)(\lfloor\frac{T_j-Q_jK}{L}\rfloor+Q_j-1)\dots(\lfloor\frac{T_j-(Q_j+1)K}{L}\rfloor+Q_j+2)}
\end{align*}
so in particular, 
\begin{eqnarray*}
\lim_{j\to\infty}\frac{Q_j}{F_j}&=&\limsup_{j\to\infty}\frac{(\lfloor\frac{T_j-Q_jK}{L}\rfloor-1)\cdots(\lfloor\frac{T_j-(Q_j+1)K}{L}\rfloor+1)}{(\lfloor\frac{T_j-Q_jK}{L}\rfloor+Q_j)\dots(\lfloor\frac{T_j-(Q_j+1)K}{L}\rfloor+Q_j+2)}.
\end{eqnarray*}

However, this is not possible because $\displaystyle\lim_{j\to\infty}\frac{Q_j}{\lfloor\frac{T_j-Q_jK}{L}\rfloor}=\lim_{j\to\infty}\frac{Q_j}{F_j}=0$ implies that the left hand side of the above equation is $0$ and the right hand side is $1$.
\end{proof}

With Lemma \ref{stirlingconditions}, we can now explicitly compute (\ref{limitbound1}).

\begin{prop}\label{computation0}
Let $K,L$ be fixed, with $K\gg L$. Let $\{T_j\}_{j=1}^\infty$ be a sequence of positive numbers such that $\displaystyle\lim_{j\to\infty}T_j=\infty$ and $\displaystyle H:=\lim_{j\to\infty}\frac{Q_j}{T_j}$ exists. Then
\[\lim_{j\to\infty}\frac{1}{T_j}\log{F_j+Q_j\choose Q_j}=H\log\bigg(1-\frac{K}{L}+\frac{1}{HL}\bigg)+\frac{1-HK}{L}\log\bigg(1+\frac{HL}{1-HK}\bigg).\]
\end{prop}

\begin{proof}
By (3) and (4) of  Lemma \ref{stirlingconditions}, we know that $0\leq\displaystyle\limsup_{j\to\infty}\frac{Q_j}{\lfloor\frac{T_j-Q_jK}{L}\rfloor}\leq1$ and $\displaystyle 0<H\leq\frac{1}{L+K}$. Also, (1) and (2) of Lemma \ref{stirlingconditions} allow us to apply Stirling's formula to obtain 
\begin{align*}
{F_j+Q_j\choose Q_j}&=\frac{(F_j+Q_j)!}{F_j!Q_j!}\\
&\sim\frac{((F_j+Q_j)\cdot\frac{1}{e})^{F_j+Q_j}\sqrt{2\pi(F_j+Q_j)}}{(F_j\cdot\frac{1}{e})^{F_j}\sqrt{2\pi F_j}(Q_j\cdot\frac{1}{e})^{Q_j}\sqrt{2\pi Q_j}}\\
&=\frac{1}{\sqrt{2\pi}}\cdot\frac{(F_j+Q_j)^{F_j+Q_j}}{Q_j^{Q_j}\cdot F_j^{F_j}}\cdot\sqrt{\frac{F_j+Q_j}{Q_j\cdot F_j}}, 
\end{align*}
i.e.
\begin{equation*}
\lim_{j\to\infty}\left({F_j+Q_j\choose Q_j}\cdot\sqrt{2\pi}\cdot\frac{Q_j^{Q_j}\cdot F_j^{F_j}}{(F_j+Q_j)^{F_j+Q_j}}\cdot\sqrt{\frac{Q_j\cdot F_j}{F_j+Q_j}}\right)=1
\end{equation*}

By taking the logarithm, we get
\begin{align}
&\lim_{j\to\infty}\left(\log{F_j+Q_j\choose Q_j}+\frac{1}{2}\log(2\pi)-\frac{1}{2}\log\bigg(\frac{F_j+Q_j}{F_j\cdot Q_j}\bigg)\right.\\
&\left.-Q_j\cdot\log\bigg(1+\frac{F_j}{Q_j}\bigg)-F_j\cdot\log\bigg(1+\frac{Q_j}{F_j}\bigg)\right)=0\nonumber.
\end{align}

To compute $\displaystyle\underset{j\to\infty}{\lim}\frac{1}{T_j}\log{F_j+Q_j\choose Q_j}$, it is now sufficient to compute 
\begin{equation*}
\lim_{j\to\infty}\frac{1}{2T_j}\left(\log\bigg(\frac{F_j+Q_j}{F_j\cdot Q_j}\bigg)\right)
\end{equation*}
and 
\begin{equation*}
\lim_{j\to\infty}\frac{1}{T_j}\left(Q_j\cdot\log\bigg(1+\frac{F_j}{Q_j}\bigg)+F_j\cdot\log\bigg(1+\frac{Q_j}{F_j}\bigg)\right).
\end{equation*}
The proposition thus follows from Lemma \ref{computation1} and \ref{computation2}.
\end{proof}

\begin{lem}\label{computation1}
\begin{equation*}
\lim_{j\to\infty}\frac{1}{2T_j}\Bigg(\log\bigg(\frac{F_j+Q_j}{F_j\cdot Q_j}\bigg)\Bigg)=0.
\end{equation*}
\end{lem}

\begin{proof}
Note that
\begin{equation*}
\log\bigg(\frac{F_j+Q_j}{F_j\cdot Q_j}\bigg)=\log\bigg(\frac{1}{Q_j}\bigg)+\log\bigg(1+\frac{Q_j}{F_j}\bigg).
\end{equation*}

Since $Q_j\geq 1$, we know that $\displaystyle\frac{\log(\frac{1}{Q_j})}{T_j}\leq 0$. Also, since $\displaystyle Q_j\leq\frac{T_j}{K}$, we have
\begin{equation*}
\liminf_{j\to\infty}\frac{\log(\frac{1}{Q_j})}{T_j}\geq\lim_{j\to\infty}\frac{-\log(\frac{T_j}{K})}{T_j}=0,
\end{equation*}
which implies
\begin{equation*}
\lim_{j\to\infty}\frac{\log(\frac{1}{Q_j})}{T_j}=0.
\end{equation*}

Also, it is clear that $\displaystyle\frac{1}{T_j}\log\bigg(1+\frac{Q_j}{F_j}\bigg)\geq 0$ for sufficiently large $j$, and (3) of Lemma \ref{stirlingconditions} implies 
\begin{equation*}
\limsup_{i\to\infty}\frac{1}{T_j}\log\bigg(1+\frac{Q_j}{F_j}\bigg)\leq\lim_{j\to\infty}\frac{\log(2)}{T_j}=0.
\end{equation*}
Thus,
\begin{equation*}
\lim_{j\to\infty}\frac{1}{T_j}\log\bigg(1+\frac{Q_j}{F_j}\bigg)=0.
\end{equation*}

Putting all these together, we get the equality in the lemma.
\end{proof}

\begin{lem}\label{computation2}
\begin{align*}
&\hspace{0.4cm}\lim_{j\to\infty}\frac{1}{T_j}\Bigg(Q_j\cdot\log\bigg(1+\frac{F_j}{Q_j}\bigg)+F_j\cdot\log\bigg(1+\frac{Q_j}{F_j}\bigg)\Bigg)\\
&=H\log\bigg(1-\frac{K}{L}+\frac{1}{HL}\bigg)+\frac{1-HK}{L}\log\bigg(1+\frac{HL}{1-HK}\bigg)
\end{align*}
\end{lem}

\begin{proof}
Observe that
\begin{align*}
&\hspace{0.4cm}\lim_{j\to\infty}\frac{1}{T_j}\left(Q_j\cdot\log\bigg(1+\frac{F_j}{Q_j}\bigg)+F_j\cdot\log\bigg(1+\frac{Q_j}{F_j}\bigg)\right)\\
&=\lim_{j\to\infty}\left(\frac{Q_j}{T_j}\cdot\log\bigg(1-\frac{K}{L}+\frac{T_j}{LQ_j}\bigg)+\frac{1}{L}\bigg(1-K\frac{Q_j}{T_j}\bigg)\cdot\log\bigg(1+\frac{L}{\frac{T_j}{Q_j}-K}\bigg)\right).
\end{align*}
By (4) of Lemma \ref{stirlingconditions}, we have $\displaystyle 0<H=\lim_{i\to\infty}\frac{Q_j}{T_j}<\frac{1}{K+L}$ so the required equality follows.
\end{proof}

We are now ready to prove Proposition \ref{technical2}.

\begin{proof}[of Proposition \ref{technical2}]
For each $i$ such that $K_i\gg L$, choose a subsequence $\{T_{j_k}\}_{k=1}^\infty$ of $\{T_j\}_{j=1}^\infty$ such that $\displaystyle H_i:=\lim_{k\to\infty}\frac{Q_{i,j_k}}{T_{j_k}}$ exists and 
\[\limsup_{j\to\infty}\frac{1}{T_j}\log{\lfloor\frac{T_j-Q_{i,j}K_i}{L}\rfloor+Q_{i,j}\choose Q_{i,j}}=\lim_{k\to\infty}\frac{1}{T_{j_k}}\log{\lfloor\frac{T_{j_k}-Q_{i,j_k}K_i}{L}\rfloor+Q_{i,j_k}\choose Q_{i,j_k}}\]
(we can do this by (4) of Lemma \ref{stirlingconditions}). Proposition \ref{computation0} then tells us that 
\begin{align}\label{computation3}
&\hspace{0.4cm}\limsup_{j\to\infty}\frac{1}{T_j}\log{\lfloor\frac{T_j-Q_{i,j}K_i}{L}\rfloor+Q_{i,j}\choose Q_{i,j}}\\
&=H_i\log\bigg(1+\frac{1-H_iK_i}{H_iL}\bigg)+\frac{1-H_iK_i}{L}\log\bigg(1+\frac{H_iL}{1-H_iK_i}\bigg).\nonumber 
\end{align}
Since $\displaystyle 0<Q_{i,j}\leq\bigg\lfloor\frac{T_j}{K_i}\bigg\rfloor\leq\frac{T_j}{K_i}$, we know that $\displaystyle 0<\frac{Q_{i,j}}{T_j}K_i\leq 1$. Thus, by choosing a subsequence of $\{K_i\}_{i=1}^\infty$, we can assume that 
\begin{equation*}
0\leq\lim_{i\to\infty}(H_iK_i)=:V\leq 1.
\end{equation*}
Since $\underset{i\to\infty}{\lim}K_i=\infty$, this implies that $\underset{i\to\infty}{\lim}H_i=0$. This means
\begin{equation*}
\limsup_{i\to\infty}H_i\log\bigg(1+\frac{1-H_iK_i}{H_iL}\bigg)\leq\lim_{i\to\infty} H_i\log\bigg(1+\frac{1}{H_iL}\bigg)=0
\end{equation*}
Also, $\displaystyle H_i\log\bigg(1+\frac{1-H_iK_i}{H_iL}\bigg)\geq 0$ for all $i$, so 
\begin{equation}\label{computation4}
\lim_{i\to\infty}H_i\log\bigg(1+\frac{1-H_iK_i}{H_iL}\bigg)=0
\end{equation}

Next, we show that $\displaystyle\lim_{i\to\infty}\frac{1-H_iK_i}{L}\log\bigg(1+\frac{H_iL}{1-H_iK_i}\bigg)=0$. Since $\underset{i\to\infty}{\lim}H_i=0$, this is clear in the case when $V<1$. In the case when $V=1$,\begin{equation*}
\frac{1-H_iK_i}{L}\log\bigg(1+\frac{H_iL}{1-H_iK_i}\bigg)\leq\frac{1-H_iK_i}{L}\log\bigg(1+\frac{L}{1-H_iK_i}\bigg)
\end{equation*}
for large enough $i$. By taking limit supremum,
\begin{equation*}
\limsup_{i\to\infty}\frac{1-H_iK_i}{L}\log\bigg(1+\frac{H_iL}{1-H_iK_i}\bigg)\leq\lim_{j\to\infty}\frac{1-H_iK_i}{L}\log\bigg(1+\frac{L}{1-H_iK_i}\bigg)=0.
\end{equation*}
Since $\displaystyle\frac{1-H_iK_i}{L}\log\bigg(1+\frac{H_iL}{1-H_iK_i}\bigg)\geq 0$ for all $i$, we have that
\begin{equation}\label{computation5}
\lim_{i\to\infty}\frac{1-H_iK_i}{L}\log\bigg(1+\frac{H_iL}{1-H_iK_i}\bigg)=0.
\end{equation}

Together, (\ref{computation3}), (\ref{computation4}) and (\ref{computation5}) imply that 
\begin{equation*}
\lim_{i\to\infty}\limsup_{j\to\infty}\frac{1}{T_j}\log{\lfloor\frac{T_j-Q_{i,j}K_i}{L}\rfloor+Q_{i,j}\choose Q_{i,j}}=0.
\end{equation*}
\end{proof}

\affiliationone{
  Tengren Zhang\\
2074 East Hall (EH 4828)\\
530 Church Street\\
Ann Arbor, MI 48109-1043\\
   USA
   \email{tengren@umich.edu}}
\end{document}